\documentclass[12pt]{article}
\usepackage{amsmath}
\usepackage{amsthm,amsmath,amsfonts,amssymb}
\usepackage{graphicx}
\usepackage{enumerate}
\usepackage{natbib}
\usepackage{url} 
\usepackage{physics}
\usepackage{siunitx}
\usepackage{enumitem}
\usepackage{arydshln}
\theoremstyle{plain}
\usepackage{scalerel} 
\usepackage{mathrsfs}
\usepackage{float}
\usepackage{bm}
\usepackage{setspace}
\usepackage{longtable}
\usepackage{xr} 
\usepackage[T1]{fontenc}    
\usepackage{textcomp}       

\pdfoutput=1
\newcommand{\blind}{0}

\theoremstyle{plain}

\newtheorem{theorem}{Theorem}[section]
\newtheorem{lemma}[theorem]{Lemma}
\theoremstyle{remark}

\newcommand{\beq}{\begin{equation}}
\newcommand{\eeq}{\end{equation}}
\newcommand{\bno}{\begin{Notes}}
\newcommand{\eno}{\end{Notes}\noindent}
\newcommand{\bdm}{\begin{displaymath}}
\newcommand{\edm}{\end{displaymath}}
\newcommand{\rbox}[2]{\raisebox{#1}{$\scriptstyle #2$}}
\newcommand{\rboxx}[2]{\ensuremath{\smash{\raisebox{#1}{$\scriptscriptstyle #2$}}}}
\newcommand{\smexp}[1]{^{ \rbox{-3pt}{#1}}}  
\newcommand{\lgamma}{\scalebox{1.3}{$\gamma$}}   
\newcommand{\tss}{\textsubscript}
\newcommand{\Tstrut}{\rule{0pt}{2.6ex}} 
\newcommand{\symquote}[1]{\textquotedbl{}#1\textquotedbl{}}

\begin{document}

\date{}
\if0\blind
{
  \title{\bf The Maximal Variance of Unilaterally Truncated Gaussian and Chi Distributions}
  \author{
Robert J. Petrella  \\
\hspace{-1cm}   Department of Chemistry and Chemical Biology, Harvard University, Cambridge, MA \\
     Harvard Medical School, Boston, MA\\}
  \maketitle
} \fi

\if1\blind
{
  \bigskip
  \bigskip
  \bigskip
  \begin{center}
    {\LARGE\bf Title}
\end{center}
  \medskip
} \fi

\bigskip
\begin{abstract}
This work explores the bounds of the variance of unilaterally truncated Gaussian
distributions (UTGDs) and scaled chi distributions (UTSCDs) with fixed
means. For any arbitrary Gaussian distribution function, $f(x;\mu,\sigma)$, with a fixed, finite mean
$M$ on the truncated domain $x \ge a$, where $a \in \mathbb{R}$, it is proven that the variance is bounded: specifically,
$\sup \mathrm{Var}(x)_{|x\ge a}= \sup \mathrm{Var}(x)_{|x\le a} =(M-a)^2$.
For a fixed cutoff, $a$, the variance
can be considered a function of only $M$, $a$, and the location
parameter $\mu$. Examples of such approximating functions, which can be used for model
calibration, are developed in addition to other, related calibration methods. For UTSCDs,
numerical evidence is presented indicating that for $n \in \mathbb{Z+}$ degrees of freedom, or dimensions, and
a fixed, finite mean, the variance, $\mathrm{Var}(R)$, over $R \in [a,\infty)$ reaches its
maximum value $M^2(\pi-2)/2$ at $a=0$, $n=1$. For a fixed cutoff value, there is a local
maximum in the variance as a function of $n$, and the number of dimensions resulting
in the maximal variance, $n_{\mathrm{vmx}}$, increases with cutoff value. However, for $n \in \mathbb{R}$,
as the cutoff approaches $0$, $n_{\mathrm{vmx}}$ approaches $-1$, while $\mathrm{Var}(R)$ appears to
grow without bound.
\end{abstract}

\noindent%

\newpage
\section{Introduction}
\label{sec:intro}
For Gaussian distributions
$f(x;\mu,\sigma)$ extending over the interval $x \in (-\infty,+\infty)$,
the mean and variance are independent. But over the interval $x \in [a,+\infty)$, or $x \in (-\infty,a]$,
where $a$ is finite, this is not the case.
Given a fixed, finite mean, $M(x;\mu,\sigma)$, over those unilaterally truncated intervals, the variance
is constrained.  This arises, for example, in the modeling of positive-valued
data with Gaussian functions, because the variance cannot be made arbitrarily large.
Over the interval $x \in [0,\infty)$, for any $(\mu,\sigma)$ ordered pair resulting in a fixed mean of
$M(x;\mu,\sigma)_{|x \ge 0} = 1$, for example, 
 the variance appears to approach an asymptote at $\mathrm{Var}(x;\mu,\sigma)_{|x \ge 0} = 1$,
as indicated by grid search calculations whose results are shown in
 Figure~\ref{fig:sigma_var_vs_mu_grid}.
In these cases, the domain of the model can be extended to $x \in (-\infty,+\infty)$, affording
 unlimited choices for the variance, but since the data is positively valued,
significant artifact may thereby be introduced.
\begin{figure}[h]
\vspace{-0.25cm}
\hspace{2cm}
\includegraphics[width=4.5in]{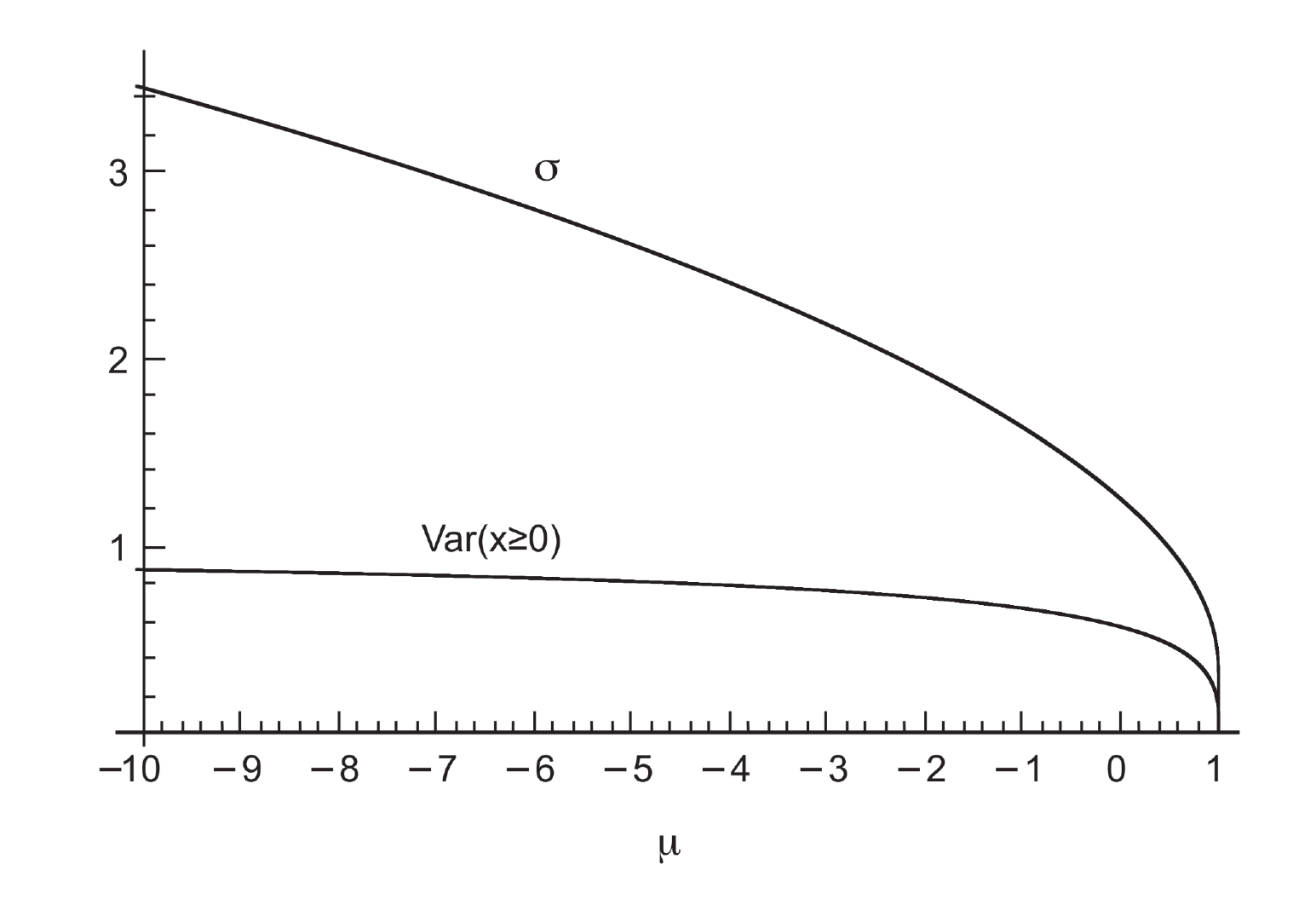}
\vspace{-0.5cm}
\caption{\small Ordered pairs $(\mu,\sigma)$ resulting in a fixed mean.  For each value of $\mu$,
the plot indicates the value of $\sigma$ resulting in a mean of 1 for a Gaussian distribution
over the interval $x \in [0,\infty)$.  Also shown is the resulting variance over the same interval,
which appears to approach 1 as $\mu$ becomes increasingly negative. The calculations involved successively finer
 2-D grid searches of $(\mu,\sigma)$ space.}
\label{fig:sigma_var_vs_mu_grid} 
\end{figure} 
The question arises as to whether the variance of Gaussian distributions having
fixed, finite means over their positive support, or over other unilaterally truncated intervals,
is, in fact, theoretically bounded. If it is, the questions then become what the bounds are,
whether they can be related to the mean, $M$, and whether the variance, itself, can be expressed 
strictly as a function of $M$ and the location parameter $\mu$. An affirmative answer
would inform attempts at using unilaterally truncated Gaussian distributions, or UTGDs,
in data modeling, and, for data that
is known to be Gaussian in form, would provide theoretical limits on the variance.

As pointed out by Horrace (\cite{Horrace2015}), Bera and Sharma (\cite{Bera1999}) showed that
over the interval of $x \in [0,\infty)$ the ratio of $M^2/\mathrm{Var}(x) > 1$.
In the current work, this idea is extended to prove that
for any arbitrary Gaussian distribution with parameters $\{\mu,a\} \in \mathbb{R}$,
 $\sigma \in \mathbb{R^{+}}$, and $a < M$, the least upper bound of the variance over the interval
$x \in [a,\infty)$, 
$\sup \mathrm{Var}(x)_{|x \ge a}=(M-a)^2$.
This is done by relating $\mu, \sigma$ and the cutoff value, $a$, through an additional parameter, 
$r = (\mu-a)/\sigma$, expressing the variance as a function of $M$, $r$, and $a$, and then 
fixing the mean and cutoff.
By symmetry, then, $ \sup \mathrm{Var}(x)_{|x \le a}=(M-a)^2$, for $a > M$.

The derivations are presented in detail for left-sided $( x \in [a,\infty))$ truncation of general,
one-dimensional (1-D) Gaussians. The results for right-sided truncation are similar and are summarized.
Although the derivations are carried out for the general case of arbitrary Gaussian distributions that
are not necessarily normalized, normalized distributions--i.e., probability density functions (pdfs)--are used
as illustrative examples in some cases.

\begin{singlespace}
There are several other main results in the current work:
\begin{itemize}
\item The variance of a UTGD,
$\mathrm{Var}(x)_{|x \ge a}$ (or $\mathrm{Var}(x)_{|x \le a}$),
can be expressed, at least approximately, strictly as a function of $\mu, M$, and $a$--i.e.,
$\mathrm{Var}(x;M,\mu)_{|x \ge a} = \tilde{\sigma}(\mu)^2 + (M-a)(\mu-a) -(M-a)^2$, 
where $\tilde{\sigma}(\mu)$ is a function approximating $\sigma(\mu)$. One application 
of such a function is to facilitate the determination of $\sigma$ and $\mu$ parameters for UTGDs
directly from the mean and variance by reducing the problem from two model degrees of freedom to one.
A general approach to the development of these approximating functions is described.

\item The variance of UTGDs can be determined from $M$ and the relative height of the distribution 
at the boundary, $H = f(a)/f_{\!_M}$, where $f_{\!_M}$ is the distribution function's maximum. The 
variance decreases as $H$ decreases. For a fixed mean, a limiting threshold
value for this parameter determines the maximal variance of the distribution.
This implies that what is conventionally considered a Gaussian distribution
of positive-valued data--i.e., without a large truncation of the left tail--will
typically have a maximal variance of $\approx 0.2M^2$.

\item It is demonstrated that, for a given UTGD, $d\sigma/d\mu$ under a fixed mean and variance differs
depending on the particular analytic form of the variance used,
and hence the corresponding approximations to $\sigma(\mu)$ curves under constant mean and variance also differ.
This property can be used to estimate $\mu$ and $\sigma$ for UTGDs from the intersection of
the curves so as to reproduce the first and second moments of the data. Successive approximation using this approach rapidly
converges to highly accurate results. In addition, it is demonstrated that, for log-normal distributions, 
the moments of either the original or the log-transformed data can be reproduced.

\item The kurtosis of a unilaterally truncated
Gaussian distribution has a minimum value of $\approx 2.757$ at an $r$ value of $\approx 1.874$.

\item Numerical data is presented for the maximal variance of truncated, scaled chi distributions with $n$ 
distributional degrees of freedom, or dimensions.  These distributions can be thought of as arising from
radially symmetric $n$-dimensional Gaussian distributions,
$g(R;M,r,n)_{|R \ge a}dR$, where $R$ is the magnitude of the sample vector, and $a$ is, again, the cutoff value.
It is shown here that for a fixed mean, $M$, and for $n \in \mathbb{Z}+$,
the maximal variance, $\mathrm{V_{max}}(R)$ over the region $R \in [a,\infty)$ 
appears to be bounded by the variance of the half-normal distribution,
$M^2(\pi-2)/2$, which occurs at $n=1, a=0$.
For fixed dimensionality ($n>0$), the maximal variance always seems to occur at a normalized cutoff of $|r| = a/\sigma \rightarrow 0$ 
and decreases with dimension as $\sim 1/\delta n$, where $\delta$ ranges within $[\pi/2,2]$.
By contrast, for fixed cutoff, there is a local maximum in
the variance with respect to $n$. The number of
dimensions resulting in the maximal variance, $n_{\mathrm{vmx}}$,
increases with cutoff value or $|r|$, as $n_{\mathrm{vmx}} \sim r^2 = a^2/\sigma^2$, while
the maximal variance, itself, decays asymptotically as an exponential for increasing cutoffs.
If $n$ is allowed to span the reals, the maximal variance over $R \in [a,\infty)$ for fixed $n$ increases further
as $n$ falls below 1 and appears to be unbounded in the region of $n \in [-2,0]\subset 
\mathbb{R}$, as $a \rightarrow 0$,
 and locally maximal at $n=-1$. 

\vspace{0.5cm}

\end{itemize}
\end{singlespace}
\vspace{-1em}
In cases where the models are over-specified, the divergence of parameter values
estimated using different moment functionals, which tends to correlate negatively with model accuracy, 
is illustrated.
Example applications of some of these findings are presented.

\section{Maximal variance of UTGDs, fixed mean and cutoff}
\label{sec:max_var_UTGD}
\subsection{Derivation of $\mathrm{Var}(x;M,r)_{|x \ge a}$}
\label{subsec:derive_var}
Consider an arbitrary Gaussian distribution function $f(x;\mu,\sigma)$, which is not necessarily normalized:
\beq
f(x;\mu,\sigma) = f_{\!_{M}} e^{-(x-\mu)^2/2\sigma^2}, 
\label{gaussian_mu_sig}
\eeq
\begin{figure}[H]
\hspace{-0.75cm}
\includegraphics[width=7.0in]{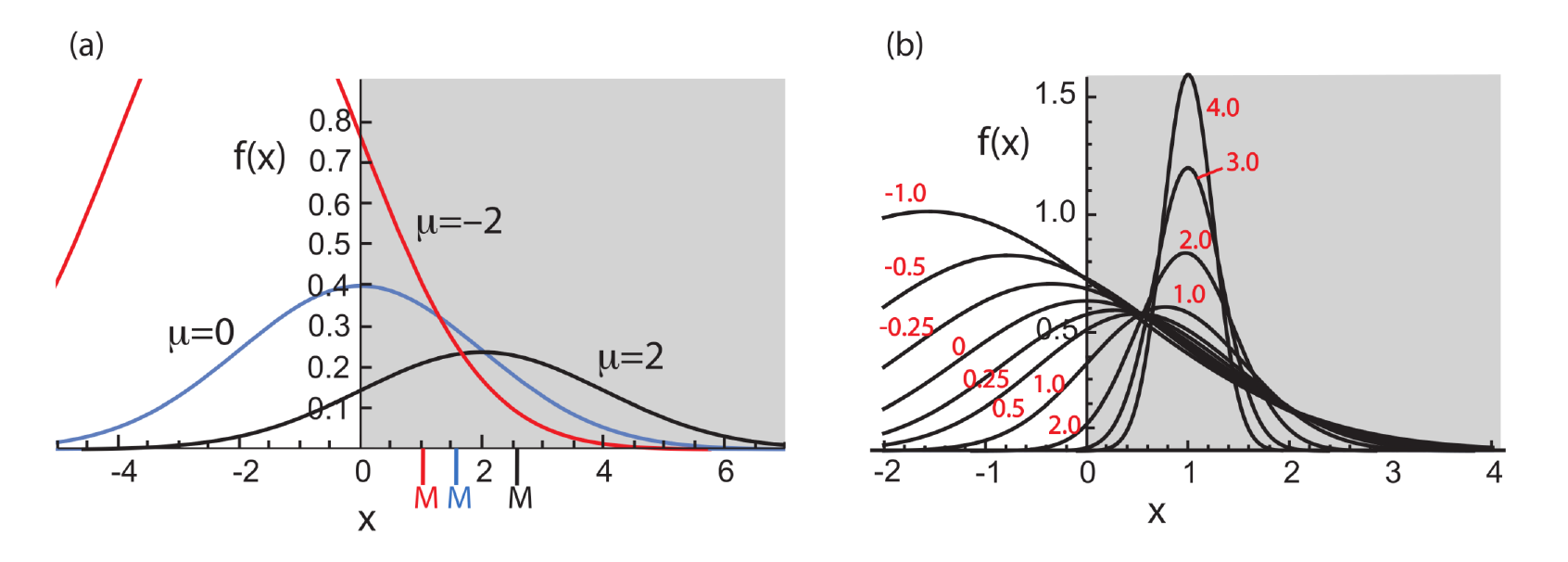}
\caption{\small \noindent Gaussian probability density functions. 
{\bf{Panel (a)}}: three Gaussian probability density functions (pdfs) with support over the shaded interval,
$x \in [0,\infty)$. For all three density functions, $\sigma=2$, but $\mu$ varies, as shown.
The means over the shaded interval are indicated by the \reflectbox{''}$M$''s.
They are approximately 1.0502, 1.5958, and 2.5752.  The variances are 0.7964, 1.4535, and 2.5187, respectively.
{\bf{Panel (b)}}: family of Gaussian pdfs with support over $x \in [0,\infty)$, 
fixed mean ($M(x)_{|x \ge 0}=1$) and varying $r$ values, which are
indicated in red. The corresponding $\mu$ and $\sigma$ values are listed in Table ~\ref{tab:mu_sig_from_r}.
The variance decreases with higher $r$ values, while the center of the pdfs, $\mu$,
approaches the mean from the left.}
\label{fig:3Gaussians_and_family}
\end{figure}
\noindent where $f_{\!_{M}} \in \mathbb{R^+} $ is the function's maximum, $f_{\!_{M}} = f(x=\mu)$.
The probability distribution is $f(x;\mu,\sigma) dx$, where $dx$ is a small interval in $x$, but
this notation is less important here than in the symmetric, multidimensional case described later.
The mean of this distribution over the interval of interest, $x \in [a,\infty)$, is
\beq
M(x;\mu,\sigma)_{|x \ge a} = 
\mu + \frac{\sqrt{2/\pi} \sigma e^{-(\mu-a)^2/2\sigma^2}}
{\erf\!\left((\mu-a)/\sqrt{2}\sigma\right) + 1}.
\label{mean_mu_sig}
\eeq
In this work, the $k$th (raw) moment is written $M_k$, but we use $M = M_1$ as shorthand for the mean.
See Figure~\ref{fig:3Gaussians_and_family}, Panel (a),
for examples of probability density functions (pdfs) with support over the interval
$x \in [0,\infty)$.
An inverse of the function in Eq.~(\ref{mean_mu_sig}), e.g., $\sigma(M,\mu)$, is sought, but no
 such inverse can be found in closed form.  Instead, let $\mu=r\sigma + a$ , where $\{r,a\} \in \mathbb{R}$ 
and $a \le M$. Then 
\beq
f(x;\sigma,r,a) = 
f_{\!_{M}} e^{-(x-r\sigma-a)^2/2\sigma^2}.
\label{trans_gaussian}
\eeq
and the mean becomes
\beq
M(x;\sigma,r)_{|x \ge a} 
= \sigma\left( r + \frac{\sqrt{2/\pi}} {e^{r^2/2}\xi(r)}\right) + a,
\label{mean_sig_r}
\eeq
where $r = (\mu -a)/\sigma$, and $\xi(r)= \erf\!\left(r/\sqrt{2}\right)+ 1$ is used as shorthand. Solving for
$\sigma$,
\beq
\sigma(M,r)_{|x \ge a} = \frac{(M-a) e^{r^2/2} \xi(r)}
{re^{r^2/2} \xi(r) + \sqrt{2/\pi}}.
\label{sigma_M_r}
\eeq
Note that $0< \xi(r) <2$ and $re^{r^2/2} \xi(r) + \sqrt{2/\pi} >  0$ for all $r \in \mathbb{R}$,
so the expressions in Eqs.~(\ref{mean_sig_r}) and (\ref{sigma_M_r}) are well-defined for all 
real $r$. 
Hence, for specified mean and $r$ value, 
$\sigma$ is fixed to the above value, while $\mu=r\sigma + a$ is
 also fixed.  For left-sided truncation, we assume $M>a$ throughout the analysis, so $\sigma > 0$ for all real $r$.
\begin{table}[H]
\setstretch{1}
\tabcolsep=0pt
\begin{tabular*}{\columnwidth}
{|@{\hspace{6pt}}@{\extracolsep{\fill}}rrrr@{\hspace{6pt}}||@{\hspace{6pt}}rrrr@{\hspace{6pt}}|}
\hline
\hline
\multicolumn{1}{|c}{r} &
\multicolumn{1}{c}{$\sigma$} &
\multicolumn{1}{c}{$\mu$} &
\multicolumn{1}{@{\hspace{0.5em}}c@{\hspace{0.5em}}||}{$\mathrm{Var}(x)_{|x \ge 0}$} &
\multicolumn{1}{c}{r} &
\multicolumn{1}{c}{$\sigma$} &
\multicolumn{1}{c}{$\mu$} &
\multicolumn{1}{@{\hspace{0.5em}}c@{\hspace{0.5em}}|}{$\mathrm{Var}(x)_{|x \ge 0}$} \\
\hline
-2.00 &  2.67942 &-5.35883  & 0.82044 & 0.25  &  1.11627 & 0.27907  & 0.52513 \\
-1.00 &  1.90427 &-1.90427  & 0.72198 & 0.50  &  0.99092 & 0.49546  & 0.47739 \\
-0.75 &  1.72778 &-1.29583  & 0.68938 & 1.00  &  0.77664 & 0.77664  & 0.37981 \\
-0.50 &  1.55987 &-0.77994  & 0.65327 & 2.00  &  0.48656 & 0.97312  & 0.20986 \\ 
-0.25 &  1.40144 &-0.35036  & 0.61366 & 3.00  &  0.33284 & 0.99852  & 0.10931 \\ 
0.00  &  1.25331 & 0.00000  & 0.57080  & 4.00  &  0.24999 & 0.99997  & 0.06246 \\
\hline
\end{tabular*}
\caption{\small Values of $\mu$ and $\sigma$ corresponding to various $r$ values, and the
resulting variance, over the interval $x \in [0,\infty)$, for a Gaussian pdf with
mean $M(x)_{|x \ge 0}=1$. (From Eqs.~(\ref{sigma_M_r}) and (\ref{var_M_r}).)}
\label{tab:mu_sig_from_r}
\end{table}

By substituting $\sigma$ into Eq.~(\ref{trans_gaussian}), the distribution $f(x)$ can now be expressed strictly 
in terms of the mean, $M$, the truncation boundary $a$, and $r$ (see Appendix, Section~\ref{subsec:form_of_f}).
Figure~\ref{fig:3Gaussians_and_family}, Panel (b), shows the family of Gaussian pdfs normalized to the interval $x \in [0,\infty)$ 
for fixed mean of $M(x)_{|x \ge 0}=1$ and varying $r$. Figure~\ref{fig:mu.and.s.vs.r.and.var.vs.r.horiz}
illustrates the dependence on $r$ of the parameters $\mu$ and $\sigma$ (Panel a) as well as the 
variance over this same interval for the same, fixed mean ($M(x)_{|x \ge 0}=1$, Panel b). 

The second moment is

\begin{figure}[h]
\vspace{-3cm}
\hspace{-1.5cm}
\includegraphics[width=7.5in]{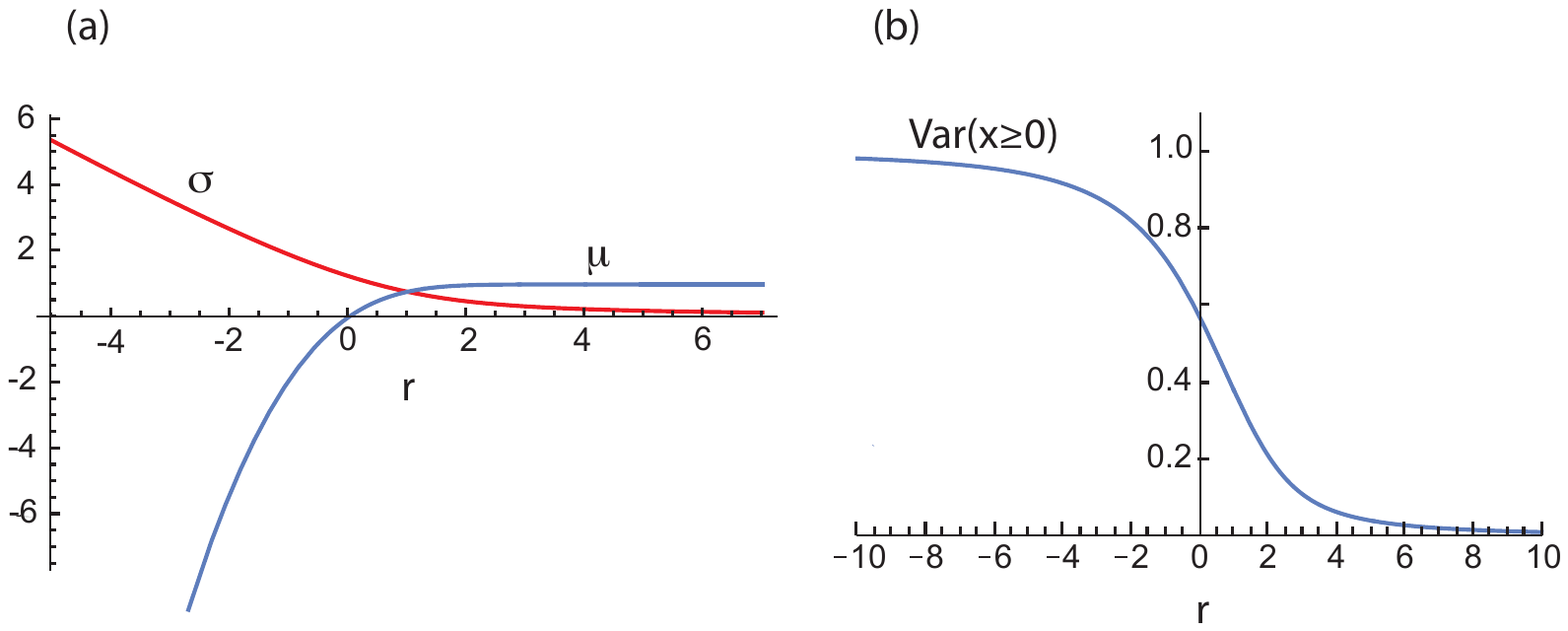}
\vspace{-4cm}
\caption{\small Dependence of parameters and variance on the $r$ value for Gaussian distributions over the
interval $x \in [0,\infty)$, in the case of $M=1$. {\bf{Panel (a)}}: the $\mu$ and $\sigma$ parameters 
plotted as a function of the r. {\bf{Panel (b)}}: the variance plotted as a function of $r$.}
\label{fig:mu.and.s.vs.r.and.var.vs.r.horiz}
\end{figure}

\beq
M_2(x;\sigma,r)_{|x \ge a} = 
\sigma^2 + (r \sigma + a)^2  + \frac{\sqrt{2/\pi} \sigma(r\sigma + 2a)} {e^{r^2/2}\xi(r)}.
\label{M_2_sig_r}
\eeq
Combining Equations (\ref{mean_sig_r}) and (\ref{M_2_sig_r}), the variance is 
$\mathrm{Var}(x;\sigma,r)_{|x \ge a} = M_2 - M^2 =$
\beq
\sigma^2\left(1-\frac{\sqrt{2/\pi} r}{e^{r^2/2} \xi(r)} - \frac{2}{\pi e^{r^2} \xi(r)^2}\right)=
 \sigma^2 \left( 1- \frac{\sqrt{2\pi}~r e^{r^2/2}\xi(r) + 2}{\pi e^{r^2}\xi(r)^2}\right).
\label{var_sig_r}
\eeq
We call this the unsubstituted form or Form I of the variance.
Then, substituting the expression for $\sigma$ from Eq.~(\ref{sigma_M_r}) 
into Eq.~(\ref{var_sig_r}), the variance can be expressed in terms of the mean and $r$, as
\beq
\mathrm{Var}(x;M,r)_{|x \ge a} =
\frac{(M-a)^2\left(\pi e^{r^2}\xi(r)^2 -
\sqrt{2\pi} r e^{r^2/2} \xi(r) -2\right) }
{\pi\left(r e^{r^2/2} \xi(r)+\sqrt{2/\pi}\right)^2 }
\label{var_M_r}
\eeq
after simplification. This is the substituted form or Form II of the variance. The distinction between
the forms of the variance is not only important in derivations but also useful in model parameterization, 
or calibration, as described in later sections.

It is proven in the next subsection (\ref{subsec:proof_var}) that $ \sup \mathrm{Var}(x;M,r)_{|x \ge a} = (M-a)^2$.
Thus, for any arbitrary unilaterally truncated Gaussian probability distribution function
with a finite mean, $\mathrm{Var}(x)_{|x \ge a} < (M-a)^2$, and, in particular, $ \sup \mathrm{Var}(x)_{|x \ge 0} = M^2$. 
The larger variances
occur with more highly negative $\mu$ (or $r$) values (see Table \ref{tab:mu_sig_from_r}), where the included portions
of the distribution shift entirely to flatter portions of the right tail.

As illustrated in Figure~\ref{fig:mu.and.s.vs.r.and.var.vs.r.horiz}, Panel (b),
$\mathrm{Var}(M,r)_{|x \ge a}/(M-a)^2$  is monotonically decreasing with $r$. Although this is not formally 
proven here for all $r \in \mathbb{R}$, it can be verified that $\partial (\mathrm{Var}(x;M,r)_{|x \ge 0})/\partial r < 0$ 
for sampled $r$ over at least an interval of  
$r \in [-10^{5},10^{5}]$.
A number of these results for the case of $M=1$ are listed in the Supplementary Material (Table \ref{tab:end_values_part_var},
Section \ref{subsec:val_partial_der_var}).
For example, at $r = -2^{18}=-262144$, the derivative of the variance remains negative while the variance
 is $\approx 1 - 3 \times 10^{-11}$.
In addition, as shown analytically in the Supplementary Material (Section~\ref{subsec:derive_var_large_r}),
the derivative of the variance approaches zero from below as $r \rightarrow \pm \infty$.
Finally, as shown in the Appendix (Section \ref{subsec:lemma_proofs}), the limiting variance as $r$
becomes very large and negative is $(M-a)^2$.

\subsection{Proof that  $\mathrm{sup Var}(x;M,r)_{|x>a}$ = $(M-a)^2$}
\label{subsec:proof_var}

There are six supporting lemmas for the proof. Each concerns a
UTGD function $f(x;\mu,\sigma)$ with mean $M$ over the interval
$x \in [a,\infty)$. They are all listed here, but the proofs of the first five are 
straightforward and are detailed in the Appendix (Section~\ref{subsec:lemma_proofs}).
\begin{lemma}
$\lim_{r \to -\infty} \mathrm{Var}(x;M,r)_{|x \ge a} = (M-a)^2.$
\label{lemma:limit_var}
\end{lemma}
\begin{lemma}
$\lim_{r \to +\infty} \mathrm{Var}(x;M,r)_{|x \ge a} = 0.$
\label{lemma:limit_var_pos_inf}
\end{lemma}
\begin{lemma}
The variance $\mathrm{Var}(x;M,r)_{|x \ge a}$ is differentiable with respect to $r$ at least 
at all finite values of $r$.
\label{lemma:var_different}
\end{lemma}
\begin{lemma}
$\partial \sigma(M,r)/\partial r = -\mathrm{Var}(x;M,r)_{|x \ge a}/(M-a).$
\label{lemma:dsigdr}
\end{lemma}
\begin{lemma}
\beq
\mathrm{Var}(x;M,r,\sigma)_{|x \ge a} = \sigma^2 + (M-a) r\sigma - (M-a)^2.
\label{var_M_r_sig}
\eeq
\label{lemma:var_M_r_sig}
\end{lemma}
\begin{lemma}
If there exists $r = r_0 \in \mathbb{R}$ such that $\mathrm{Var}(x;M,r_0,\sigma)_{|x \ge a}$ is
a global maximum, then $\mathrm{Var}(x;M,r_0,\sigma)_{|x \ge a} < (M-a)^2$.
\label{lemma:glob_max}
\end{lemma}

\begin{proof}
From Eq.~(\ref{var_M_r_sig}) in Lemma~\ref{lemma:var_M_r_sig},
writing the variance and $\sigma$ as functions of $r$, we have:
\beq
\mathrm{Var}(x;M,r,\sigma)_{|x \ge a} = \sigma(r)^2 + (M-a) r\sigma(r) - (M-a)^2.
\label{var_M_mu_sig}
\eeq
Since the variance is proportional to $(M-a)^2$ (from Eq.~(\ref{var_M_r})),
it is sufficient to set $M =1, a=0$ and show that $\mathrm{Var}(x;M=1,r=r_0)_{|x \ge 0} < 1.$
Hence, begin with \\
$\mathrm{Var}(r) = \sigma(r)^2 + r\sigma(r) - 1$,
where $\mathrm{Var}(r)$ is shorthand for $\mathrm{Var}(x;M=1,r)_{|x \ge 0}$.

Solving for $\sigma(r)$, $\sigma(r) = -\left(r \pm \sqrt{4(\mathrm{Var}(r)+1)+r^2}\right)/2$.

Taking derivatives with respect to $r$, and since from Lemma~\ref{lemma:dsigdr} it is known that
$-\partial \sigma/\partial r = \mathrm{Var}(x;1,r)_{|x \ge 0}$
\beq
\frac{\partial \sigma}{\partial r} = -\frac{1}{2} \left(1 \pm \frac{2 \mathrm{Var}^{\prime}(r) + r}
{\sqrt{4(\mathrm{Var}(r)+1)+r^2}}\right)  = -\mathrm{Var}(r),
\eeq
where $\mathrm{Var}^{\prime}(r) = \partial \mathrm{Var}(r)/\partial r$.

Suppose that there exists a value of $r = r_0 \in \mathbb{R}$ which globally 
maximizes the variance, so that $\mathrm{Var}(r_0) = \mathrm{V_{max}}$, where $\mathrm{V_{max}}$ refers
to the maximal variance under a fixed mean. Then $\mathrm{Var}^{\prime}(r_0) = 0$, 
because by a corollary to Fermat's Theorem (\cite{FermatTheoremWiki}), 
a global maximum that occurs at a non-boundary value of $r$ 
must also occur at a stationary point, provided the function is differentiable there.
By Lemma~\ref{lemma:var_different}, $\mathrm{Var}(r)$ is differentiable at least at all finite values of $r$.
Hence,
\beq
-\frac{1}{2} \left(1 \pm \frac{r_0}
{ \sqrt{4(\mathrm{V_{max}}+1)+{r_0}^2}}\right)  = -\mathrm{V_{max}}.
\eeq
Solving for $r_0$, then,
\beq
r_0 = \pm \frac{2 \mathrm{V_{max}-1}}{\sqrt{1-\mathrm{V_{max}}}}
\sqrt{\frac{1+\mathrm{V_{max}}}{\mathrm{V_{max}}}}.
\eeq
Since $r_0 \in \mathbb{R}$, for either solution, $\mathrm{V_{max}} < 1$,
which means that $\mathrm{Var}(x;M,r=r_0)_{|x \ge a} < (M-a)^2.$
\end{proof}
\begin{theorem}
For a Gaussian probability distribution $f(x;\mu,\sigma)$ with mean $M$,
over the interval $x \in [a,\infty)$, $ a < M$, or else over the interval
$x \in (-\infty,a]$, $ a > M$,
\begin{singlespace}
\bdm
\mathrm{Var}(x)_{|x<a \oplus x>a} < (M-a)^2.
\edm
\end{singlespace}
\label{theorem:main}
\end{theorem}
\begin{proof}
The maximum or supremum of $\mathrm{Var}(x;M,r)_{|x \ge a}$ occurs either at
A) $r \rightarrow -\infty$, B) $r \rightarrow +\infty $, 
or C) a finite value of $r$ that can be called $r_0$.
Since from Lemma~\ref{lemma:limit_var}, 
$\lim_{r \to -\infty} \mathrm{Var}(x;M,r)_{|x \ge a} = (M-a)^2$
and since from Lemma~ \ref{lemma:limit_var_pos_inf}, $\lim_{r \to +\infty} \mathrm{Var}(x;M,r)_{|x \ge a} = 0$,
it is clear that B) is false.  This means either A) or C) is true, or both. 

By Lemma~\ref{lemma:limit_var}, it is clear that if A) is true, $\sup \mathrm{Var}(x)_{|x \ge a} = (M-a)^2$.
By Lemma~\ref{lemma:glob_max}, even if C) were true, the 
maximum of the variance would be $ < (M-a)^2$. Since it is known that the variance
approaches $(M-a)^2$ as $r \rightarrow -\infty$, this would imply $\sup \mathrm{Var}(x)_{|x \ge a} = (M-a)^2$.
Hence, A) is true, and whether C) is true or not, $\sup \mathrm{Var}(x)_{|x \ge a} = (M-a)^2$.
By symmetry, then, $\sup \mathrm{Var}(x)_{|x \le a} = (M-a)^2$, for $M < a$, as well.
\end{proof}
\subsection{Related observations}
\label{subsec:related_obs}
\begin{itemize}[label={\scriptsize$\bullet$}]
\item Some other relevant limits are:
\beq
\mathrm{From~Eq.~(\ref{var_M_r})}:~\lim_{r \to 0} \mathrm{Var}(x;M,r)_{|x \ge a} = (M-a)^2(\pi-2)/2.
\label{lim_var_r_to_0}
\eeq
\beq
\mathrm{From~Eq.~(\ref{sigma_M_r})}:~\lim_{r \to 0} \sigma(M,r)_{|x \ge a} = (M-a) \sqrt{\pi/2},
\label{lim_sigma_r_to_0}
\eeq
\bdm
\noindent \lim_{r \to -\infty} \sigma(M,r)_{|x \ge a} = +\infty \,\text{,~ and }\lim_{r \to \infty} \sigma(M,r)_{|x \ge a} = 0.
\edm
Noting that $\mu = r \sigma + a$, from Eq.~(\ref{sigma_M_r}), then  
\bdm
\lim_{r \to -\infty} \mu(M,r)_{|x \ge a} = -\infty, ~\lim_{r \to \infty} \mu(M,r)_{|x \ge a} = M,
\edm
\beq
\text{and}~\lim_{r \to 0} \mu(M,r)_{|x \ge a} = a.
\label{lim_mu_r_to_0}
\eeq
These limits hold for any arbitrary Gaussian distribution with a finite mean.
\item
From Eq.~(\ref{sigma_M_r}), 1st-order Laurent series expansion shows that for large, positive $r$, 
\beq
\sigma \sim \frac{M-a}{r}\left(1-\frac{e^{-r^2/2}}{\sqrt{2\pi} r}\right) \sim \frac{M-a}{r},
\label{sigma_sim_M_r}
\eeq
and for a similar expansion of $\mu = r \sigma + a$ for large, positive $r$,\\
$\mu \sim M - (M-a) e^{-r^2/2}/\sqrt{2\pi} r \sim M$.
Strictly, $\mu < M$ in all but the limiting case of $r \rightarrow \infty$,
as was suggested by the results of the numerical calculations shown in
 Figure~\ref{fig:sigma_var_vs_mu_grid}
and Table \ref{tab:mu_sig_from_r}.

\item In the case of $a=0$, Equations (\ref{lim_var_r_to_0}), (\ref{lim_sigma_r_to_0}), 
and (\ref{lim_mu_r_to_0}) 
correspond to the half-normal distribution (\cite{HalfNormalWiki}), $f(x;\mu = 0,\sigma)_{|x \ge 0}$,
and substituting in the well-known result of $M(x)_{|x \ge 0} = \sigma \sqrt{2/\pi}$, 
we recover $\sigma$ from Eq.~(\ref{lim_sigma_r_to_0}), as well 
as the standard result of $\mathrm{Var}(x)_{|x \ge 0} = (\pi-2)\sigma^2/\pi$ 
from Eq.~(\ref{lim_var_r_to_0}).

\item As illustrated in Figure~\ref{fig:3Gaussians_and_family}, Panel (b),
Gaussian probability density functions over $x \in [0,\infty)$ 
with mean $M=1$ and $r \le 0$ pass through a point close to $(\frac{1}{2},\frac{1}{\sqrt{e}})$. 
In general, for 
Gaussian pdfs over this interval, there is a cluster of points around $(\frac{M}{2},\frac{1}{M \sqrt{e}})$
as a function of $r\le 0$,
because the curves are rather flat in $r$ in this region and 
$\lim_{r \to -\infty} f(x;M,r)_{|x \ge 0} = e^{-x/M}/M.$

\item Since $\sigma = (\mu -a)/r$, from the relation in Lemma~\ref{lemma:var_M_r_sig}, Eq.~(\ref{var_M_r_sig}), 
\beq
\mathrm{Var}(x;M,\mu,\sigma)_{|x \ge a} = 
\sigma^2 + (M-a)(\mu-a) - (M-a)^2,
\label{var_M_sig}
\eeq
provided that $\sigma,\mu$ and $M$ are consistent, or congruent. This equation is used below 
(Section~\ref{sec:approx_sigma_of_mu}) in developing the $\sigma(\mu)$ approximations. 
It can be seen from this relation that as $\mu \rightarrow M$ (i.e., for large, positive $r$), 
$\mathrm{Var}(x)_{|x \ge a}$ approaches $\sigma^2$. From Eq.~(\ref{var_M_r}), $\mathrm{Var}(x)_{|x \ge a}$ 
approaches $(M-a)^2/r^2$ for large, positive $r$ (see also Lemma~\ref{lemma:limit_var_pos_inf} above),
and together, these relations are consistent with Eq.~(\ref{sigma_sim_M_r}).
\end{itemize}

\subsection{Right-sided truncation}
\label{subsec:right_hand_trunc}
For the case of truncation of the right side of the Gaussian--i.e., with $x \in (-\infty,a]$ and 
$a > M$, the relations are similar to those detailed above for left-sided truncation.
The expressions for the mean, $M(x;\sigma,r)_{|x \le a}$, the second moment, 
$M_2(x;\sigma,r)_{|x \le a}$ and $\sigma(M,r)_{|x \le a}$
are identical to their left-sided truncation ($x \ge a$) counterparts
shown in Eqs.~(\ref{mean_sig_r}), (\ref{sigma_M_r}), and (\ref{M_2_sig_r}),
 except that $\xi(r)$ is replaced by $\xi^\dag(r) = \erf\!\left(r/\sqrt{2}\right)- 1$. 
For the variance, $\mathrm{Var}(x;M,r)_{|x \le a} = \mathrm{Var}(x;M,-r)_{|x \ge a}$, from 
Eq.~(\ref{var_M_r}), so it is monotonically increasing in $r$. 
The limits given in Section~\ref{subsec:related_obs} are likewise reversed with respect to $r$. For example, 
\bdm
 \lim_{r \to +\infty} \mathrm{Var}(x;M,r)_{|x \le a} = 
\lim_{r \to -\infty} \mathrm{Var}(x;M,r)_{|x \ge a} = (M-a)^2,
\edm
\bdm
\mathrm{and}~~~\lim_{r \to -\infty} \mu(M,r)_{|x \le a} =
\lim_{r \to +\infty} \mu(M,r)_{|x \ge a} = M.
\edm
Eq.~(\ref{var_M_sig}) holds regardless of the side of the truncation.
\section{Dependence of variance on truncation threshold}
\label{sec:max_var_threshold}
From Figure~\ref{fig:3Gaussians_and_family}, Panel (b),
it can be seen that what is typically referred to as
a Gaussian distribution of positive-valued data--i.e., without a large truncation of
the left tail of the Gaussian--has an $r$ value of about 2 or more. From Eq.~(\ref{var_M_r})
or Table \ref{tab:mu_sig_from_r},  
this means that in practice
the maximal variance of positive-valued Gaussian datasets is often $\approx 0.2M^2$ (maximal
standard deviation\footnote{More precisely, for a crossing value threshold of
$5\%$ of the curve's maximum ($f(0) = 0.05 f_{\!_{M}}$), the $r$ value is $\approx 2.448$, implying a maximal
variance of $\approx 0.152 M$ or a maximal s.d. of $\approx 0.39M$. See below.} of $\approx 0.45M$).
For a fixed mean and truncation boundary, 
the smaller the truncation of the left tail, the higher the $r$ value and the lower 
the variance. 
Thus, for a given mean over the interval $x \in [a,\infty)$, the truncation threshold specifies
a minimum $r$ value, which in turn specifies a maximal possible variance for the Gaussian distribution.
This is demonstrated as follows. 

\begin{figure}[h]
\vspace{-1.25in}
\hspace{-1.0cm}
\includegraphics[width=7in]{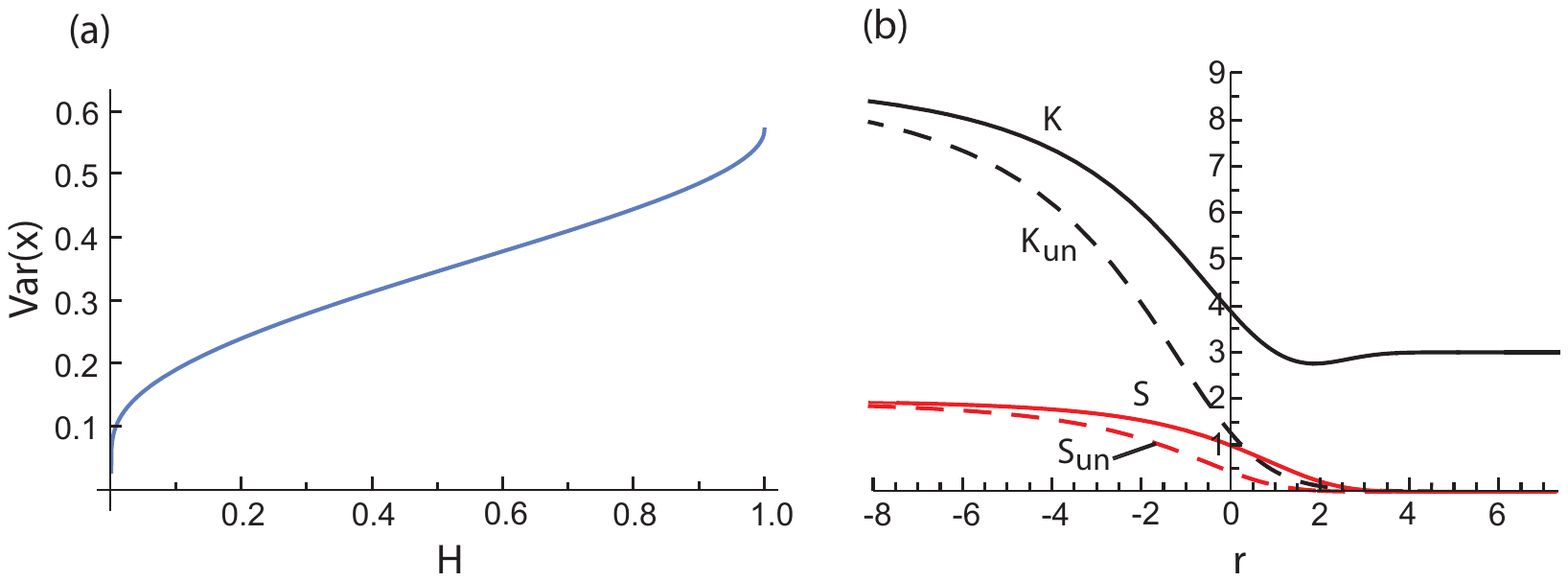}
\vspace{-1.5in}
\caption{\small Plots of the variance vs. boundary height and higher central moments vs. the $r$ value. 
{\bf{Panel (a)}}: the variance, $\mathrm{Var}(x)_{|x \ge 0}$, plotted as a function of the relative
height, $H$, at the truncation boundary for a Gaussian with fixed mean, $M(x)_{|x \ge 0}=1$,
from Eq.~(\ref{var_M_H}). {\bf{Panel (b)}}: higher central moments (CMs) of UTGDs.  The 3rd (red) 
and 4th (black) CMs are plotted as a function of $r$. The solid lines indicate the standardized
CMs (conventional skewness and kurtosis), and the dotted lines the unstandardized
or unnormalized CMs assuming $M=1, a=0$. 
Note K has a minimum below the canonical value of 3 at $r \approx 1.87$.
 K--kurtosis; K\textsubscript{un}--unnormalized kurtosis;
S--skewness; S\textsubscript{un}--unnormalized skewness.
}
\label{fig:var_M_H_and_CM}
\end{figure}
An arbitrary Gaussian function crosses the truncation boundary $x=a$ at a value of 
$f(a;\mu,\sigma) = f_{\!_{M}} e^{-((a-\mu)/\sigma)^2/2}$,
and the maximum value of the distribution is $f(\mu) = f_{\!_{M}}$.
The ratio of the crossing value to the maximal value, which can be called 
$H = f(a)/f(\mu) = e^{-((a-\mu)/\sigma)^2/2} = e^{-r^2/2}$, since $u = r \sigma + a$. Then,
\beq
r = \pm \sqrt{2 \ln(1/H)},
\label{eq_r_H}
\eeq
which are the two possible $r$ values corresponding to $H$.
Because truncating only the left tail is of interest in this section, 
and because the truncation occurs on the right tail when $\mu<a$, 
only the positive branch of the $r$ solutions is considered here (since $r = (\mu-a)/\sigma$).
Substituting $r$ from Eq.~(\ref{eq_r_H}) into the expression for the variance (Eq. \ref{var_M_r}), then
\beq
\mathrm{Var}(x; M,H)_{|x \ge a} = 
\frac{(M-a)^2\left(\pi {\xi_{_H}}^2 - 2 H \xi_{_H} \sqrt{\pi \ln(1/H)} - 2H^2\right) }
{2\left(\xi_{_H}\sqrt{\pi \ln(1/H)} +H\right)^2}
\label{var_M_H},
\eeq
where $\xi_{_H} = \erf(\sqrt{\ln(1/H)}) + 1$.  Hence, Eq.~(\ref{var_M_H}) gives the
maximal variance over $x \in [a,\infty)$ that can be achieved with any Gaussian that
crosses the truncation boundary ($x=a$) at or below $f(a)/f_{\!_{M}} = H$. Since, by definition,
$0< H \le 1$, the relation is always well-defined.
The plot is shown in Figure~\ref{fig:var_M_H_and_CM}, Panel (a),
for the case of $M=1,a=0$. The maximal variance ($(\pi-2)/2$) occurs at $H=1$, 
which corresponds to $r=0, \mu=0$.
\section{Approximating $\sigma(M,\mu)$ and $\mathrm{Var}(x;M,\mu)$ in UTGDs}
\label{sec:approx_sigma_of_mu}
\subsection{General method}
\label{subsec:approx_sig_method}
From the expression for $\sigma(M,r)_{|x \ge a}$ in Eq.~(\ref{sigma_M_r}) and
the fact that $\mu = r \sigma + a$, it is clear that, at least over particular intervals of $\mu$, 
the parameter $\sigma$ as well as $\mathrm{Var}(x)$ 
can be expressed approximately as strict functions of $\mu$, $M$ and $a$. 
This can be useful, for example, in modeling an unknown truncated Gaussian data
 distribution with a known mean and variance, because it reduces the problem 
from two dimensions ($\{\mu,\sigma\}$ space) to one ($\mu$ space), and
$\mu$ can usually be determined easily by recursion. 

The general approach is the following.
Dividing Eq.~(\ref{var_M_sig}) through by  $(M-a)^2$,
\beq
\frac{\mathrm{Var}(x;M,\mu,\sigma)_{|x \ge a}}{(M-a)^2} = 
\left(\frac{\sigma(\mu)}{M-a}\right)^2 + (\mu-a)/(M-a) - 1,
\label{var_norm_sig_M_a}
\eeq 
where $\sigma(\mu)$ is a shorthand for $\sigma(M,\mu)_{|x \ge a}$. 
Since a closed-form analytic expression for $\sigma(\mu)$ cannot be found,
we instead find a function $\tilde{\sigma}(\mu)$ approximating $\sigma(\mu)$
such that 
\beq
\sigma(\mu) \approx \tilde{\sigma}(\mu) = (M-a) \cdot \tilde{\sigma}\!\left(\frac{\mu-a}{M-a}\right),
\label{sigma_mu_U}
\eeq
so that the argument of the $\tilde{\sigma}()$ function on the right is the shifted and scaled $\mu$ parameter.
Letting $U = \frac{\mu-a}{M-a}$, then, from Eqs.~(\ref{var_norm_sig_M_a}) and (\ref{sigma_mu_U}), 
\beq
\frac{\mathrm{Var}(x;M,U)_{|x \ge a}}{(M-a)^2} \approx \tilde{\sigma}\!\left(U\right)^2 + U -1.
\label{var_sig_U}
\eeq
This schema ensures that the parameterizations of the $\tilde{\sigma}(\mu)$ functions are scale-invariant. 
In addition, by effectively transforming the approximations to the $M=1$, $a=0$ case, 
they increase numerical stability by reducing precision-related errors when the parameter
values are very small or large.
From Eq.~(\ref{var_sig_U}), $U$ can usually be determined by recursion (or 1-D search), 
and then $\sigma \approx  (M-a) \cdot \tilde{\sigma}(U)$ and $\mu \approx (M-a)\cdot U +a$.

\begin{table}[h]
  \centering
  \begin{minipage}[t]{0.48\textwidth}
    \centering
\begin{tabular}{@{\hspace{1em}}cr@{\hspace{1em}}}
\hline
\hline
\multicolumn{1}{c}{parameter} & \multicolumn{1}{c}{value} \\
\hline
I) & \\
$c_1$ &  0.8388698504360610\\
$c_2$ & -0.0079952671525833\\
$c_3$ & -0.0000323315551536\\
\hline
II) &\\
$c_1$ &  0.8458237100024218\\
$c_2$ & -0.0076717015064936\\
$c_3$ & -0.0000294382245497\\
\hline
$d_1$ &  0.3625552742358368\\
$d_2$ &  0.0015595446175739\\
$d_3$ &  0.0000206350887888\\
\hline
\end{tabular}
\caption{\small Parameters for Eq.~(\ref{sigma_exp_approx}) and (\ref{var_approx}).
Two alternative parameters sets for $\{c_1,c_2,c_3\}$ are listed.
In I) the fit was done on all 3 parameters by least squares.
In II) $c_1$ was fixed at $\ln(2/(4-\pi))$,
which gives the exact variance and $\sigma$ values at $\mu=0$, and $c_2,c_3$
were fit by least squares.}
\label{tab:c_and_d_param}
  \end{minipage}
  \hfill
  \begin{minipage}[t]{0.48\textwidth}
    \centering
\begin{tabular}{@{\hspace{1em}}cr@{\hspace{1em}}}
\hline
\hline
\multicolumn{1}{c}{parameter} & \multicolumn{1}{c}{value} \\
\hline
\multicolumn{2}{c}{$\mathrm{V^\ddagger_{max}}_{,n}(r)$} \\
\hline
$d_1$  &  0.005395899517140 \\
$d_2$  &  0.044337051307607 \\
$d_3$  &  1.360279573341640 \\
\hline
\multicolumn{2}{c}{$n_{\mathrm{vmx}}(r)$} \\
\hline
$c_1$  &  0.355590614404546 \\
$c_2$  &  2.616552453455175 \\
$c_3$  &  0.087938290974657 \\
\hline
\end{tabular}
\caption{\small Parameter values for estimation of the maximal variance, $\mathrm{V^\ddagger_{max}}_{,n}(r)$
(top panel, Eq. \ref{Vmax_approx})
and the corresponding dimensionality $n_{\mathrm{vmx}}(r)$ (bottom panel, Eq. \ref{nvmx_approx}),
as a function of the $r$ parameter,
for scaled chi distributions over the interval $R \in [a,\infty)$. See Section~\ref{subsec:inner_truncation}
and Figure~\ref{fig:Vmax_approx_comb}, Panel (a) in Appendix (Section~\ref{subsec:approx_max_var_inner}).}
\label{tab:par_nmax_eq}
  \end{minipage}
\end{table}

\subsection{Approximating function 1}
\label{subsec:approx_sigma_fx_1}
Over the interval of approximately $U \in [-100,0.9]$:
\beq
\frac{\sigma(\mu)_{|x \ge a}}{M-a} \approx \tilde{\sigma}(U)_{|x \ge a} = \sqrt{2-e^{-\alpha (1-U)^{\beta}} - U}, 
\label{sigma_exp_approx}
\eeq
where $\alpha= c_1 + c_2 U + c_3 U^2$ and 
$\beta= d_1 + d_2 U + d_3 U^2$.
Two sets of $c$ and $d$ parameters are listed in Table~\ref{tab:c_and_d_param}. They were determined 
by performing separate least-squares fits of $\alpha$ and $\beta$ values, generated for pairs
 of $(\mu,\sigma)$ ordered pairs, to the quadratics in $\mu$.
Substituting Eq.~(\ref{sigma_exp_approx}) into  Eq.~(\ref{var_sig_U}), the variance is found to be
\beq
\mathrm{Var}(x;M,U)_{|x \ge a} \approx (M-a)^2\left(1 - e^{-\alpha (1-U)^{\beta}}\right).
\label{var_approx}
\eeq
The resulting approximations (using parameter set II) 
are compared to the exact results in Figure~\ref{fig:approx_figure}, Panel (a),
over the interval $\mu \in [-10,0.9]$, for $M=1$, $a=0$ ($U = \mu$).
Over this  interval the maximal fractional errors (in absolute value) in $\sigma$ and
$\mathrm{Var}(x)_{|x \ge 0}$ are less than 0.5\% and 2.4\%, respectively.
The absolute and fractional errors over the interval 
$U \in [-10,0.9]$ are shown in Figure~\ref{fig:approx_figure}, Panel (b),
The fractional errors are scale-invariant to within precision-related error. 
Rearranging Eq.~(\ref{var_approx}) gives the following expression for 
determining $U$ by recursion:
\beq
 U = 1- \left(-\frac{\ln \left(1 -\hat{\mbox{V}} \right)}{\alpha}\right)^{1/\beta},
\label{mu_exp_approx}
\eeq
where $\hat{\mbox{V}} = \mathrm{Var}(x;M,U)_{|x \ge a}/(M-a)^2 = 1 - e^{-\alpha (1-U)^{\beta}}$ 
is the normalized variance.

\begin{figure}[H]
\hspace{-1.5cm}
\includegraphics[width=7in]{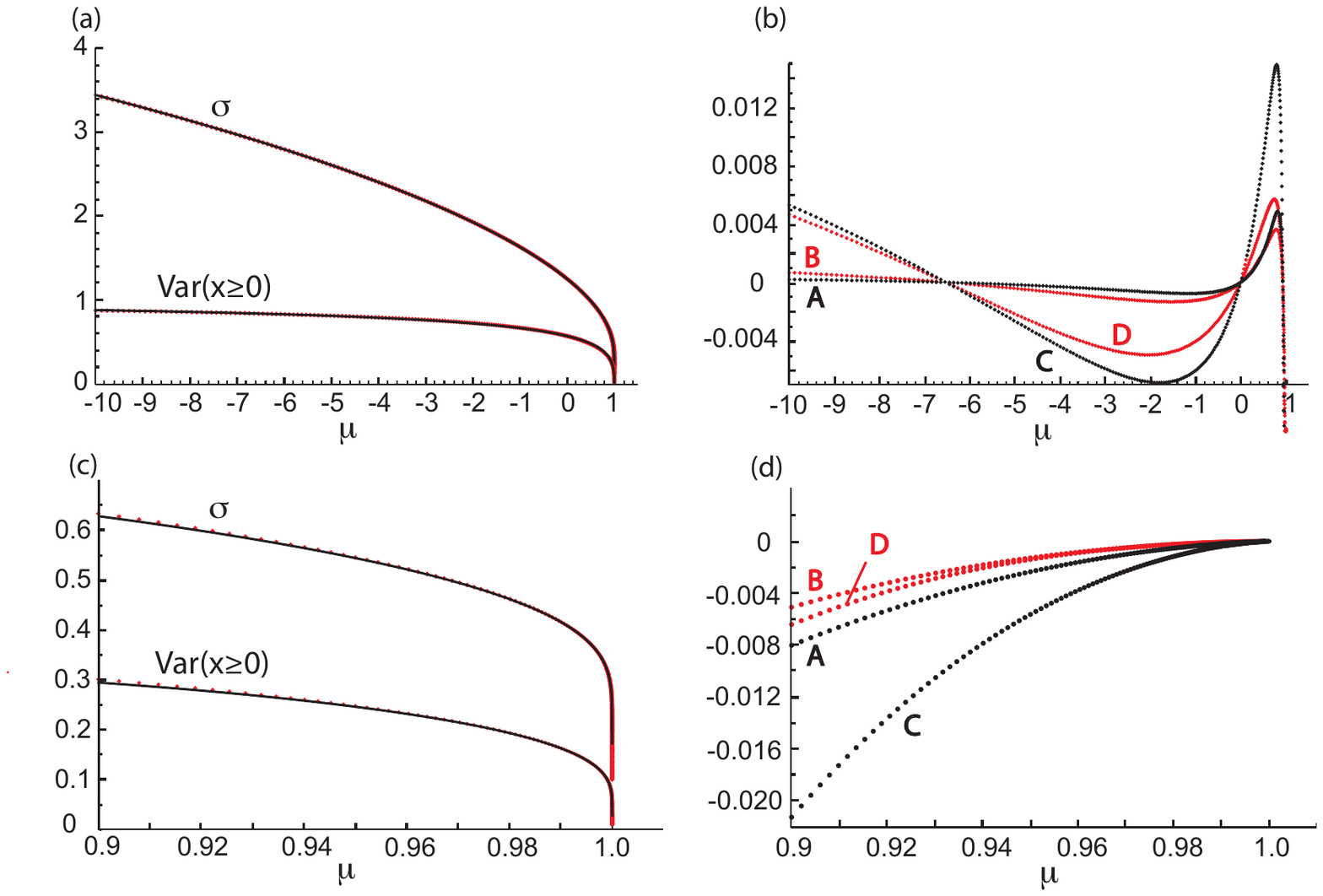}
\caption{\small Plots of parameterization results for UTGDs obtained with approximating functions.
{\bf{Panel (a)}}:  Plots of the variance, $\mathrm{Var}(x)_{|x \ge 0}$, and $\sigma$ as a function of the $\mu$
parameter for Gaussian distributions having a fixed mean (of 1) over the interval $x \in [0,\infty)$,
using approximating function 1 (given by Eq.~(\ref{sigma_exp_approx})).
The red dots represent the exact values (calculated with Eqs~(\ref{sigma_M_r}) and~(\ref{var_M_r})),
and the black curves the functional approximation.
{\bf{Panel (b)}}: Plot of the errors in the data shown in Panel (a).
The black plots (A and C) represent errors in $\mathrm{Var}(x)_{|x \ge 0}$,
while the red plots (B and D) represent errors in $\sigma$. Curves A and B are the differences,
calculated as (approximate - exact). Curves C and D are the fractional differences, calculated
as ((approximate-exact)/exact). Note the change in scale relative to the prior figure. 
{\bf{Panel (c)}}: Results for approximating function 2, given by Eq.~(\ref{approx_Lambert}), for Gaussian
distributions with a mean of 1 ($M=1$) over $x \in [0,\infty)$. 
{\bf{Panel (d)}}: Errors in the results shown in Panel (c)--i.e.,  
approximations obtained with
Eq.~(\ref{approx_Lambert}). The plots in Panels (c) and (d) are otherwise as described above.
}
\label{fig:approx_figure}
\end{figure}

\subsection{Approximating function 2}
\label{subsec:approx_sigma_fx_2}
Approximating function 1 becomes less accurate for $U > 0.9$ .
For $U \in [0.9,1]$,
a different approximation for $\sigma(U)$ can be derived by considering large, positive values of $r$.
Second-order series expansions of $\sigma(M,r)_{|x \ge a}$ from Eq.~(\ref{sigma_M_r})
around $r \rightarrow +\infty$, result in
\bdm
\sigma(M,r)_{|x \ge a} \approx \frac{(M-a)\left(\sqrt{2/\pi}e^{-r^2/2}-2r\right)} {2r^2}.
\edm
Substituting $r = (\mu -a)/\sigma$ and $\mu = (M-a) \cdot U + a$ here and solving for $\sigma$, one obtains
\beq
\frac{\sigma(\mu)_{|x \ge a}}{M-a} \approx 
\tilde{\sigma}(U)_{|x \ge a} = 
\frac{U}{\sqrt{W\!\!\left(\frac{1}{2\pi(1-U)^2}\right)}},  ~~\mbox{where}
\label{approx_Lambert}
\eeq
$W()$ is the Lambert W function. Since $U<1$, this is always well-defined for finite $U$. 

The results are shown in Figure~\ref{fig:approx_figure}, Panel (c), for the case of $M=1$ and $a=0$. 
The errors over the interval $\mu \in [0.9,1]$ 
are all within 0.9\% of the exact values 
for $\sigma$ and within 2.4\% for $\mathrm{Var}(x)$. Figure~\ref{fig:approx_figure}, Panel (d),
shows the errors in the approximation. As expected, the maximal errors 
occur at $\mu=0.9$ and diminish in absolute magnitude for higher values of $\mu$. 
This is the case until approximately $\mu > 1-10^{-10}$, after which there 
are increasing precision-related errors (not shown). At $\mu = 1 - 10^{-9}$,
the fractional errors for $\sigma$ and $\mathrm{Var}(x)_{|x \ge a}$ are approximately 
$10^{-9}$ and $2 \times 10^{-9}$ in absolute value, respectively. 
 
As with other functions approximating $\sigma(M,\mu)$, the Lambert approximation
shown above can be used together with Eq.~(\ref{var_sig_U})
to determine $U$, and therefore $\mu$, by recursion, given the variance, 
$M$, and $a$, as
\beq
U  = \frac{\mathrm{Var}(x)_{|x \ge a}}{(M-a)^2} + 1 -\tilde{\sigma}(U)^2.
\label{general_recur_mu}
\eeq 
Both approximating functions are differentiable.

\section{Moment-intersecting approximations of $\sigma$ and $\mu$}
\label{sec:two_point_approx}
An alternative approach to parameterizing these distributions can be called {\it{moment-intersecting approximation}}
and is related to the more general Method of Moments (\cite{Hazelton2011,MethodofMomentsWiki}).
The current method follows from the
observation that the $\sigma$ values derived from the Form I and Form II variances, under fixed mean,
cutoff value and variance, generally differ. They only converge at a single point on the
$(\mu,\sigma)$ plane, which can be called $(\mu_{_0},\sigma_{_0})$. This occurs because the Form I variance
does not explicitly depend on the mean and cutoff value, while Form II does, and hence by fixing those parameters,
we change the dependence of Form II on $\sigma$ and $\mu$ relative to Form I. 
It is shown in the Appendix (Section~\ref{subsec:dsigma_dmu})
that the slopes of $\sigma(\mu)$ under fixed variance for Forms I and II, which we can call $d\sigma\!_{_1}/d\mu$ and 
$d\sigma\!_{_2}/d\mu$, 
generally differ.
In addition, $d\sigma\!_{_{1}}/d\mu$ under constant variance is always constant (see Appendix, Section
 \ref{subsec:dsigma_dmu}), and $d\sigma\!_{_{2}}/d\mu$ is often close to constant. Thus, in general, the 
coordinates of the point of intersection, $(\mu_{_0},\sigma_{_0})$ can be well-approximated
by solving a pair of linear equations. One approach uses two points for each form and the other, one point and the slope. 
\subsection{Two-point method}
In the two-point approach, $\mu_{_0}$ may be found as:
\beq
\mu_{_0} = \frac{\mu_{_1}(\sigma_{_{2,2}}-\sigma_{_{1,2}}) - \mu_{_2}(\sigma_{_{2,1}}-\sigma_{_{1,1}})}
{\sigma_{_{2,2}} - \sigma_{_{2,1}} - (\sigma_{_{1,2}}-\sigma_{_{1,1}})}, 
\label{two-point-MI-formula}
\eeq
where $\sigma_{_{1,1}}$ and $\sigma_{_{1,2}}$ are determined recursively from Form I of the variance
at points $\mu_{_1}$ and $\mu_{_2}$. Likewise, $\sigma_{_{2,1}}$ and $\sigma_{_{2,2}}$ are determined
from Form II at the same $\mu$ values. 
The spread parameter value can then be found as 
$\sigma_{_0} = \sigma_{_{1,1}} +(\mu_{_0}-\mu_{_1})(\sigma_{_{1,2}}-\sigma_{_{1,1}})/(\mu_{_2}-\mu_{_1})$. 

\subsection{Point-slope method}
In a similar method, the slopes, $d\sigma_{_1}/d\mu$ and $d\sigma_{_2}/d\mu$, corresponding to Forms I and II of the variance,
can be calculated explicitly at a single point $\mu_{_1}$, as given in the Appendix (Section~\ref{subsec:dsigma_dmu}) and then
\beq
\mu_{_0} = \frac{ \sigma_{_{2,1}} - \sigma_{_{1,1}} + \mu_{_1} \left( \frac{d\sigma_{_1}}{d\mu}- \frac{d\sigma_{_2}}{d\mu} \right) 
} 
{\left(\frac{d\sigma_{_1}}{d\mu}- \frac{d\sigma_{_2}}{d\mu}\right)}.
\label{mu_zero_by_slopes}
\eeq
The spread parameter can then be determined as $\sigma_{_0} = \sigma_{_{1,1}} + (\mu_{_0}-\mu_{_1})\cdot d\sigma_{_1}/d\mu$.

Both of these linear approximations are more accurate the closer the trial $\mu$ 
values ($\mu_{_1}$ and $\mu_{_2}$) are to $\mu_{_0}$. They are also more accurate the lower the $r$ value,
because $\sigma(\mu)$ tends to be closer to linear there (see, e.g., 
Figure~\ref{fig:sigma_var_vs_mu_grid} and Figure~\ref{fig:approx_figure}, Panel a).
Successive approximation using these methods tends to rapidly converge to very accurate results.
Examples are given under Section~\ref{sec:example_apps} and in the Appendix (Section~\ref{log_normal_example}).
Although it is not explored here, the 
methods can be modified to use the mean instead of Form I of the variance, and possibly other moments as well.
\section{Skewness and kurtosis of UTGDs}
\label{sec:skewness_kurtosis}
The 3rd and 4th (raw) moments of $f(x;\mu,\sigma)_{|x > a}$ are
\bdm
M_3 = 3\sigma^2\mu + \mu^3 + 
\frac{\sqrt{2/\pi} \sigma e^{-\frac{(a-\mu)^2}{2\sigma^2}} (2\sigma^2 + \mu^2 +a \mu + a^2)}
{\xi(r)},
\edm
\bdm
\mathrm{and}~~~M_4 = \mu^4 + 3\sigma^4+ 6\sigma^2\mu^2 + \sqrt{2/\pi}\sigma e^{-\frac{(a-\mu)^2}{2\sigma^2}}
\left( \frac{5\sigma^2\mu + \mu^3 +a^2\mu + a^3 + a(3 \sigma^2 + \mu^2)}
{\xi(r)}\right).
\edm
The unnormalized skewness and kurtosis are the 3rd and 4th central moments:\\
$\mathrm{S_{un}} = \mathrm{E}[(x-M)^3] = M_3 - 3M M_2 + 2M^3$ ~and\\
$\mathrm{K_{un}} = \mathrm{E}[(x-M)^4] =  M_4 -4M M_3 + 6M^2 M_2 -3M^4$, respectively.

Substituting in the individual moments, then $u = r\sigma +a$ and finally the expression for $\sigma$ from Eq.~(\ref{sigma_M_r}),
the results are
\bdm
\mathrm{S_{un}} =  \left(\frac{M-a}{r B + 2}\right)^3  2 ((r^2-1) B^2 + 6 r B +8),
\edm
\bdm
\mathrm{and}~~~~\mathrm{K_{un}} =  \left(\frac{M-a}{ rB + 2}\right)^4 
\left(3 B^4 - 2r(r^2+3) B^3 - 8 (2r^2+1) B^2 -48 r B - 48\right),
\edm
where the shorthand of $B = \sqrt{2\pi}e^{r^2/2} \xi(r)$ is used. 
Since $rB = -2$ only as $r \rightarrow -\infty$, the expressions are well-defined for all $r \in \mathbb{R}$.
They can be shown to be equivalent to those given by Horrace (\cite{Horrace2015}) when $a=0$.
In their current forms, however, note that both $\mathrm{S_{un}}$ and $\mathrm{K_{un}}$, like the variance, can be
written as $\overline{M}{_k} = \left(\frac{M-a}{rB+2}\right)^k P(r,B)$,
where $\overline{M}{_k}$ is the $k$th central moment and $P(r,B)$ is a
polynomial in $r$ and $B$, and, although this is not proven here, it seems to hold for all (unnormalized) 
central moments of order 2 or greater.
The expressions for the 5th and 6th central moments are given in Supplementary Material 
(Section \ref{subsec:5th_6th_CM}).

Normalizing as $\mathrm{S} = \mathrm{S_{un}}/\mathrm{Var}(x)^{3/2}$ and $\mathrm{K} = 
\mathrm{K_{un}}/\mathrm{Var}(x)^2$, one obtains expressions for the conventional skewness and kurtosis:
\beq
\mathrm{S} = \frac{2 ((r^2-1)B^2 + 6rB + 8)}
{(B^2 -2rB -4)^{3/2}},
\eeq
\beq
\mathrm{and}~~~~~~~\mathrm{K}=3 - \frac{2 B^3 r (r^2-3) + 4(7 r^2-4)B^2 +96 (Br + 1)}{(B^2 - 2rB -4)^2}.
\label{conv_kurtosis}
\eeq

\noindent At $r = 0$, these expressions reduce to the known skewness and kurtosis of the half-normal distribution:
$\mathrm{S}_{|x \ge 0} = \sqrt{2}(4-\pi)/(\pi-2)^{3/2}$ and $\mathrm{K}_{|x \ge 0} = 3 + 8(\pi-3)/(\pi-2)^2.$

For large, negative $r$, the kurtosis is greater than the canonical value of 3, because the fractional term 
in Eq.~(\ref {conv_kurtosis}) is negative. 
Notably, there is a local minimum in the kurtosis of $\approx 2.75705603817495$ at $r\approx 
1.87412433954420$, which 
arises from the normalization of $\mathrm{K_{un}}$ by the variance.
K then approaches 3 from below as $r$ becomes large and positive and the truncation diminishes.
See Figure~\ref{fig:var_M_H_and_CM}, Panel (b).

In the limit of $r\rightarrow -\infty$, $\mathrm{K_{un}} \rightarrow 9 (M-a)^4$, $\mathrm{K} \rightarrow 9$, 
$\mathrm{S_{un}} \rightarrow 2 (M-a)^3$ and $\mathrm{S} \rightarrow 2$.

\section{Unilaterally truncated, scaled chi distributions}
\label{sec:RSnD}
Consider a distribution in $n$ independent (uncorrelated)
Gaussian random variables of the same form that is centered
at the origin (i.e., $\mu = 0$) and is symmetric with respect to all $n-1$ non-radial dimensions. 
The normalized probability distribution is
\beq
f(x_1,x_2,...x_n) d_{x_1} d_{x_2}...d_{x_n} = 
\left(\frac{1}{2\pi \sigma^2}\right)^{\!\!n/2} \!e^{-(x_1^2 + x_2^2 + ... + x_n^2)/2\sigma^2} d_{x_1} d_{x_2}...d_{x_n}.
\label{ndimens_func_in_x}
\eeq
Examples of two such distributions in two dimensions are shown in Figure~\ref{fig:2Dgaussian}.
These distributions are a subclass of conventional multivariate normal distributions (\cite{multivariateWiki})
whose components are mutually independent and have the same, variable spread parameter.

By analogy with the 1-D case, where the variate can be interpreted as the distance from
the distribution's center, here we use as a variate the quantity $R = \sqrt{x_1^2 + x_2^2 + ... + x_n^2}$,
which is the distance of the point $(x_1, x_2,...x_n)$ from the center of the $n$-D distribution (the origin).
The distribution in $R$ can also be interpreted as that of the magnitude of the sampling $n$-vector.
\begin{singlespace}
\begin{figure}[H]
\vspace{-2.0cm}
\centering
\includegraphics[width=4.5in]{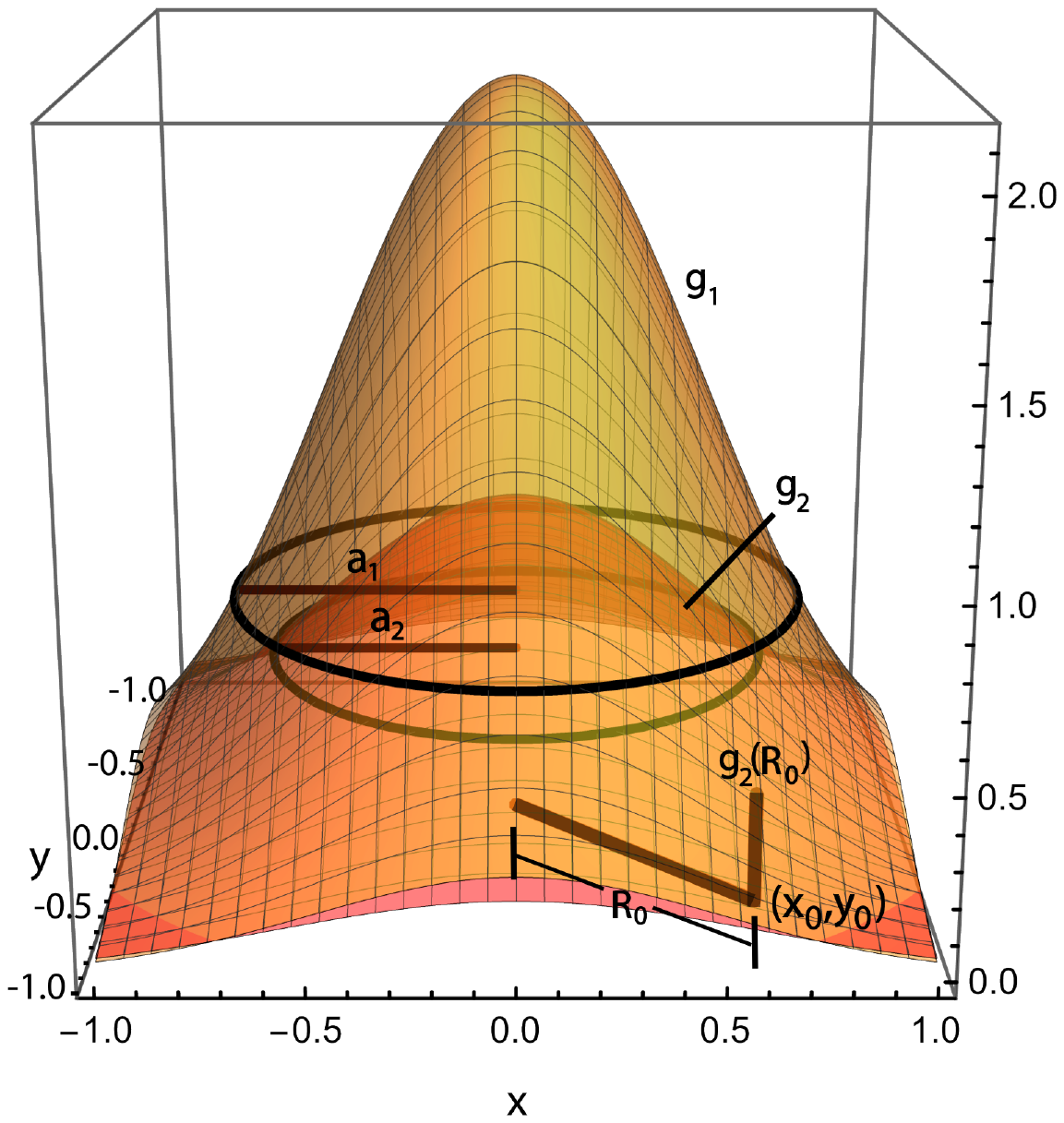}
\vspace{-1cm}
\caption{\small Two radially symmetric, 2-D Gaussian distributions. The plots are of
$f_i(x,y;\sigma_i) = g_i(R,\sigma_i)$, for $R = \sqrt{x^2 + y^2}$, as given
in Eqs.~(\ref{ndimens_func_in_x}) and (\ref{ndimens_g_of_R}),
where the variables $x$ and $y$ have the same univariate distribution.
For the surfaces labeled $g_{_1}$ and $g_{_2}$, the spread parameters are $\sigma_{_1} = 0.491$ and
$\sigma_{_2} = 0.572$.
The cutoffs in these 2-D cases are circles of radius $R=\mathrm{a}$ centered at $(0,0)$, and for the inner truncation
case, the surface under consideration is that below $g(\mathrm{a})$, or $g(R)_{|R > \mathrm{a}}$.
Here, the cutoffs are $a_{_1} =0.75 $ and $a_{_2}=0.65$.
Point $(x_{_0}, y_{_0})$ in the plot is a distance $R_{_0}>\mathrm{a}$ from the origin
and hence $g_{_2}(R_{_0})$ would be included in the calculation for the inner truncation case, but would
be excluded in that for the outer truncation ($R < \mathrm{a}$) case.
Because the mean and dimensionality are fixed (at 1 and 2, respectively), and there is only one remaining model degree of freedom,
$\sigma$ and $a$ have a fixed relationship.  The larger the $\sigma$ value
(i.e., the ``spread'' of the surface), the smaller the cutoff value ($a$) must be, and \it{vice versa}.
}
\label{fig:2Dgaussian}
\end{figure}
\end{singlespace}

Changing to (hyper)spherical coordinates and noting, again, the function's $(n-1)-$D symmetry
and that the area of the surface of an $n$-ball (an $n-1$-sphere) is $2\pi^{n/2}R^{n-1}/\Gamma(n/2)$,
where $\Gamma()$ is the Gamma function, the pdf is given by
\beq
g(R;\sigma,n) dR = 
\frac{2 e^{-R^2/2\sigma^2} R^{n-1}}{(\sqrt{2} \sigma)^n \Gamma(\frac{n}{2})}  dR.
\label{RSND_eq}
\eeq
Since the Gamma function has no zeroes, the expression is well-defined for all $n$.
These distributions are related to the chi distribution (\cite{chi-distributionWiki}),
 except that here, as in (\cite{Dyer1973chidist}), the spread parameter is not restricted to $\sigma=1$, so
we refer to them as \emph{scaled} chi distributions.
The $n=1$ case corresponds to the half-normal distribution,  the $n= \{2, 3,4\}$ cases to the Rayleigh (\cite{bidaoui2019wind}), 
Maxwell-Boltzmann (\cite{Peckham1992}), 
and 4-D Maxwell-Boltzmann (\cite{suman2020}) distributions, respectively, and so on for higher $n$. 

Though this formulation precludes covariances between the component random variables, it
permits a simplified analysis of the behavior of systems in higher, lower, fractional 
and even negative numbers of dimensions by effectively reducing it to a single dimension.
\subsection{Variance in the inner truncation case, $R \in [a,\infty)$}
\label{subsec:inner_truncation}
The integral of $g(R;\sigma,n) dR$ over $R \in [a,\infty)$ is 
$\Gamma\!\left(\frac{n}{2},\frac{r^2}{2}\right)/\Gamma\!(\frac{n}{2})$, where
$\Gamma(,)$ is the upper incomplete Gamma function. Hence, the pdf over this interval is:
\beq
g(R;\sigma,r,n)_{|R\ge a}\, dR = \frac{2 e^{-R^2/2\sigma^2} R^{n-1}}
{(\sqrt{2}\sigma)^{\rbox{2pt}{n}}\, \Gamma\!\left(\frac{n}{2},\frac{r^2}{2}\right)} \, dR,
\label{gaussian_pdf_of_R}
\eeq  
where $r = -a/\sigma$ and $a \ge 0$. This is a unilaterally truncated scaled chi distribution, or UTSCD.
Note that $\Gamma(,)>0$ when the function's second argument is positive, so this expression is well defined and
describes a positive probability distribution for all $\{n,r\} \in \mathbb{R}$.
Since $\{a, \sigma \} \ge 0$ here, $r$ is always negative (or zero).
In addition, note that, for a given dimensionality, $n$, there are only two parameters $\{\sigma,a\}$ here, 
instead of three,
because $\mu=0$. Hence, when $M$ is fixed as it is below, because there is only one remaining model degree 
of freedom, any of the remaining parameters, $\sigma, a,$ or $r$, entirely specify the system. 
Also, due to the nature of the upper incomplete Gamma function and its derivatives, 
some of the results in this section are inferred from numerical analysis rather than derived, and these 
conjectures are indicated as such. 
 
The distribution function $g(R;\sigma,r,n)$, which gives the probability at any particular point on the 
$n$-D surface as a function of its distance from the origin, is
\beq
g(R;\sigma,r,n)_{|R\ge a} = \left(\frac{1}{\sqrt{2\pi} \sigma}\right)^{\!\!n} \frac{ e^{-R^2/2\sigma^2} 
\Gamma\!\left(\frac{n}{2}\right)} {\Gamma\!\left(\frac{n}{2},\frac{r^2}{2}\right)}.
\label{ndimens_g_of_R}
\eeq
Since the Gamma function has poles at the negative integers and zero, the expression above 
is well-defined for $n \notin \{0,-2,-4,...\}$.  The $k$th raw moment over $R \in [a,\infty)$ is 
\beq
M_k(\sigma,r,n)_{|R\ge a} = \frac{\left(\sqrt{2} \sigma\right)^k \Gamma\!\left(\frac{n+k}{2},\frac{r^2}{2}\right)}
{\Gamma\!\left(\frac{n}{2},\frac{r^2}{2}\right)}.  
\label{kth_raw_moment}
\eeq
This gives Form I of the variance as
\beq
\mathrm{Var}(R;\sigma,r,n)_{|R \ge a} = \frac{2 \sigma^2\left(\Gamma\!\left(\frac{n}{2},\frac{r^2}{2}\right) 
\Gamma\!\left(\frac{n+2}{2},\frac{r^2}{2}\right) 
-\Gamma\!\left(\frac{n+1}{2},\frac{r^2}{2}\right)^2\right)} 
{\Gamma\!\left(\frac{n}{2},\frac{r^2}{2}\right)^2}.
\label{inner_unsub_var}
\eeq

\begin{table}[H]
\setstretch{1}
\centering
\tabcolsep=0pt
\begin{tabular*}{\columnwidth}{
  |@{\hspace{0.75em}\extracolsep{\fill}}lrrrr
  @{\hspace{0.2em}}||@{\hspace{1em}}lrrrr@{\hspace{0.75em}}|
}
\hline
\hline
\multicolumn{1}{|c}{$n$} & \multicolumn{1}{c}{$r$} & \multicolumn{1}{c}{$\sigma$} & \multicolumn{1}{c}{$a$}& 
\multicolumn{1}{c}{$\mathrm{Var}(R)_{|R \ge a}$} & \multicolumn{1}{c}{$n$} & \multicolumn{1}{c}{$r$} & 
\multicolumn{1}{c}{$\sigma$} & \multicolumn{1}{c}{$a$}&
\multicolumn{1}{c|}{$\mathrm{Var}(R)_{|R \ge a}$}  \\
\hline
 -6 & -4.0 & 0.23992 & 0.95968 & 0.00157 & 0.5 &  -4.0 &  0.23694 &  0.94778 &  0.00251 \\
 -6 & -2.0 & 0.45730 & 0.91461 & 0.00735 & 0.5 &  -2.0 &  0.42518 &  0.85036 &  0.01887 \\
 -6 & -1.0 & 0.87159 & 0.87159 & 0.01879 & 0.5 &  -1.0 &  0.68035 &  0.68035 &  0.07975 \\
 -6 & -0.5 & 1.69386 & 0.84693 & 0.03058 & 0.5 &  -0.5 &  0.96141 &  0.48071 &  0.20568 \\
 -6 & 0.0 & $\infty$ & 0 & 0.04167 &     0.5 &  0.0 &  2.09210 &  0. &  1.18843 \\
 -3 & -4.0 & 0.23871 & 0.95485 & 0.00194 & 1 & -4.0 & 0.23665 & 0.94661 & 0.00261  \\ 
 -3 & -2.0 & 0.44584 & 0.89169 & 0.01122 & 1 & -2.0 & 0.42137 & 0.84274 & 0.02029  \\
 -3 & -1.0 & 0.81167 & 0.81167 & 0.03789 & 1 & -1.0 & 0.65568 & 0.65568 & 0.08560  \\
 -3 & -0.5 & 1.48829 & 0.74414 & 0.08583 & 1 & -0.5 & 0.87636 & 0.43818 & 0.20620  \\
 -3 & 0.0 &  $\infty$ & 0. & 0.33333 &   1 & 0.0 & 1.25331 & 0. & 0.57080  \\
 -2 & -4.0 & 0.23825 & 0.95301 & 0.00208  & 2 & -4.0 & 0.23604 & 0.94414 & 0.00283  \\
 -2 & -2.0 & 0.44089 & 0.88179 & 0.01301  & 2 & -2.0 & 0.41299 & 0.82598 & 0.02336  \\
 -2 & -1.0 & 0.78204 & 0.78204 & 0.04808  & 2 & -1.0 & 0.60398 & 0.60398 & 0.09438  \\
 -2 & -0.5 & 1.37064 & 0.68532 & 0.12199  & 2 & -0.5 & 0.72655 & 0.36328 & 0.18772  \\
 -2 & 0.0 &   $\infty$  & 0. & $\infty$  & 2 & 0.0 & 0.79788 & 0. & 0.27324  \\
 -1 & -4.0 & 0.23776 & 0.95103 & 0.00224  & 3 & -4.0 & 0.23537 & 0.94148 & 0.00307  \\
 -1 & -2.0 & 0.43523 & 0.87046 & 0.01510  & 3 & -2.0 & 0.40356 & 0.80712 & 0.02667  \\
 -1 & -1.0 & 0.74616 & 0.74616 & 0.06022  & 3 & -1.0 & 0.55189 & 0.55189 & 0.09772  \\
 -1 & -0.5 & 1.22162 & 0.61081 & 0.16394  & 3 & -0.5 & 0.61172 & 0.30586 & 0.15658  \\
 -1 & 0.0 &  $\infty$ & 0. & $\infty$   &  3 & 0.0 & 0.62666 & 0. & 0.17810  \\
 0 & -4.0 & 0.23722 & 0.94890 & 0.00242  & 6 & -4.0 & 0.23303 & 0.93212 & 0.00393  \\
 0 & -2.0 & 0.42876 & 0.85751 & 0.01753  & 6 & -2.0 & 0.36927 & 0.73855 & 0.03636  \\
 0 & -1.0 & 0.70378 & 0.70378 & 0.07336  & 6 & -1.0 & 0.42160 & 0.42160 & 0.08013  \\
 0 & -0.5 & 1.04955 & 0.52478 & 0.19762  & 6 & -0.5 & 0.42544 & 0.21272 & 0.08628  \\
 0 & 0.0 &  $\infty$ & 0. & $\infty$   & 6 & 0.0 & 0.42553 & 0. & 0.08650  \\
\hline
\end{tabular*}
\caption{\small Variance of $R = (\sum_{i}^n {x_i}\smexp{^2})\smexp{^{1/2}}$,
as a function of dimensionality (distributional degrees of freedom) and parameter values for a scaled 
chi distribution.
For a given number of dimensions ($n$) and set of parameters $a$,
$\sigma$, and $r$  (cutoff, spread and their ratio, $r = -a/\sigma$, respectively),
the table indicates the resulting variance over the interval of $R \in [a,\infty)$, assuming
a mean of 1.}
\label{tab:ndim_variance}
\end{table}

Setting $k=1$ in Eq.~(\ref{kth_raw_moment}) gives $M$, the mean over the same interval, and
solving for $\sigma$
\beq
\sigma(M,r,n)_{|R\ge a} = \frac{M}{\sqrt{2}} \cdot 
\frac{\Gamma\!\left(\frac{n}{2},\frac{r^2}{2}\right)}
{\Gamma\!\left(\frac{n+1}{2},\frac{r^2}{2}\right)}, ~~~\mbox{and,~~hence,}
\label{sigma_multi}
\eeq
\bdm
M_k(M,r,n) =  \frac{M^k \Gamma\!\left(\frac{n}{2},\frac{r^2}{2}\right)^{k-1} \Gamma\!\left(\frac{n+k}{2},\frac{r^2}{2}\right)}
{\Gamma\!\left(\frac{n+1}{2},\frac{r^2}{2}\right)^k}.
\edm
Form II of the variance is then
\beq
\mathrm{Var}(R;M,r,n)_{|R\ge a} = M^2 \left(
\frac{\Gamma\!\left(\frac{n}{2},\frac{r^2}{2}\right) 
\Gamma\!\left(\frac{n+2}{2},\frac{r^2}{2}\right)}
{\Gamma\!\left(\frac{n+1}{2},\frac{r^2}{2}\right)^2} -1\right),
\label{multidim_var}
\eeq
where the fractional term is the second raw moment divided by $M^2$.

An alternative form of the variance is
\beq
\mathrm{Var}(R;\sigma,r,n) = \sigma^2 \left(\frac{2 e^{-r^2/2}|r|^n}{(\sqrt{2})^n
 \Gamma\!\left(\frac{n}{2},\frac{r^2}{2}\right)}+n\right)-M^2.
\label{alter_var}
\eeq

As is illustrated in Figure~\ref{fig:3D.and.2D_var.vs.r}, Panel (a),
and Table \ref{tab:ndim_variance} for the case of $M=1$, 
the variance appears to decrease monotonically with increasing $|r|=a/\sigma$ for any $n$, 
given a fixed mean, $M > a$. For fixed $n \in \mathbb{R}$ and $M>a \in \mathbb{R+}$, the maximal variance,
 $\mathrm{V_{max}}_{,r}(n)$ always appears to occur in the limit as $|r| \rightarrow 0$.  
Although this is not proven here, it was 
numerically verified for over 1000 values of $n \in \mathbb{R}$ in a range of $-100 \le n \le 512$.
Since the corresponding spread, which can be called the maximal variance spread for fixed $n$, or $\sigma_{\mathrm{vmx},r}(n)$,
is determined and finite for specified mean,
 the associated cutoff value, $a_{\mathrm{vmx},r} = -r \cdot \sigma_{\mathrm{vmx},r} = 0$ , as well. Hence, including
all of the inner (and outer) regions of the curve--i.e., the maximal extent of the distribution--appears to
maximize the variance.

Thus, for fixed $M$ and real $n > 0$, $\mathrm{V_{max}}_{,r}(R;M,n)= \mathrm{Var}(R;M,r=0,n)$ and
\beq
\mathrm{V_{max}}_{,r}(R;M,n)= M^2\left(\frac{n~\Gamma\!\left(\frac{n}{2}\right)^2}
{2 \Gamma\!\left(\frac{n+1}{2}\right)^2} - 1\right) =
 M^2\left(\frac{2^{2n-3} n \Gamma(\frac{n}{2})^4}{\pi \Gamma(n)^2} - 1\right).
\label{var_ndim_r=0}
\eeq

At $n=1$, the expressions\footnote{The second expression arises from the Gamma duplication formula.} 
correspond analytically to the variance of the half-normal 
distribution, i.e., $M^2(\pi-2)/2$, which is the maximal variance over all cutoff (or $r$) values and all
positive integral dimensionalities. For small $n \in \mathbb{R^+}$, $\mathrm{V_{max,r}}(n) \approx 2M^2/n \pi$,
 and for large $n$, 
$\mathrm{V_{max,r}}(n) \sim M^2\left(\frac{1}{2n} + \frac{1}{8n^2}\right)$ for any finite mean,
as can be determined from series expansions of Eq.~(\ref{var_ndim_r=0}). The plot of an approximating
function for $\mathrm{V_{max,r}}(n)$, $n > 0$, based on these limiting expressions is presented in 
Figure~\ref{fig:Vmax_approx_comb}, Panel (b) in the Appendix (Section~\ref{subsec:approx_max_var_inner}).
The corresponding $\sigma$ value is

\begin{figure}[h]
\hspace{-1.5cm}
\includegraphics[width=7.5in]{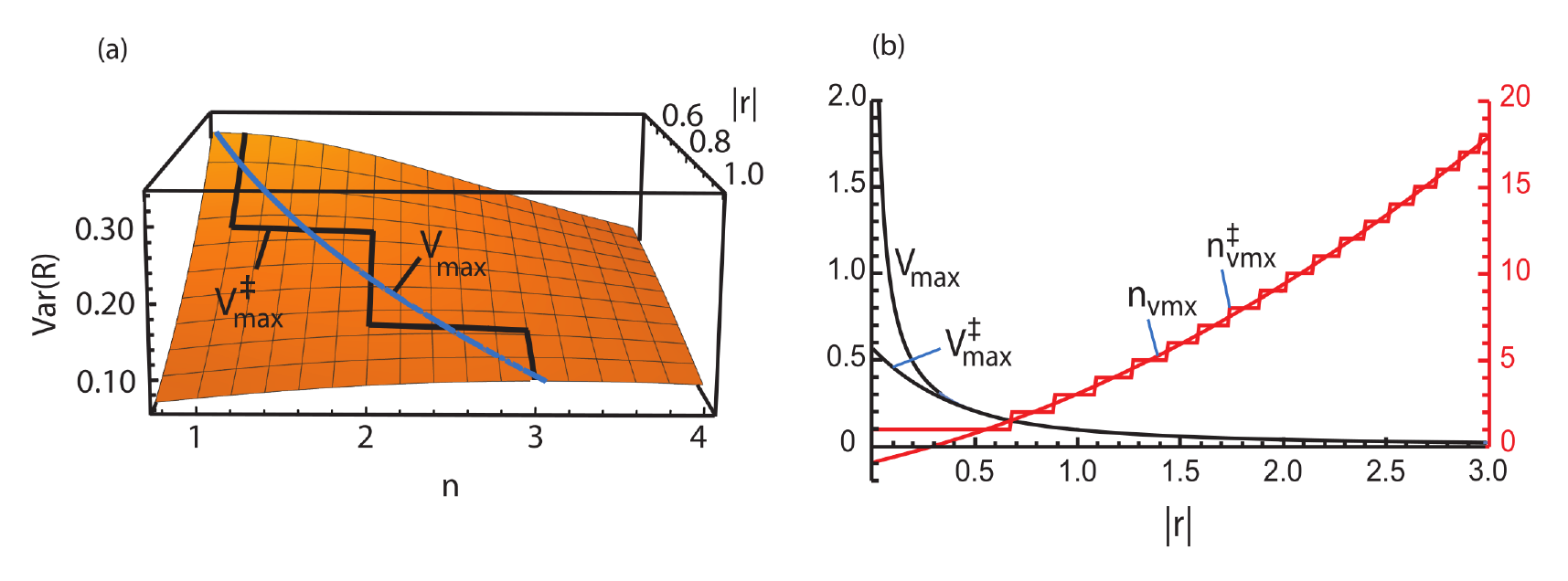}
\caption{\small Plots of the variance of a UTSCD.   
{\bf{Panel (a)}}: 3-D plot of the variance, $\mathrm{Var}(R)$, of a radially symmetric, 
multidimensional Gaussian distribution, $g(R)$ (from Eq. \ref{gaussian_pdf_of_R}),
vs. $|r|=a/\sigma$ and the number of dimensions ($n$), given a fixed mean of 1, over the interval $R \in [a,\infty)$, for $a>0$.
The curve labeled $\mathrm{V_{max}}$ indicates the maximal variance if the number of dimensions is allowed to vary over the real numbers,
whereas the $\mathrm{V^\ddagger_{max}}$ indicates that which results if the number of dimensions is constrained to the positive integers.
{\bf{Panel (b)}}: Plot of the maximal variance and corresponding number of dimensions for UTSCDs with
support of $R \in [a,\infty)$ and a fixed mean of 1, as a function of the $|r|=a/\sigma$ parameter.
$\mathrm{V_{max}}$ and $\mathrm{V^\ddagger_{max}}$ (left-hand scale, black) are the maximal variance as a function of 
$|r|$ for $n \in \mathbb{R}$ and $n \in \mathbb{Z}+$, respectively. The quantities $n_{\mathrm{vmx}}$ and 
$n^\ddagger_{\mathrm{vmx}}$ (right-hand scale, red)
are the number of dimensions that result in those maximal variances.
The functional dependences of $n_{\mathrm{vmx}}$ and $n^\ddagger_{\mathrm{vmx}}$ on $|r|$ do not change with varying mean.
}
\label{fig:3D.and.2D_var.vs.r}
\end{figure}

\bdm
\sigma_{\mathrm{vmx},r}(n) = \frac{M \Gamma(n/2)}{\sqrt{2} \Gamma((n+1)/2)} = \sqrt{\frac{\mathrm{V_{max},r}(n)+M^2}{n}}
\edm
for $n > 0 \in \mathbb{R}$.
For large (positive) $n$,  $\sigma_{\mathrm{vmx},r}(n) \sim M/\sqrt{n}$.

This formalism allows us to explore results for fractional as well as negative dimensionalities--essentially, any
real-valued $n$. 
Under fixed mean, for all tested values of $n \in [-2,0] \subset \mathbb{R}$ (using Eq. \ref{multidim_var}), 
$\mathrm{V_{max}}_{,r}(n) \rightarrow \infty$ as $|r| \rightarrow 0$;
and for $n \in (-\infty,-2) \subset \mathbb{R}$, $\mathrm{V_{max}}_{,r}(n) \rightarrow M^2/n(n+2)$ as 
$|r| \rightarrow 0$ 
(see Tables \ref{tab:ndim_variance} and \ref{tab:in_out_comp_trans}). 
These $\mathrm{V_{max}}_{,r}(n)$ expressions also correspond to the limits of the function 
$\mathrm{Var}(R;r,n,M)$ (Eq. \ref{multidim_var}) as $|r| \rightarrow 0$ for those particular ranges of $n$,
which are derived in the Appendix, Section~\ref{subsec:lim_var_chi_dist}.

Although $\mathrm{V_{max}}_{,r}(n)$ always seems to occur as $|r| \rightarrow 0$, by contrast, for fixed $r=-a/\sigma$,
both the maximal variance $\mathrm{V_{max}}_{,n}(r)$ and the number of dimensions in which 
that maximal variance occurs vary with $r$. We will refer to the latter quantity as the 
{\textit{maximal variance dimensionality}}
or the $\mathrm{vmx}$ dimensionality, and denote it as $n_{\mathrm{vmx}}$ for $n \in \mathbb{R}$ and
 $n^\ddagger_{\mathrm{vmx}}$ for $n \in \mathbb{Z+}$.

As illustrated in Figure~\ref{fig:3D.and.2D_var.vs.r}, Panel (b),
the maximal variance $\mathrm{V_{max}}_{,n}(r)$ decreases with $|r|$ monotonically and 
asymptotically as an exponential.  When $n$ is restricted to the positive integers, 
then $\mathrm{V^\ddagger_{max}}_{,n}(r)$ varies approximately as 
\beq
\mathrm{V^\ddagger_{max}}_{,n}(r) \approx \frac{d_{_1}}{\left(\frac{2 d_1}{\pi-2} + 1\right) e^{d_2 |r| ^{d_3}} -1},
\label{Vmax_approx}
\eeq
where the coefficients of the fitted function are given in Table \ref{tab:par_nmax_eq}. The resulting curve is 
nearly superposable with that in the figure (maximum deviation is $\approx 0.009$; 
see also Figure~\ref{fig:Vmax_approx_comb}, Panel (a), in the Appendix).

The $\mathrm{vmx}$ dimensionality increases with the cutoff, $a$, or $|r|$, as shown in Figure~\ref{fig:3D.and.2D_var.vs.r}, Panel (b)
and Figure~\ref{fig:var_nmax_vs_a_sigma}, Panel a.
If $n$ is allowed to vary across the real numbers, $n_{\mathrm{vmx}}$ increases monotonically, 
whereas if it is constrained to the positive integers, $n^\ddagger_{\mathrm{vmx}}$ 
rises in a step-wise fashion.
In either case, the larger the cutoff ($a$), or $|r|$ value, the larger the number of dimensions 
that is required to achieve the maximal variance. For large values of $|r|$, $n_{\mathrm{vmx}} \sim r^2$, 
by numerical examination.
As shown in Figure~\ref{fig:var_nmax_vs_a_sigma}, Panel (b),
the maximal variance increases approximately linearly with the spread parameter,
(even out to $\sigma > 100$, which is not shown), while $n_{\mathrm{vmx}}$ drops rapidly from $\approx 100$ at 
$\sigma = 0.1$ to $\approx 25$ at 
$\sigma = 0.2$. Note that $\sigma$ is given exactly (as a function of $M$) by 
substituting $r$ and $n$ into 
Eq.~(\ref{sigma_multi}), but it converges quickly to $M/\sqrt{n}$ for increasing $n$ (and, thus,  
$n_{\mathrm{vmx}} \approx M/\sigma_{\mathrm{vmx}{,n}}^2$).
\begin{singlespace}
\begin{figure}[H]
\vspace{-2.5cm}
\hspace{-1.5cm}
\includegraphics[width=7.5in]{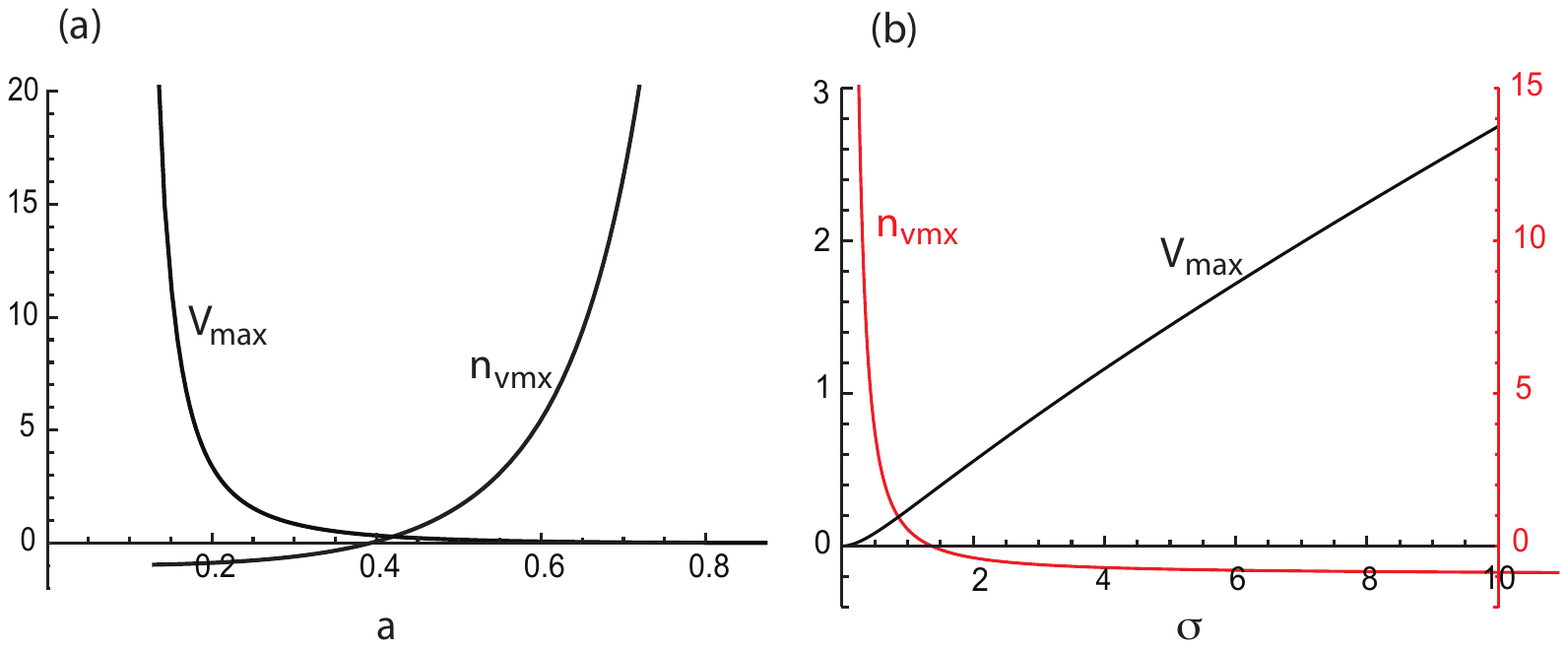}
\vspace{-3.5cm}
\caption{\small
Plots of the maximal variance and $\mathrm{vmx}$ dimensionality as a function of cutoff and spread for a UTSCD
with support over the interval $R \in [a,\infty)$ and a fixed mean of 1.
In {\bf{Panel (a)}}, the maximal variance, $\mathrm{V_{max}}$, and the number of dimensions in which it occurs, 
$n_{\mathrm{vmx}}$, are plotted as a function of the cutoff distance, $a$, with both plots using the same scale.
In {\bf{Panel (b)}}, the same quantities are plotted as a function
of the $\sigma$ parameter, (which is the optimal with $n_{\mathrm{vmx}}$ plotted in red and corresponding to the axis on the right. 
The number of dimensions, $n$, was allowed to vary over the real numbers.
As shown in a), as the cutoff value increases, there is less variance in the system, although increasing the
number of dimensions helps to increase it. Conversely, as illustrated in b), 
 as $\sigma$ increases, $\mathrm{V_{max}}$ also increases,
nearly linearly, while $n_{\mathrm{vmx}}$ drops rapidly, with $n_{\mathrm{vmx}}$ crossing 1 at 
$\sigma \approx 0.78837727124$
and approaching an asymptote at \mbox{-1} as $\sigma$ grows large. }
\label{fig:var_nmax_vs_a_sigma}
\end{figure}
\end{singlespace}

At the other extreme, for $|r| \lessapprox 0.6709572056$, if $n$ is restricted to the positive integers, 
then $n^\ddagger_{\mathrm{vmx}}=1$, and,
again, $\mathrm{V^\ddagger_{max}}_{,n} (r=0)$ corresponds to the variance of the half-normal distribution. 
If, on the other hand, $n \in \mathbb{R}$, then for $|r| \lessapprox 0.5541532357$,  
$n_{\mathrm{vmx}}$ is theoretically less than one, and for $|r| \lessapprox 0.2877193594$,
$n_{\mathrm{vmx}} < 0$. In fact, as shown in Figure~\ref{fig:3D.and.2D_var.vs.r}, Panel (b),
numerical calculations indicate that as $|r| \rightarrow 0$, $\mathrm{V_{max}}_{,n}(r)$ grows without bound, and
$n_{\mathrm{vmx}} \rightarrow \mbox{-1}$.
Hence, for UTSCDs with $n$ distributional dimensions or degrees of freedom, $n \in \mathbb{R}$,
the maximal variance appears to be theoretically unbounded as the number of dimensions
approaches -1. Although it appears true that $\mathrm{V_{max}} \rightarrow \infty$ as $|r| \rightarrow 0$ for 
all $n \in [-2,0] \subset \mathbb{R}$,
there seems to be a local maximum in $\mathrm{Var}(R)$ with respect to $n$ at $n=\mbox{-1}$ 
for any sufficiently small value of $r$.

While the interpretation of fractional and negative dimensions (\cite{MANDELBROT1990,Chhabra1991,Mikhailov1996,Lukaszyk2022}) 
in the current context is unclear,
it is notable that both $\mathrm{V_{max}}_{,n}(r)$ and $n_{\mathrm{vmx}}$ vary smoothly across 
all values of $|r|>0$, including 
the regions around $n_{\mathrm{vmx}} \approx 0$ and 1. See Appendix (Section~\ref{inner_outer_converge})
for further discussion of the variance of scaled chi distributions over inner- and outer-truncated 
intervals for negative dimensionalities.

Further, $n_{\mathrm{vmx}}(r)$ can be approximated as
\beq
n_{\mathrm{vmx}}(r) \approx r\smexp{^2} + c_{_1} |r|\smexp{^{3/2}} + c_{_2} |r| + c_{_3} |r|\smexp{^{1/2}} -1, 
\label{nvmx_approx}
\eeq
where the $c_i$ coefficients are given in Table \ref{tab:par_nmax_eq}.
This approximation is accurate to within $1.282\%$ for $n \in [1,10,000]$ and is exact (to within 
machine precision) for $n = \{1,2,10\}$, because those points were used as constraints in the 
parameterization. Since the relation gives a local (and global) maximum of the variance with respect 
to $n \in \mathbb{R}$ for a given $r$ value, the integral value of $n$ in which the maximal variance 
occurs (i.e., $n^\ddagger_{\mathrm{vmx}}$) can then be obtained by evaluating the variance at the 
two flanking integral $n$ values using Eq.~(\ref{multidim_var}).
\subsection{Variance in the outer truncation case, $R \in [0,a]$}
\label{subsec:outer_truncation}
Over the interval $R \in [0,a]$, the functional forms for the scaled chi distributions mirror those for inner truncation, 
except that the upper incomplete gamma function $\Gamma(,)$ is replaced by the lower incomplete 
gamma function, $\lgamma(,)$. 
The forms for the pdf $g(R;\sigma,n)_{|R \le a} dR$, the $k$th (raw) moment, 
Forms I and II of the variance and $\sigma(M,r,n)_{|R\le a}$ are given in the Appendix
 (Section~\ref{subsec:express_outer_trun}). 
Importantly, the $k$th moment seems to vary monotonically as a function of $\sigma$ between 
the bounds of $0$ and $\frac{n a^k}{n+k}$, which
are the limits of $M_k(\sigma,r,n)_{|R\le a}$ as $\sigma \rightarrow 0$ and $\sigma \rightarrow \infty$.
Hence the mean, $M < \frac{a n}{1+n}$, which implies the cutoff value ($a$) must always be greater than $M(1+n)/n$.
\begin{figure}[h]
\centering
\includegraphics[width=5in]{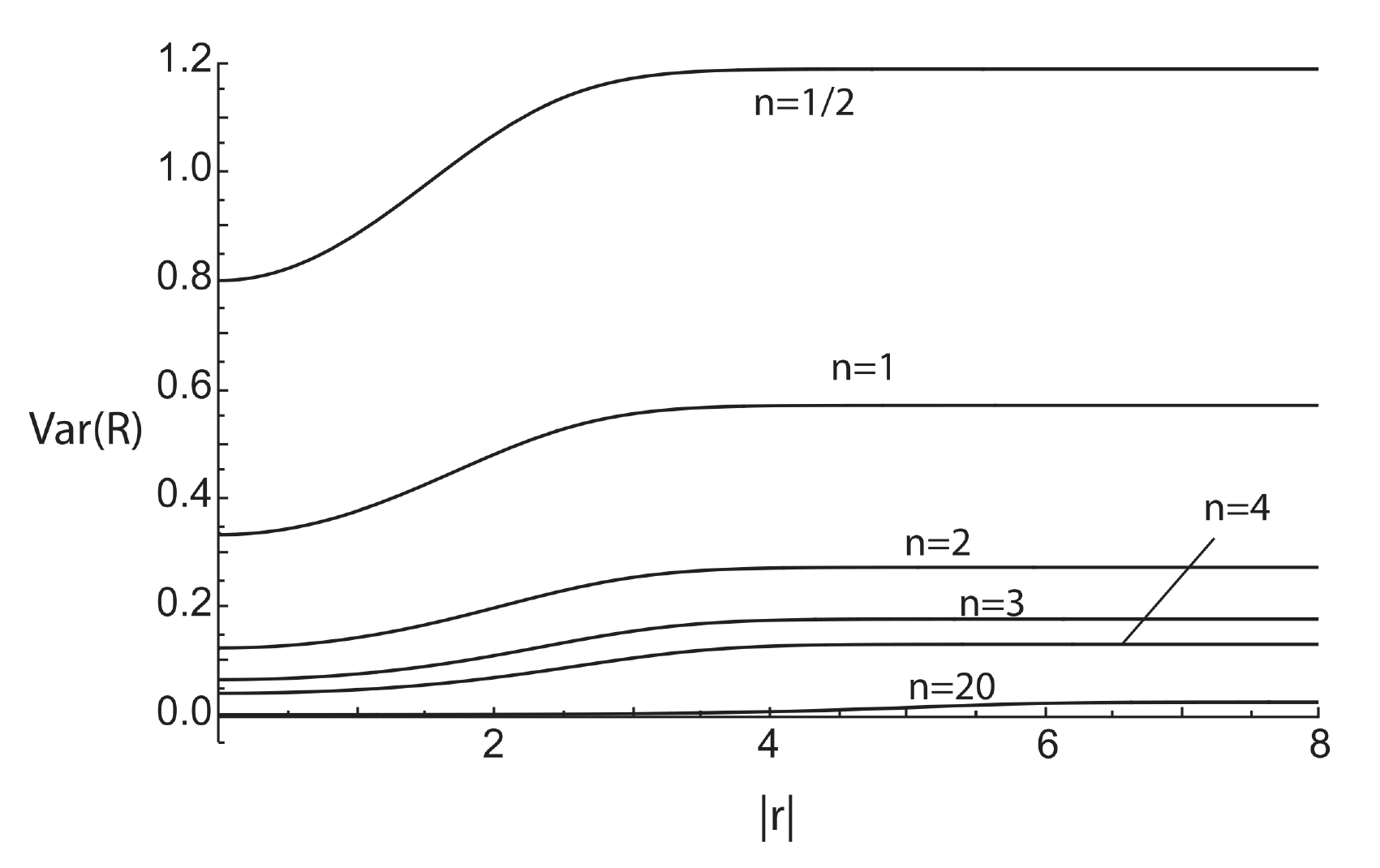}
\caption{\small
Variance of scaled chi distributions with $n$ dimensions or distributional degrees of freedom,
with truncation of their outer regions. The plots are of $\mathrm{Var}(R;M=1,n,r)$
over the interval $R \in [0,a]$,
for several different numbers of dimensions, $n$, as a function of the $|r|$ parameter.}
\label{fig:variance_inner}
\end{figure}
\noindent As may be expected, and as illustrated in Figure~\ref{fig:variance_inner},
for fixed $n>0$ and fixed $M < \frac{a n}{1+n}$, the maximal variance occurs when the outer interval is
included to its maximal extent--i.e., as $|r| \rightarrow \infty$--where the variance approaches the 
limiting value $\mathrm{V_{max}}_{,r}(R;M,n)$ as given in equation (\ref{var_ndim_r=0}). 
Since the corresponding value for the spread parameter, 
$\sigma_{\mathrm{vmx},r}(n)$ is, again, fixed
here (because $r,M$ and $n$ are specified), and finite for $n>0$ (its value is given by Eq. \ref{sigma_multi_inv}, 
Section~\ref{subsec:express_outer_trun}),
the cutoff, $a_{\mathrm{vmx},r}(n) = -r \, \sigma_{\mathrm{vmx},r}(n)$, also approaches $+\infty$ as $|r| \rightarrow \infty$.

For a fixed cutoff (or $r$), if we restrict the dimensions to the positive integers, the $\mathrm{vmx}$ dimensionality,
$n^\ddagger_{\mathrm{vmx}}=1$ for any $|r|>0$.  If $n \in \mathbb{R+}$, then $n_{\mathrm{vmx}} = 0$, for any $|r|>0$, 
as verified numerically for a range of $r$ values. Hence, in contrast to the case for inner truncation, for outer truncation,
the dimensionality resulting in the maximal variance appears to be fixed at the lowest possible (positive) value, for any $r$.

For a fixed, finite mean, $M$, the {\textit{minimum}} variance under outer truncation is not 0, but
$M^2/n(n+1)$ for $n > 0$. 
This occurs as $|r| \rightarrow 0$ and is also illustrated in Figure~\ref{fig:variance_inner}
for several instances of $n >0$, in which cases the cutoff value is the minimum possible: $a = M(n+1)/n$.  
\begin{table}[H]
\setstretch{1}
\renewcommand{\arraystretch}{0.4} 
\tabcolsep=0pt
\begin{tabular*}{1.0\columnwidth}{@{\extracolsep{\fill}}ccc@{}}
\hline
\hline
\noalign{\vskip 4pt}
\multicolumn{1}{c}{$\lim_{|r| \to }$} & \multicolumn{1}{c}{outer truncation} 
& \multicolumn{1}{c}{inner truncation}\\[4pt]
\hline
\hline
\noalign{\vskip 4pt}
\multicolumn{3}{c}{$\mathrm{\mathbf{Var}}\boldsymbol{(R;M,r,n)}$ } \\[4pt]
\hline
\rule{0pt}{20pt} 
0  &   $\frac{M^2}{n(n+2)}$ , $n > 0 $ &
$M^2\left(\frac{n~\Gamma\!\left(\frac{n}{2}\right)^2}
{2 \Gamma\!\left(\frac{n+1}{2}\right)^2} - 1\right),$ $n>0$ \\
\\
& & $\infty$, $n \in [-2,0]$\\
\\
& & $\frac{M^2}{n(n+2)}$, $n<-2$\\
\\
\hdashline
\\
$\infty$ &   $M^2\left(\frac{n~\Gamma\!\left(\frac{n}{2}\right)^2}
{2 \Gamma\!\left(\frac{n+1}{2}\right)^2} - 1\right), n >0$ &
0 \\
\\
\hline
\hline
\noalign{\vskip 4pt}
\multicolumn{3}{c}{$\boldsymbol{\sigma(M,r,n)}$}\\[4pt]
\hline
\\
0   &$\infty, n>0$ & $\frac{M \Gamma(\frac{n}{2})}{\sqrt{2} \Gamma(\frac{n+1}{2})}$, $n>0$     \\
\\
& & $\infty$, $n \le 0$\\
\\
\hdashline
\\
$\infty$  & $\frac{M \Gamma(\frac{n}{2})}{\sqrt{2} \Gamma(\frac{n+1}{2})}$, $n >0$  &  0   \\
\\
\hline
\hline
\noalign{\vskip 4pt}
\multicolumn{3}{c}{$\boldsymbol{a}$ {\bf{(cutoff)}}}  \\[4pt]
\hline
\\
0  &  $\frac{M(n+1)}{n},n>0$  &  0, $n \ge -1$ \\
\\
&  &  $\frac{M(n+1)}{n}$, $n < -1$  \\
\\
\hdashline
\\
$\infty$  & $\infty, n> 0 $   &   $M$  \\[2pt]
\hline
\end{tabular*}
\caption{\small Comparison between the behavior of outer- ($R \le a$) and inner- ($R \ge a$) truncated 
scaled chi distributions with
$n$ distributional degrees of freedom at limiting values of the parameter $|r|$, for $n \in \mathbb{R}$.
Shown are the values or expressions for the variance (from Form II, $\mathrm{Var}(R;M,r,n)$, top panel),
the spread parameter ($\sigma(M,r,n)$, middle panel), and the cutoff value
 ($a$, bottom panel) for the cases of $|r| \rightarrow 0$ and $r \rightarrow \infty$. The expressions were
derived analytically and confirmed numerically.
The derivations are given in the Appendix, Section~\ref{subsec:lim_var_chi_dist}.}
\label{tab:in_out_comp_trans}
\end{table}
Interestingly, this is the same expression
for the cutoff value that occurs as $|r| \rightarrow 0$ in the case of inner truncation for dimensionalities 
less than \mbox{-1}.
The relationships between the variances, spread parameters and cutoff values for the two types of truncation at the limiting
$r$ values of $0$ and $\infty$ are summarized in Table \ref{tab:in_out_comp_trans}. 
Note from the table that for $n>0$,
the results for outer-truncated distributions as $|r| \rightarrow \infty$ and those for inner-truncated distributions
 as $|r| \rightarrow 0$ converge, as expected. However, as discussed in the Appendix 
(Section~\ref{inner_outer_converge}),
for negative dimensionalities this is not the case. Instead, the two types of distributions unexpectedly converge 
as $|r| \rightarrow 0$ for both, 
if the domain of the moments for outer truncation is extended to $n<0$.

The forms of the raw moments and variances of the doubly truncated case (scaled chi distributions 
with both inner and outer truncation)
are analogous to those of the inner and outer truncation cases and are given in the 
Appendix (Section~\ref{doubly_truncated_chi}). 
\section{Example Applications}
\label{sec:example_apps}
\subsection{Modeling two-body collisions}
\label{subsec:examp_two_body_collisions}
The results in Section~\ref{sec:max_var_UTGD} regarding maximal variances can be useful
in determining theoretical bounds in models of interacting two-body or multi-body systems, for example.
The system-wide probability of hash value collisions with two input files
in cryptography or antibody cross-reactions with two epitopes in immunology, can be shown to be
$p_c \propto (\mathrm{Var}(R_j)+1)/N$ (\cite{Petrella2025}), 
where the $R_j$ are the normalized degeneracies of the antibodies or hash values, and $N$ is 
their total number. Since the mean in these cases is, by construction, 1, 
if the distributions are shown to conform to UTGDs,
then the theoretical maximum for the variance is $ \sup \mathrm{Var}(R_j)=1$, 
and the limiting probability of a collision is therefore
$\max(p_c)\propto 2/N$.  If they are shown to conform to ``conventional'' Gaussian
curves, i.e., UTGDs with $\mu \ge 0$, then the maximal
variance is proportional to $(\pi -2)/2$ and $\max(p_c) \propto \pi/2N$. The latter bounds also hold for
scaled chi distributions with a mean of 1. 
\subsection{Estimating parameters of a UTGD}
\label{subsec:examp_estimating_para}
 In this section, UTGDs are parameterized first using the the approximating functions and then 
using moment-intersecting methods, for two example datasets. Parameterization results for a log-normal distribution
are presented in the Appendix (Section~\ref{log_normal_example}).
The calculations here and in the following sections, as well as in the Appendix, are all done in double-precision.
\subsubsection{Using the approximating functions}
\hfill\\
\indent {\bf{Dataset A.}} 
A data distribution is thought to be a UTGD, with a left-sided 
truncation boundary of \mbox{-1.0}, a mean of 1.3 and variance of 3.0,
 but unknown $\mu$ and $\sigma$ parameters. The distribution is sketched in Figure~\ref{fig:Datasets_A_B_and_noisy}, Panel (a).

Here, approximating function 1 is used, with parameter set II (Table \ref{tab:c_and_d_param}).
From Eq.~(\ref{mu_exp_approx}), letting $M=1.3$, $a=-1$, and $\mathrm{Var}(x,M,U)_{|x \ge -1} = 3$, 
the value for $U$ can be determined by recursion
to be approximately $0.56710775$, so that $\mu = U(M-a)+a \approx -0.93790632$.
Then from Eq.~(\ref{sigma_exp_approx}), it is found that $\sigma \approx 2.85432733.$ 
Substituting these values for $\mu$ and $\sigma$ into 
 standard formulas for the mean and variance of a truncated Gaussian verifies that 
$M(x)_{|x \ge -1}= 1.30013513$ and $\mathrm{Var}(x)_{|x \ge -1} = 2.99938674$. 

The problem is now changed by  
scaling the distribution by a factor of 1000, with $a = -1000.0, M = 1300.0$ and 
$\mathrm{Var}(x)_{|x \ge -1000} = 3.0 \times 10^6$. Repeating the calculations, the 
$U$ value is unchanged and $\mu$ is a thousand-fold greater than the prior value, both 
to within machine precision, as expected. The $\sigma$ parameter value is within
$\approx 1 \times 10^{-7}\%$ of the expected 1000-fold scaled value and the resulting variance is 
$2.99938667 \times 10^6$. This illustrates the scale invariance of the formulation's fractional errors.

\begin{figure}[H]
\hspace{-1.0cm}
\includegraphics[width=7in]{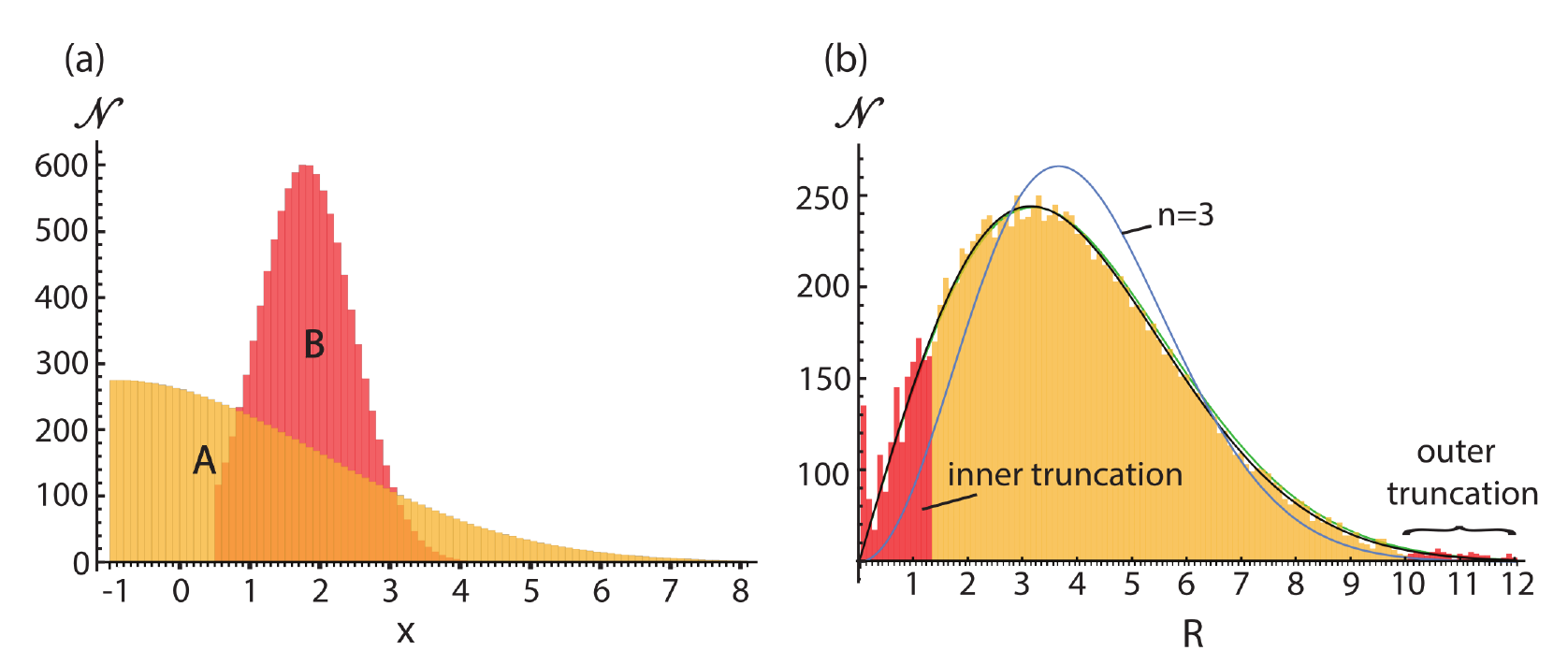}
\caption{\small
Datasets used for the examples.
{\bf{Panel (a)}}: Plots of two (unnormalized) Gaussian dataset distributions with left-sided truncation. Dataset A (yellow) is
truncated at $a= -1.0$ and dataset B (red) at $a = 0.5$.
$\mathscr{N}$ is the number of datapoints in each bin of the distribution.
{\bf{Panel (b)}}: Fit of three scaled chi distribution models to noisy Rayleigh distribution data.
$\mathscr{N}$-- number of data points
per bin, $R$--distance of the sampling vector from the origin.
The black curve is the model using inner truncation, $n=2$, and the data mean for parameterization.
The green curve is that using outer truncation, $n=2$, and Form I of the variance. The black
and green curves are nearly superposable. The blue 
curve is the model using $n=3$, inner truncation, and the data mean. The regions of inner
and outer truncation are indicated in red at the extremes of the plotted values of $R$.}
\label{fig:Datasets_A_B_and_noisy}
\end{figure}

{\bf{Dataset B.}}
A different Gaussian data set has a left-sided truncation boundary of $a=0.5$, a mean of 1.8,
a variance of 0.4, and unknown $\sigma$ and $\mu$ parameters (see Figure~\ref{fig:Datasets_A_B_and_noisy}, Panel (a)).
In this case, approximating function 2 (from Eq. \ref{approx_Lambert}) 
is used for $\tilde{\sigma}(U)$ in Eq.~(\ref{general_recur_mu}) to obtain by recursion that $U \approx
0.95724570$, which gives $\mu \approx 1.74441942$. 
From Eq.~(\ref{approx_Lambert}), then, $\sigma =  0.68720795$.  From the standard formulae, 
the resulting variance and mean are verified to be $\mathrm{Var}(x)_{|x \ge 0.5} = 0.40059879.$ and
$M(x)_{|x \ge 0.5}=1.79955814.$
Scaling the Gaussian by a factor of $1.0 \times 10^6$ and repeating the calculations, 
the $U$, $\mu$, and $\sigma$ parameters are all within machine precision of the values that 
are expected, given the unscaled results above, and the resulting mean and variance 
are $1.79955814 \times 10^6$ and $4.00598788 \times 10^{11}$.

\subsubsection{Using moment-intersecting approximations}
\hfill\\
\indent {\bf{Dataset A.}}
By inspection of the Dataset A distribution (Figure~\ref{fig:Datasets_A_B_and_noisy}, Panel a),
it appears the centering parameter should be somewhere 
in the interval $\mu_{_0} \in [-0.995,-0.7]$, taking a generous interval and noting that $\mu \ne a$. 
Using those endpoints as $\mu_{_1}$ and $\mu_{_2}$ and applying the
Newton-Raphson recursion method to determine $\sigma$ numerically from either Form I or 
II of the variance 
(Eqs. \ref{var_sig_r} and \ref{var_M_r}), given the known mean, variance and cutoff value, we obtain
$(\sigma_{_1,_1}, \sigma_{_1,_2}) = (2.87179324,2.78224205)$ from Form I and  $(\sigma_{_2,_1}, 
\sigma_{_2,_2}) =(0.24118430,14.47105787)$ 
from Form II. Substituting into the two-point formula (Eq. \ref{two-point-MI-formula}), we obtain  $\mu_{_0} = -0.94080581$
and the resulting variance is within about $\approx 0.01 \%$ of the actual value. The results obtained for both
datasets using the various parameterization methods are summarized in Table \ref{tab:param_method_results}. 
For illustrative purposes, the $\sigma(\mu)$ curves for the two forms of the variance around $\mu_{_0}$ are 
shown in Figure~\ref{fig:sigma_vs_mu_lines}.

\begin{figure}[H]
\hspace{-1.0cm}
\includegraphics[width=7in]{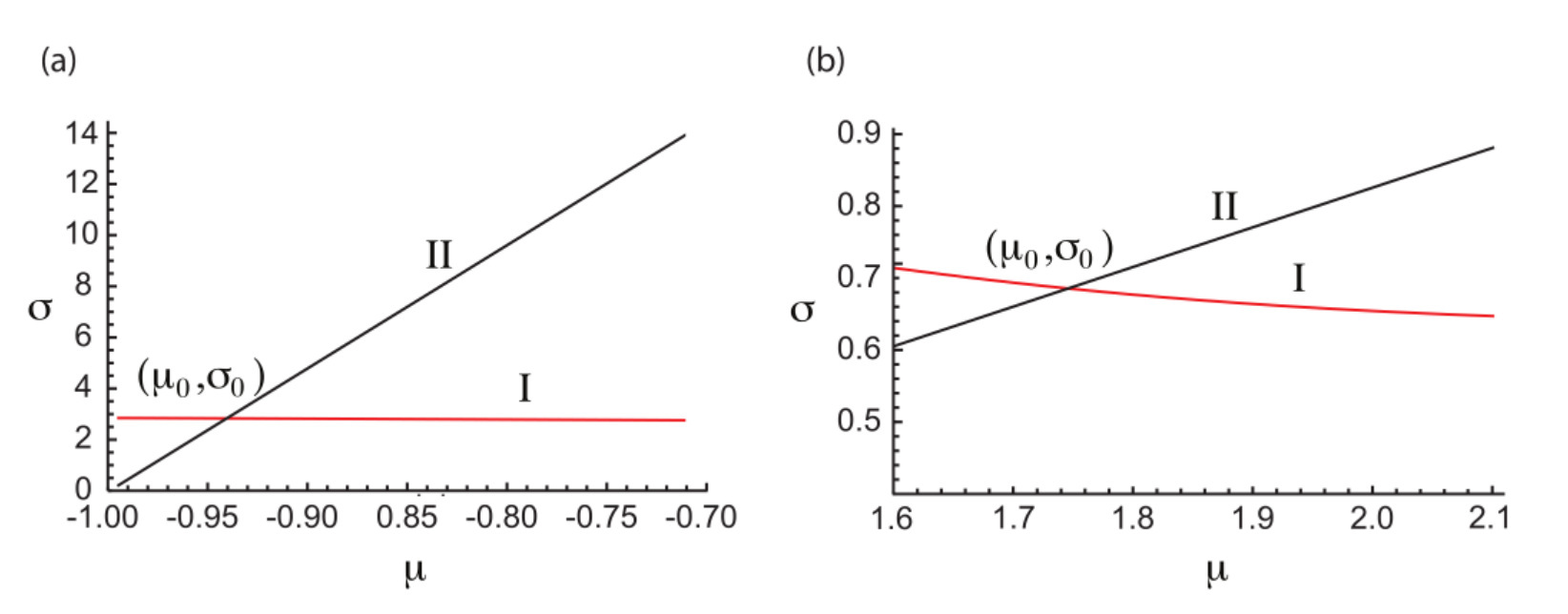}
\caption{\small
Plots of $\sigma(\mu)$ under constant mean and variance, derived from either the unsubstituted form (Form I)
 or substituted form (Form II) of the variance. {\bf{Panel (a)}} shows the results for Dataset A 
and {\bf{Panel (b)}} for Dataset B. The \symquote{I} and \symquote{II} labels
indicate the curves derived from Forms I and II of the variance, respectively, using the Newton-Raphson
method of successive approximation over the range of $\mu$ values.  The curves intersect at the point
labeled $(\mu_{_0},\sigma_{_0})$.
}
\label{fig:sigma_vs_mu_lines}
\end{figure}

Using the point-slope method, instead, we calculate the slopes $d\sigma_{_1}/d\mu = -0.30009822$ and 
$d\sigma_{_2}/d\mu = 48.23685957 $ explicitly using Eq.~(\ref{dsdu_form1}), (Appendix, 
Section~\ref{subsec:dsigma_dmu}) for Form I and $d\sigma_{_2}/d\mu = \sigma/(\mu-a)$ for Form II,
 at $\mu_1=-0.995$. Substituting then into Eq.~(\ref{mu_zero_by_slopes}), 
we have $\mu_{_0} = -0.94080194$ and the resulting variance is accurate to within $0.003 \%$.  If this method is
then repeated, taking this $\mu_{_0}$ value as the starting value ($\mu_{_1}$), the resulting variance is 
within $\approx 10^{-14}$ of actual (see table).

\indent {\bf{Dataset B.}}
The above calculations were repeated for Dataset B, using the broad range of $\mu_{_0} \in [1.6,2.2]$.
The results are given in Table \ref{tab:param_method_results}.  In general, they give about half as many
decimal places of accuracy as the results obtained for Dataset A, because at the higher $r$ value
 ($\approx 1.8$, compared to 0.02 for Dataset A) 
$\sigma(\mu)$ derived from Form I of the variance deviates further from linearity, as illustrated in
Figure~\ref{fig:sigma_vs_mu_lines}.
Nonetheless, if the point-slope method is applied recursively three times, 
the results are within machine precision of actual.
\subsection{Truncation threshold of a UTGD}
\label{subsec:examp_truncation_thresh}
A data distribution has a mean and variance of $50.2$ and
342.5 and is left-truncated at $a=20$.  Is it possible to model this 
data with a UTGD function such that the value at the 
\begin{table}[H]
\setstretch{1}
\tabcolsep=0pt
\begin{tabular*}{1.0\columnwidth}{@{\extracolsep{\fill}}lrrrr@{}}
\hline
\hline
\multicolumn{1}{c}{Method }  & \multicolumn{1}{c}{$\mu_{_0}$ } & \multicolumn{1}{c}{ $\sigma_{_0}$ }  
& \multicolumn{1}{c}{$\mathrm{Var}(x)$ }  & \multicolumn{1}{c}{mean }  \\
\hline
\multicolumn{5}{c}{Dataset A} \\
\hline
actual & & &                                  3.00000000  &  1.30000000 \\
approx f\tss{x}1  &    -0.93790632  &  2.85432733 & 2.99938674  &  1.30013513 \\
two-point     &   -0.94080581  &  2.85534187 & 2.99968031  &  1.29987745 \\
pt-slope  &        -0.94080194 &   2.85552850 & 3.00007245  &  1.30002777 \\
pt-slope $\times 2$   &   -0.94080265  &  2.85549402 & 3.00000000 &   1.30000000   \\
\hline
\multicolumn{5}{c}{Dataset B} \\
\hline
actual & & &                                  0.40000000  &  1.80000000 \\
appr f\tss{x}2 &    1.74441942  &  0.68720795 & 0.40059879 &   1.79955814 \\
two-point    &        1.74141742 &   0.68426408 & 0.39752228&  1.79596745 \\
pt-slope   &       1.74003554  &   0.68350240 & 0.39663777 & 1.79452484 \\
pt-slope $\times 2$ &     1.74527287  &   0.68638919 & 0.39999527  &1.79999231 \\
pt-slope $\times 3$ &     1.74528023  &   0.68639325 & 0.40000000  & 1.80000000 \\
\hline
\end{tabular*}
\caption{\small Parameterizations of truncated Gaussian distributions using different approximation methods.
$\boldsymbol{\{\mu_0, \sigma_0\}}$--the calculated location and spread parameters, 
{\bf{approx f\tss{x}1}}--approximating function 1, {\bf{appr f\tss{x}2}}--approximating function 2, {\bf{two-point}}--
two-point moment-intersecting method,
{\bf{pt-slope}}--point-slope moment-intersecting method, {\bf{pt-slope $\bm{\times 2}$}}--point-slope method
applied twice successively, taking the $\mu_{_0}$ value from the first trial as starting point in the second,
{\bf{pt-slope $\bm{\times 3}$}}-the point-slope method applied three times consecutively, {\bf{actual}}--the actual
values for the variance and mean of the distribution.
}
\label{tab:param_method_results}
\end{table}
\noindent truncation boundary ($f(20)$) is no more than $20\%$ of the maximum?
Letting $H=0.2$ in Eq.~(\ref{var_M_H}), the variance at the threshold is found to be $\approx 218.657$. 
Hence, a UTGD with these mean and threshold parameters cannot adequately model the variance 
of the data.
Using the same equation, it is found by recursion that a threshold value of $H \approx 0.59308951$ would be 
required. This corresponds (by Eq. \ref{eq_r_H}) to an $r$ value of $\approx 1.02216431 $, which means the 
relationship of the parameters in the model would be 
$\mu = 1.02216431 \sigma + 20$. From Eq.~(\ref{var_sig_r}), 
it is found that for this $r$ value and variance, $\sigma \approx 23.20169269$, and hence $\mu \approx 
43.71594225$. Substitution of these values into Eq.~(\ref{gaussian_mu_sig}) verifies that $f(x=a;\mu,\sigma) = 
0.59308951 f_{\!_{M}}$.
\subsection{Estimating parameters of scaled chi distributions} 
\subsubsection{Exact specification}
A dataset with mean 2.3 and variance 0.95 is thought to conform to an inner-truncated Rayleigh distribution, but
the spread ($\sigma$) parameter is unknown.  Letting $n=2$, $M=2.3$ in the Form II expression for the variance 
(Eq. \ref{multidim_var})
and setting $\mathrm{Var}(R) = 0.95$, the Newton-Raphson method can be used to find that $r=-0.53589710$.  
Substituting this value into Eq.~(\ref{sigma_multi}), $\sigma$ is found to be $\approx 1.65173960$. The
cutoff value is then $a = -r \sigma \approx 0.88516246$. Substituting these resulting $\sigma$ and $r$ values into the
 Form I expression for the variance (Eq. \ref{inner_unsub_var}) verifies it to be 0.95. Since we have imposed 3 
specifications ($n$, $M$, and $\mathrm{Var}(R)$),
and the formulation has 3 model degrees of freedom (including the dimensionality), this specification is exact and the 
parameter values result in mutually consistent moments.
\subsubsection{Over-specification}
We now illustrate an example where the model is over-specified and the parameter values cannot result in mutually
consistent moments.
Suppose a dataset contains a distribution thought to conform to either a Rayleigh ($n = 2$) or Maxwell-Boltzmann ($n=3$)
distribution (e.g., wind speed), but it is not clear which.  In addition, there is significant noise in the data.
To construct such a test, a noisy dataset was generated with the function  $\mathscr{N}(R) = |100 g(R)+ 10 \tau(R)|$, 
rounded to the nearest integer. Its distribution is shown in Figure~\ref{fig:Datasets_A_B_and_noisy}, Panel (b).
The $g(R)$ term is an exact Rayleigh ($n=2$) distribution with $\sigma = 3.14$ and support over $R \in [0,\infty]$, 
and $\tau$ is a noise term, $\tau(R)=\mathrm{rand}(-0.02,0.02)/R^{3/4}$, where $\mathrm{rand}(,)$ selects
a uniformly distributed random number within the range. Hence, the noise in the test data diminishes in magnitude 
with increasing $R$, and because $g(R) \rightarrow 0$ for large $R$, the fractional deviation of the resulting 
test data from the exact Rayleigh distribution is greatest at its 
highest and lowest $R$ values. 

The spread parameter ($\sigma$) for the models was determined using 
the variance in either Form I or II, or else the mean. In all three cases 
(Eqs. \ref{inner_unsub_var}, \ref{multidim_var}, and \ref{kth_raw_moment}, $k=1$), the substitution $r = -a/\sigma$ 
allows $\sigma$ to be found by recursion.  Here, the Newton-Raphson method was used. 
In the main set of calculations, $n$ was fixed at two, and the data was truncated either at low values 
of $R$ (``inner truncation''), high values (``outer truncation''), both (``double truncation''),
or neither (``none'').  A set of calculations with $n=3$ was also carried out for the inner-truncated data. 
\begin{table}[h]
\setstretch{1}
\tabcolsep=0pt
\begin{tabular*}{1.0\columnwidth}{@{\extracolsep{\fill}}lcrrr@{}}
\hline
\hline
\multicolumn{1}{l}{truncation} & \multicolumn{1}{c}{$\sigma$ source} &
\multicolumn{1}{c}{$\sigma$} & \multicolumn{1}{c}{ $\mathrm{RMSE}_{_1} \times 1000$} &
 \multicolumn{1}{c}{ $\mathrm{RMSE}_{_2} \times 1000$} \\
\hline
 $n$ = 2 & & & & \\
\hline
 inner & $\mathrm{Var}_{_I}$ &      3.14669418 &  4.077 & 10.665 \\
inner &  $\mathrm{Var}_{_{I\!I}}$ &    3.17732811  & 4.565 & 10.084 \\
 inner & mean &                    3.13829970   &  4.044 & 10.859 \\
\hline
 outer & $\mathrm{Var}_{_I}$ &       3.18696618 &  9.954 & 2.150 \\
 outer & $\mathrm{Var}_{_{I\!I}}$ &     5.53296118   &  640.660 & 648.935 \\
 outer & mean                &    3.09455395  & 9.845 & 2.595\\
\hline
 double&  $\mathrm{Var}_{_I}$ &     3.13040972 & 4.429 & 13.127 \\
 double&  $\mathrm{Var}_{_{I\!I}}$ &     3.10457826 & 4.767 & 13.856 \\
 double&  mean                &    3.13200235 & 4.424 & 13.087 \\
\hline
 none  & $\mathrm{Var}_{_I}$  &  3.18419313 & 9.021 & 1.998 \\
none & mean                   &  3.10374045 & 8.910  & 1.671 \\
\hline
 $n$ = 3 & & & & \\
\hline
 inner & $\mathrm{Var}_{_I}$   &   5.11754942  &   98.883 & 71.550 \\
 inner &  $\mathrm{Var}_{_{I\!I}}$  & 238.81739807  & 122.548 & 112.005 \\
 inner & mean                  &    2.57848811 &   21.349 & 39.797 \\
\hline
\end{tabular*}
\caption{\small Correspondence of over-specified scaled chi distribution models to noisy data.
The test data was generated over the interval $R \in [0.1,12.0]$ as described in the text.
{\bf{truncation}}--inner: the calculation was done over the interval $R \in [1.35,\infty]$,
outer: $[0.0,10.05]$, double: $[1.35,10.05]$,
none: $[0,\infty]$;
$\bm{\sigma}$ {\bf{source}}--the expression used to derive the spread parameter: $\mathrm{Var}_{_I}$ Form I of 
the variances of the test
data, (Eqs. \ref{inner_unsub_var}, \ref{outer_unsub_var}, and \ref{gen_unsub_var}),
$\mathrm{Var}_{_{I\!I}}$--Form II (Eqs. \ref{multidim_var}, \ref{multidim_var_inv},
and \ref{multidim_var_gen}), mean--mean of the test data (Eqs. \ref{kth_raw_moment}, \ref{kth_raw_moment_outer},
and \ref{kth_raw_moment_gen}, with $k=1$).
$\bm{\sigma}$--spread parameter, {\bf{RMSE}}${\bm{_{_1} \times 1000}}$--
the root-mean-squared error in the pdf of the model curve relative to the test data,
 normalized to a total area of 1, $\times 1000$, {\bf{RMSE}}${\bm{_{_2} \times 1000}}$--
the RMSE of the model curve pdf relative to the original Rayleigh distribution pdf (without noise term) $\times 1000.$
For the top portion of the table ($n$=2), the models used two distributional degrees of freedom;
for the bottom portion ($n=3$), they used three.
}
\label{tab:chi_accuracy}
\end{table}

Note that here we are imposing four conditions ($n$, $M$, $\mathrm{Var}(R)$ and $a$), 
so the system is over-specified. Thus, for a given $n$, unless the cutoff value, mean and 
variance all happen to be consistent within the context of the model--which is unlikely--
the results will depend on how the parameterization is done. 
The divergence of the parameterization results under the use of different moment functionals is a measure of the mutual 
consistency, or congruence, of the parameter values under the model.

As shown in Table \ref{tab:chi_accuracy} and illustrated in Figure~\ref{fig:Datasets_A_B_and_noisy}, Panel (b),
the results of the modeling differ not only under different truncation schemes, but also within the same scheme,
 depending on which moment functional was used to determine the $\sigma$ parameter 
(labeled ``$\sigma$ source''). 
The schemes resulting in less divergence among the $\sigma$ parameters determined using different
moment functionals also tend to result in a lower RMS error of the model distribution from the test data 
($\mathrm{RMSE}_1$ in the table).

It is clear from all the results that the $n=2$ models generally produce
a much better fit to the test data than the $n=3$ models. The exception is the $n=2$, Form II-variance outer-truncation model, 
because when the higher-valued portion of the data is removed, the variance of the test data
($\approx 4.04164$) exceeds the maximal variance defined by \\
Eq.~(\ref{multidim_var_inv})\footnote{Eqs. \ref{kth_raw_moment_outer} -- \ref{multidim_var_inv} are found in the 
Appendix, Section~\ref{subsec:express_outer_trun}.}($\approx 4.03353$), 
given the mean, and this leads to anomalous results for $\sigma$ under Form II of the variance.
A similar effect is observed with the inner-truncated $n=3$ case (maximal variance $\approx 3.18385$).
By contrast, the spread parameter as determined with Form I of the variance 
(Eqs. \ref{inner_unsub_var} and \ref{outer_unsub_var}) 
is less sensitive to this qualitative inconsistency between the mean and variance.

The closest fit to the data (RMSE 4.044 $\times 10^{-3}$) is obtained with $n=2$,  inner truncation and use of the mean 
(Eq. \ref{kth_raw_moment}, $k$=1) for determining $\sigma$. 
Notably, in all truncation schemes, the best fit to this dataset is obtained using the mean rather than the variance to 
determine $\sigma$, 
because the noise at the periphery of the data distribution affects the variance to a greater extent than the mean. 
The model with the best fit to the underlying Rayleigh distribution (i.e., without the noise term) is that using no truncation 
and either Form I of the variance or the mean to determine $\sigma$ (RMSE 1.998 $\times 10^{-3}$ and 1.671 $\times 10^{-3}$, 
respectively).

An example of scaled chi-distribution modeling that involves higher numbers of distributional
degrees of freedom, or dimensions, is given in the Appendix (Section~\ref{n-dim_velocities}).

\renewcommand{\thetable}{A\arabic{table}} 
\setcounter{table}{0} %
\renewcommand{\thefigure}{A\arabic{figure}}
\setcounter{figure}{0}
\clearpage
\section{Appendix}
\label{sec:appendix}

\subsection{Proof of supporting Lemmas for Theorem \ref{theorem:main}}
\label{subsec:lemma_proofs}
\hfill\\
\indent {\bf{Lemma} \ref{lemma:limit_var}}.
\bdm
\lim_{r \to -\infty} \mathrm{Var}(x;M,r)_{|x \ge a} = (M-a)^2.
\edm
\begin{proof}
For the variance as defined in Eq.~(\ref{var_M_r}), taking the limits of the numerator and denominator directly
as $r\rightarrow -\infty$ or using L'H\^{o}pital's rule
result in the indeterminate form 0/0. However, the numerator and denominator terms can be expanded in
Laurent Series as follows.
From Eq.~(\ref{var_M_r}),
\beq
\mathrm{Var}(x;M,r)_{|x \ge a} = \frac{(M-a)^2\left(T_1 - T_2 -2\right)}{T_3},
\label{var_reduced}
\eeq
where  $T_1 = \pi e^{r^2}\xi(r)^2, ~~~ T_2 = \sqrt{2\pi} r e^{r^2/2}\xi(r)$, and 
$T_3=\pi \left(\sqrt{2/\pi} + r e^{r^2/2}\xi(r)\right)^2.$

To find the overall limit of $\mathrm{Var}(x;M,r)_{|x \ge a}$ by expanding each term around
$r \rightarrow -\infty$, the expansions have to
be done to at least 4th order, because the 1st- through 3rd-order terms in the
expansion of $T_3$ are zero.
To fourth order, \\
\indent ~~~~~~~~~~~~$T_1 \sim -\frac{4}{r^4} + \frac{2}{r^2},
 ~~T_2 \sim - \frac{6}{r^4} + \frac{2}{r^2} -2,~ \mbox{and}~~ T_3 \sim \frac{2}{r^4}.$

In the limit of $r \rightarrow \infty$, higher-order terms (5th order and above) do not contribute.
(Longer expansions are provided in Supplementary Material, Section~\ref{subsec:Laurent_ser_expan}.)
Substituting into Eq.~(\ref{var_reduced}), then,
\bdm
\mathrm{Var}(x;M,r)_{|x \ge a} \sim
\frac{(M-a)^2\left(-\frac{4}{r^4} + \frac{2}{r^2} -(- \frac{6}{r^4} + \frac{2}{r^2}-2) -2\right)}
{\frac{2} {r^4}}
= \frac{(M-a)^2\left(\frac{2}{r^4}\right)}{\frac{2}{r^4}}=(M-a)^2.
\edm
The result follows.
\end{proof}
{\bf{Lemma} \ref{lemma:limit_var_pos_inf}}.
\bdm
\lim_{r \to +\infty} \mathrm{Var}(x;M,r)_{|x \ge a} = 0.
\edm
\begin{proof}
This can be determined by inspection. Since $\lim_{r \to +\infty}\xi(r) = \\
\lim_{r \to +\infty}(\erf(r/\sqrt{2}+1) = 2$, $T_1$ from Lemma~\ref{lemma:limit_var}
approaches $4 \pi e^{r^2}$ and since
it is the dominant term at large $r$, the numerator in Eq.~(\ref{var_reduced}) approaches $4\pi e^{r^2} (M-a)^2$.
Similarly, the denominator approaches $\pi(2r e^{r^2/2})^2 = 4 \pi r^2 e^{r^2}$. Hence,
\bdm
\mathrm{Var}(x;M,r)_{|x \ge a} \sim
\frac{ 4\pi e^{r^2} (M-a)^2}{4 \pi r^2 e^{r^2}} = \left(\frac{M-a}{r}\right)^2.
\edm
The result follows.

Alternatively, a fourth-order Laurent Series expansion of $T_1$, $T_2$ and $T_3$ for $r \rightarrow +\infty$,
(see Supplementary Material, Section~\ref{subsec:Laurent_ser_expan})
using Eq.~(\ref{var_reduced}) from Lemma~\ref{lemma:limit_var} results in
\bdm
\mathrm{Var}(x;M,r)_{|x \ge a}  \sim
(M-a)^2 \cdot \frac{2\pi r^5 e^{r^2} - \sqrt{2\pi} (r^6 + 2r^4 -2r^2 + 6) e^{r^2/2} -2 r^3 + 5r}
{2\pi r^7 e^{r^2} + \sqrt{8\pi}r^2 (r^2-3) e^{r^2/2} +r}
\edm
\bdm
\sim (M-a)^2 \cdot \frac{ 2\pi r^5 e^{r^2}}
{2\pi r^7 e^{r^2}} = \left(\frac{M-a}{r}\right)^2.
\edm
\end{proof}
{\bf{Lemma} \ref{lemma:var_different}}.
The variance $\mathrm{Var}(x;M,r)_{|x \ge a}$ is differentiable with respect to $r$ at least
at all finite values of $r$.
\begin{proof}
From Eq.~(\ref{var_M_r}), differentiating the variance with respect to $r$,
\beq
\frac{1}{(M-a)^2}\cdot \frac{\partial (\mathrm{Var}(r))}{\partial r} =
\frac{\sqrt{8/\pi} r + e^{r^2/2}\xi(r)\left(\sqrt{2\pi} r (r^2+3) e^{r^2/2}\xi(r)
- 2\pi e^{r^2}\xi(r)^2 +4 r^2 +6\right)}
{ \pi\left(re^{r^2/2}\xi(r) + \sqrt{2/\pi}\right)^3},
\label{eq_dvar}
\eeq
where the shorthand $\mathrm{Var}(r) = \mathrm{Var}(x;M,r)_{|x \ge a}$ is being used.

As the numerator and denominator are both polynomials in $r$, $e^{r^2}, e^{r^2/2}$
and $\xi(r)$, and each of those functions is finite for $r \in \mathbb{R}$, the numerator and denominator
are also finite.
The derivative, $\partial (\mathrm{Var}(r))/\partial r$, therefore exists at least for all
finite $r$ such that $\pi\left(re^{r^2/2}\xi(r) + \sqrt{2/\pi}\right)^3 \ne 0$. The term
$re^{r^2/2}\xi(r)$ in this expression approaches $-\sqrt{2/\pi}$ only in the limit as $r \rightarrow -\infty$. Elsewhere,
the denominator $\ne 0$ . Thus, the derivative of the variance exists at least for all finite values of $r$.
\end{proof}
{\bf{Lemma} \ref{lemma:dsigdr}}.~~
$\partial \sigma(M,r)/\partial r = -\mathrm{Var}(x;M,r)_{|x \ge a}/(M-a)$.
\begin{proof}
Let $G(r) = re^{r^2/2} \xi(r)$.  From Eq.~(\ref{sigma_M_r})
\bdm
\frac{1}{M-a} \frac{\partial \sigma(M,r)_{|x \ge a}}{\partial r}  =
\edm
\beq
\frac{d}{dr} \left(\frac{G(r)}{r \left(G(r) + \sqrt{2/\pi} )\right)}\right) =
 \frac{ \sqrt{2/\pi} \left(r G^{\prime}(r) - G(r)\right) - G(r)^2}
{r^2 (G(r) +  \sqrt{2/\pi} )^2},
\label{sigma_G_r_part}
\eeq
where $G^{\prime}(r) = d (G(r))/dr$.
\\

Note that ~ $G^{\prime}(r) = \sqrt{2/\pi} r + (1+r^2)G(r)/r$,
~so that
\bdm
\frac{1}{M-a} \frac{\partial \sigma(M,r)_{|x \ge a}}{\partial r}  =
 \frac{-\pi G(r)^2/r^2 + \sqrt{2\pi} G(r) +2}
{\pi \left( \sqrt{2/\pi} + G(r)\right)^2},
\edm
which, from Eq.~(\ref{var_M_r}), equals $-\mathrm{Var}(x;M,r)_{|x \ge a}/(M-a)^2$.
\end{proof}
{\bf{Lemma} \ref{lemma:var_M_r_sig}}. ~~~
$\mathrm{Var}(x;M,r,\sigma)_{|x \ge a} = \sigma^2 + (M-a) r\sigma - (M-a)^2.$
\begin{proof}

Let
\bdm
T(r) = \frac{\sqrt{2/\pi}}  {e^{r^2/2} \xi(r)}.
\edm
From Eq.~(\ref{mean_sig_r}), the mean is $M(x;\sigma,r)_{|x \ge a} = \sigma(r + T(r)) + a$,
and solving for $T$,
\beq
T(r) = \frac{M(x;\sigma,r)_{|x \ge a} -a}{\sigma}  -r.
\label{H_eq}
\eeq
From Eq.~(\ref{var_sig_r}) the unsubstituted variance can be expressed as
\beq
\mathrm{Var}(x;T,\sigma)_{|x \ge a} = \sigma^2\left(1 - T(r)r - T(r)^2 \right).
\eeq
From Eq.~(\ref{H_eq}) above, then
\beq
\mathrm{Var}(x;r,\sigma)_{|x \ge a} = \sigma^2\left(1 - \frac{(M-a -r \sigma)r}{\sigma} -
\frac{(M-a -r \sigma)^2}{\sigma^2}\right),
\eeq
where $M$ is shorthand for the mean.
Rearranging, the result follows.
\end{proof}

\subsection{Additional modeling examples}
\label{additional_model_examp}
\subsubsection{Truncated Gaussian models of log-normal data}
\label{log_normal_example}
Figure~\ref{fig:US_census_data} shows the distribution of $\log_{e}$-transformed U.S. income data,
which we call dataset $X$, from a 2022 Census Bureau survey of over 1M individuals consisting of 766678 non-zero, 
non-blank responses (\cite{USCensusIncome2022}). The distribution contains a quasi-regular train of 
impulsive spikes, which arises from respondents' proclivity for round numbers\footnote{For example, 
the bin containing $\$100,000$ ($y=11.525$) has a frequency that is more than $3\times$ that of the 
two adjacent bins.}, but overall it corresponds approximately to a Gaussian, particularly above income levels of $\approx \$3000$ (log-income $\approx 8$). Hence, the original data, which we call $Y = e^X$ is approximately 
log-normally distributed over support $y > \$3,000$. 

We would like to model the log-transformed yearly incomes above $\ln (\$15000/\$1) \approx 9.62$
as\footnote{Effectively a cutoff of $x= \$14951$ or $y \approx 9.6125$ due to binning effects.}
a truncated normal distribution (UTGD) with the same
mean and variance as the data ($X$). This can be done using the methods described
previously (See Sections \ref{sec:approx_sigma_of_mu} and \ref{subsec:examp_estimating_para}).
The results are shown in Table \ref{tab:us_census}, Panel (a), columns 1-5.
As in prior examples, the mean and variance of the modeled distributions closely match those of the target data.

Further, we can back-calculate an estimate for the mean and variance of the original data ($Y$)
using the $\sigma$ and $\mu$ parameters we have just generated.
Note that, for a UTGD, $f_X(x)$, with support over $x > a$ and parameters $\sigma$ and $\mu$
that describes the distribution of log$_e$-transformed data $X$,
the mean and variance of the distribution, $f_Y(y)$, of the original (untransformed) data, $Y$, over support $y > e^a$,
can be expressed in terms of those same $\mu$ and $\sigma$ parameters as\footnote{Since $\xi(r)$ is zero only as $r \rightarrow -\infty$,
the expressions are well-defined for all real $\{r,\sigma\}$.}:

\beq
M(y)_{|y>e^a} = \frac{e^{\frac{\sigma^2}{2} + \mu} ~\xi(r+\sigma)}
{\xi(r)} ~~~~~~~~  \mathrm{and}~~
\eeq
\beq
\mathrm{Var}(y)_{|y>e^a} = \frac{e^{\sigma^2+2\mu} \left(e^{\sigma^2}\xi(r)\cdot \xi(r+2\sigma) -\xi(r+\sigma)^2\right)} 
{\xi(r)^2} ~~~\mathrm{(Form~I)}~=
\label{var_orig_formI}
\eeq
\begin{figure}[H]
\includegraphics[width=6in]{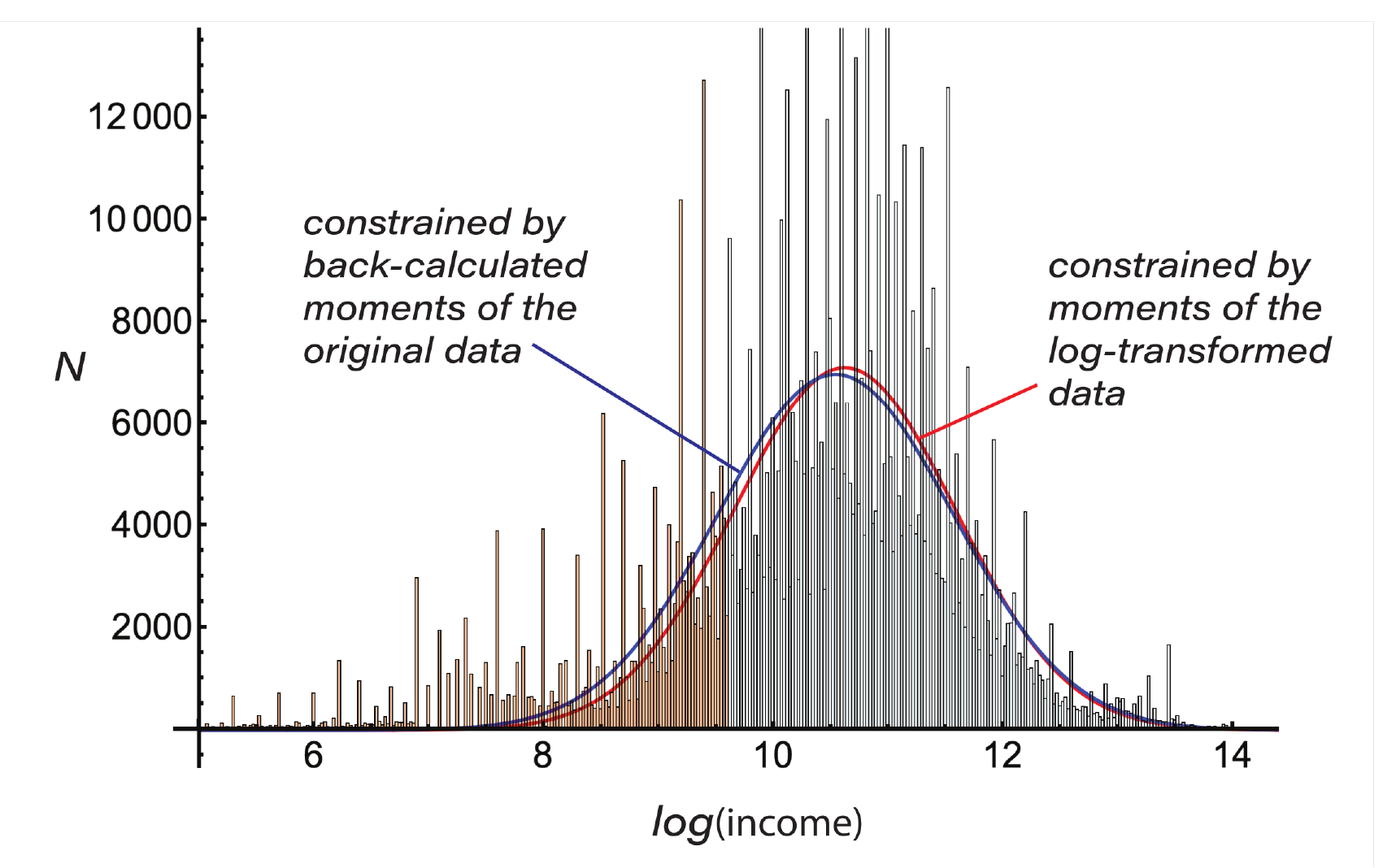}
\caption{\small Left-truncated distribution of log$_e$-transformed income data from the 2022 U.S. Census (see text).
The brown-shaded region, representing incomes less than $\$15000$, is removed in the analysis.
The red curve represents the Gaussian for which the 1st and 2nd moments best match that of the
log-transformed data--i.e., the truncated distribution shown. The blue curve represents the Gaussian 
that results from a parameterization by which the back-calculated estimates of the moments 
of the original (non-transformed) data,
computed using Eqs.~(\ref{var_orig_formI}) and (\ref{var_orig_form2}), best match the actual 
moments of that original data.}
\label{fig:US_census_data}
\end{figure}

\beq
=M^2\left(\frac{e^{-2\mu} M^2 \xi(r)^3 \cdot \xi(r+2\sigma)}
{\xi(r+\sigma)^4}-1\right) ~~~\mathrm{(Form~II)},
\label{var_orig_form2}
\eeq
where $r = (\mu -a)/\sigma$ and $\xi(r) = \erf(r/\sqrt{2} + 1$).

The results are shown in the the 6th and 7th columns of Table \ref{tab:us_census}, 
Panel (a). As indicated, the $\mu$ and $\sigma$ values obtained by reproducing the moments of the 
log-transformed data ($M(x)$ and $\mathrm{Var}(x)$) underestimate the actual variance of the original 
dataset by $\approx 8.9\%$ to $17.4\%$  and either under- or overestimate the mean by $0.87\%$ to $1.5\%$.
These discrepancies arise from the deviations of the dataset from an exact log-normal distribution. 

However, we can repeat the calibration using a formalism based on the expressions above so as to match the mean and variance
of the original data, $M(y)$ and $\mathrm{Var}(y)$, rather than those of the log-transformed data.
In our parameterization methods, we substitute Forms I and II of the variance given by Eqs.~(\ref{var_orig_formI}) and (\ref{var_orig_form2})  
above--or more precisely, their natural logs-- 
for Forms I and II of the variance as defined previously (Eqs. \ref{var_sig_r} and \ref{var_M_r}) and use the central moments
of the original data ($\ln(\mathrm{Var}(y))$ and $M(y)$) as the constraints, rather than those of the log-transformed data ($\mathrm{Var}(x)$ and $M(x)$). 
In this approach, the expressions for $d \sigma_1(\mu)/d\mu$ and 
$d \sigma_2(\mu)/d\mu$ used in the point-slope method, for example, are derived from the log of 
Eqs.~(\ref{var_orig_formI}) and (\ref{var_orig_form2})
(and are provided in Supplementary Material, Section~\ref{subsec:calib_log-normal}).
The log of the variance is used rather than the variance, itself, because 
Eqs.~(\ref{var_orig_formI}) and (\ref{var_orig_form2}) involve large numbers and have steep dependences on $\mu$ and $\sigma$, 
making convergence of the results much less efficient.
Even using the log of the expressions, 
the starting value for $\mu$ should be taken in a region as close as
possible to what might be expected as the optimal value to improve convergence.  
For these calculations, we chose the value obtained from the original trials (Panel (a) 
in the table) using approximating function 1. 

\begin{table}[H]
\small
\setstretch{1}
\tabcolsep=0pt
\begin{tabular*}{1.0\columnwidth}{@{\extracolsep{\fill}}lrrrrrr@{}}
\hline
\hline
\multicolumn{7}{c}{US Census income data} \\
\hline
\multicolumn{1}{c}{ }  & \multicolumn{1}{c}{}  & \multicolumn{1}{c}{ }  
& \multicolumn{2}{c}{\underline{~log-transformed data~}}  & \multicolumn{2}{c}{\underline{~~~~~original data~~~~~} }\\
\multicolumn{1}{c}{Method }  & \multicolumn{1}{c}{$\mu_{_0}$ } & \multicolumn{1}{c}{ $\sigma_{_0}$ }  
& \multicolumn{1}{c}{$\mathrm{Var}(x)$ }  & \multicolumn{1}{c}{mean } & 
\multicolumn{1}{c}{$\mathrm{Var}(y)$ }  & \multicolumn{1}{c}{mean } \\ [0.5ex]
\hline
\multicolumn{1}{c}{(a)} &\multicolumn{6}{c}{Moments matched to log-transformed data}\\
\hline
actual  &    &    &  0.59624965  &  10.87794700  &  8.30314328E09  &  75588.26676\\
approx1  &  10.62072268  &  0.95980918  &  0.59381611  &  10.87920851  &  6.86251643E09  &  74429.39513\\
two-point$^{**}$  &  10.62597505  &  0.97573110  &  0.61037287  &  10.89284606  &  7.56086257E09   &  76243.14095\\
pt-slope\textsuperscript{$\dagger$}  &  10.61705449  &  0.96351573  &  0.59625846  &  10.87923477  &  6.95525687E09  & 74563.17154 \\ 
pt-slope\! $\times 2$   &  10.61417290  &  0.96437852  &  0.59624910  &  10.87794276  &  6.95077578E09  &  74477.64809 \\
pt-slope\! $\times 3$  &  10.61418159  &  0.96437647  &  0.59624965  &  10.87794700  &  6.95081031E09  &  74477.95700 \\
\hline
\multicolumn{1}{c}{(b)} &\multicolumn{6}{c}{Moments matched to original data}\\
\hline
actual  &         &    &  0.59624965  &  10.87794700  &  8.30314328E09  &  75588.26676  \\
pt-slope\textsuperscript{$\ddagger$}      &     10.54083623&  1.02134843 & 0.62878841 & 10.87028469  &  8.31968113E09 & 75794.98588    \\       
pt-slope\! $\times 2$     &  10.53369913 & 1.02332445 & 0.62853040 & 10.86734490  &  8.30327221E09 & 75589.21062  \\
pt-slope\! $\times 3$     &  10.53367109 & 1.02333081 & 0.62852804 & 10.86733243 &   8.30314328E09 & 75588.26678   \\  
\hline
\end{tabular*}
\caption{\small Modeling results for log-transformed income data. In {\bf{Panel (a)}} (top),
log$_e$-transformed 2022 income data for U.S. citizens (see text) is modeled using the expressions 
for the variance of UTGDs (Eqs. \ref{var_sig_r} and \ref{var_M_r}), 
with the variance and mean of the log-transformed data used as constraints, similarly to prior examples.
In {\bf{Panel (b)}} (bottom), the same data are modeled using the derived expression for the variance and mean
of the original (untransformed) data (Eqs.~\ref{var_orig_formI} and \ref{var_orig_form2}; more precisely, the log of these expressions),
with the mean and the log of the variance of the original data used as constraints.  
$\boldsymbol{\{\mu_0, \sigma_0\}}$--the calculated location and spread parameters, 
{\bf{approx1}}--approximating function 1, 
{\bf{two-point}}--two-point moment-intersecting method,
{\bf{pt-slope}}--point-slope moment-intersecting method, {\bf{pt-slope$\bm{\times2}$}}--point-slope method
applied twice successively, taking the $\mu_{_0}$ value from the first trial as starting point in the second,
{\bf{pt-slope$\bm{\times3}$}}-the point-slope method applied three times consecutively, {\bf{actual}}--the actual
values for the variance and mean of the distributions.
$**$ using $\mu_1 = 9.7$ and $\mu_2=11.0$.
$\dagger$ the starting $\mu$ value was the result from the two-point calculation.
$\ddagger$ the starting $\mu$ value was the result from approx1 in the top panel.
}
\label{tab:us_census}
\end{table}

As shown in Panel (b) of the table, after one application of the point-slope method, the estimate for 
$\mathrm{Var}(y)$ is correct to within $0.2\%$, or one decimal place; after 3 applications, to within 
$3 \times 10^{-8}\%$ or 8 decimal places. The $\%$ errors for the mean, $M(y)$, are similar. 

As expected with this approach, 
the mean and variance estimates for the log-transformed data ($M(x)$ and $\mathrm{Var}(x)$) 
are less accurate than in the prior set of calculations (Panel a).

\noindent Because there are only two model degrees of freedom, 
both sets of moments cannot be matched simultaneously. 
The ``best fit'' curves for the original-data-moment matching and the transformed-data-moment matching approaches, according to
Table \ref{tab:us_census}, are shown in Figure~\ref{fig:US_census_data}.  

Generally, this approach to reproducing the first and second central moments, $M(y)$ and $\mathrm{Var}(y)$, of an original dataset, $Y$, 
based on the UTGD parameters $\sigma$ and $\mu$ obtained from the modeling of a transformation of that dataset, $X = T(Y)$,
 would seem to be applicable\footnote{assuming $T$ satisfies the usual conditions for a valid change-of-variables transformation.},
 when the Form I and Form II variances  ($\mathrm{Var}(y;\sigma,\mu$))
 of the original dataset 1) are derivable or expressible in terms of those same parameters, and 2) give rise to
corresponding, implicit $\sigma(\mu)$ curves that are sufficiently smooth at their intersection.

\subsubsection{Truncated $n$-dimensional velocity distributions}
\label{n-dim_velocities}
In the spirit of (\cite{suman2020}), consider a model of the velocities of argon-like particles in 
$n$-D space that uses an extension of the canonical Maxwell-Boltzmann formulation. Assuming a straightforward
 extension of the known 1- through 4-D versions of the model, the general form of such a distribution of 
particle velocities, from Eq.~(\ref{gaussian_pdf_of_R}), is
\beq
g(v) = \frac{2 e^{-\frac{v^2}{2\sigma^2}} v^{n-1}}
{(\sqrt{2}\sigma)^{\rbox{2pt}{n}}~ \Gamma\!\left(\frac{n}{2},\frac{r^2}{2}\right)} dv, 
\eeq
where $v = \sqrt{\sum_{i=1}^n {v_i}^2}$, each $v_i$ is the $i$th velocity component of the particles, 
 $\sigma = \sqrt{RT/m}$, $m$ is the particle mass, and $R$ is the Boltzmann constant in SI units, 
$8.31446262$ J K$^{-1}$ mol$^{-1}$. The mass is constant, but the temperature varies, as from the
 definition of $\sigma$, $T = m \sigma^2/R$. 
\begin{singlespace}
\begin{figure}[H]
\vspace{-1em}
\centering
\includegraphics[width=4in]{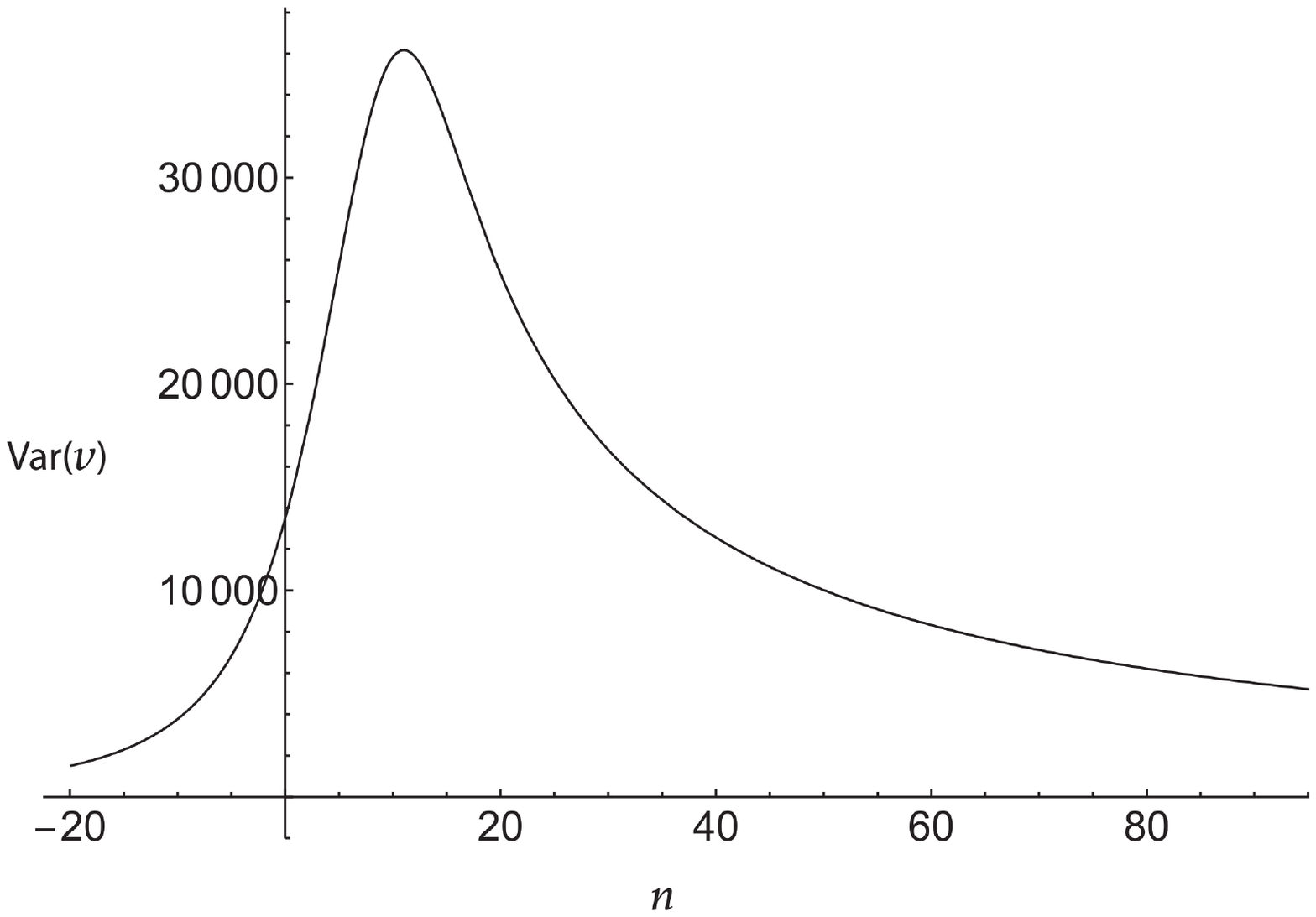}
\vspace{-1em}
\caption{\small  
Plot of the variance of the particle velocities, $\mathrm{Var}(v)$, in $m{\smexp{^2}}/s{\smexp{^2}}$
vs. the number of spatial dimensions, $n \in \mathbb{R}$, for an argon-like gas conforming to $n$-D 
Maxwell-Boltzmann statistics with lower-velocity truncation and a fixed mean velocity of 1000 m/s. 
From Eq.~(\ref{multidim_var}).}
\label{fig:Max-Boltz}
\end{figure}
\end{singlespace}

Suppose the constraints on the system are such that the normalized velocities, $a/\sigma = -r$ must all
be above $2.2$ and the mean velocity of the system, $M$, is fixed at $M=1000$ m/s. We would like to 
determine the dimensionality ($n$) required to maximize the variance in the velocities under these conditions, 
as well as the associated temperature, $\sigma$ value and velocity cutoff ($a$) value.

Letting $|r| = 2.2$ in Eq.~(\ref{nvmx_approx}), the estimate for $n_{\mathrm{vmx}} \approx 10.887$.
Then, substituting the fixed value for the mean velocity (1000 m/s) into Form II of the variance 
(Eq. \ref{multidim_var}), we find that $\mathrm{Var}(v,M,n=10) = 35974.945$ m$\smexp{^2}$/s$\smexp{^2}$  
and $\mathrm{Var}(v,M,n=11) = 36227.769$ m$\smexp{^2}$/s$\smexp{^2}$. 

Hence, the maximal variance dimensionality, 
$n^{\ddagger}_{\mathrm{vmx}} = 11$. For illustrative purposes, a plot of the variance in the velocities as 
a function of $n$, for $n \in [-20,100] \subset \mathbb{R}$, is shown in Figure~\ref{fig:Max-Boltz}.

From Eq.~(\ref{sigma_multi}), $\sigma = (1000 ~\mathrm{m/s}) ~\Gamma\!\left(\frac{11}{2},\frac{(2.2)^2}{2}\right)/
\sqrt{2} ~\Gamma\!\left(\frac{11+1}{2},\frac{(2.2)^2}{2}\right) = $  300.605 m/s and from the expression above 
for the temperature, $T = 434.163$ K. The required cutoff value is, thus, $a = -r \sigma = 661.330$ m/s.
Form I of the variance (Eq. \ref{inner_unsub_var}) evaluated at the above $\sigma$ value and $n=11$ verifies that 
$\mathrm{Var}(v)$ = 36227.769 m$\smexp{^2}/$s$\smexp{^2}$.

\subsection{Form of a UTGD in terms of the mean and $r$ value}
\label{subsec:form_of_f}
Assuming that the parameters $\mu, \sigma $ and $a$ are consistent with the mean, $M$, over
the interval of $x \in [a,\infty)$, the general form of a Gaussian distribution over that
interval is $f(x;M,r)_{|x\ge a} = f_{\!_{M}} e^{\omega}$.
where
\bdm
\omega = - \frac{\left(r e^{r^2/2} \xi(r)(x-M) + \sqrt{2/\pi}(x-a)\right)^2}
{2 e^{r^2} \xi(r)^2 (M-a)^2},
\edm
$\xi(r)= \erf\!\left(r/\sqrt{2}\right)+ 1$, and $r = (\mu -a)/\sigma$.
\subsection{Slopes of $\sigma(\mu)$ under constant variance for a UTGD}
\label{subsec:dsigma_dmu}
The general scheme is to derive $d\sigma/d\mu$ as
\beq
\frac{d\sigma}{d\mu} = - \frac{d\mathrm{Var}(x;M,\sigma,\mu,r)/d\mu}
{d\mathrm{Var}(x;M,\sigma,\mu,r)/d\sigma},
\label{dsdu}
\eeq
under the condition that $\mathrm{Var}(x;M,\sigma,\mu,r)$ is constant
and noting  
that $\frac{d\sigma}{d\mu}$ is ``tied'' to the specific form of the variance.
This relation arises from the fact that, around a point on the level curve of a function, $f(x,y) = c$,
$df = 0 = \frac{\partial f}{\partial x} dx + \frac{\partial f}{\partial y}dy$, where $c$ is a constant. However,
 in the current case, if one uses functions of $r$, such as $\mathrm{Var}(x;\sigma,r)$ (Form I),
there are implicit dependencies and thus the total derivatives must be used.

The total derivatives of Form I of the variance with respect to $\mu$ and $\sigma$ are
\bdm
\frac{d \mathrm{Var}(x;\sigma,r)}{d\mu} =\frac{\partial \mathrm{Var}(x;\sigma,r)} {\partial r}
\frac{\partial r} {\partial \mu}
~~~\mathrm{and}~~~\frac{d \mathrm{Var}(x;\sigma,r)}{d\sigma} =\frac{\partial \mathrm{Var}(x;\sigma,r)} {\partial r}
\frac{\partial r} {\partial \sigma} + \frac{\partial \mathrm{Var}(x;\sigma,r)} {\partial \sigma}.
\edm
Noting that $\partial r/\partial \sigma = -r/\sigma$ and $\partial r/\partial \mu = 1/\sigma$, from Eq.~(\ref{dsdu}), we have
\bdm
\frac{d\sigma_1}{d\mu} = -\frac{\frac{d \mathrm{Var}(x;r)}{d\mu}}
{\frac{d \mathrm{Var}(x;r)}{d\sigma}} =
\frac{\partial \mathrm{Var}(x;\sigma,r)/\partial r}
{r\cdot \partial \mathrm{Var}(x;\sigma,r)/\partial r - \sigma\cdot \partial \mathrm{Var}(x;\sigma,r)/\partial \sigma}.
\edm
where the ``1'' subscript indicates that Form I of the variance is being held constant.
Carrying out the differentiation, the result is
\beq
\frac{d\sigma_1}{d\mu} = \frac{ \sqrt{2}\pi (r^2-1) e^{r^2} \xi(r)^2+ 6\sqrt{\pi} r  e^{r^2/2}\xi(r) + 4\sqrt{2}         }
{\mathscr{H}(r)},
\label{dsdu_form1}
\eeq
where $\xi(r) = \erf\!\left(r/\sqrt{2}\right)+ 1$ and
\bdm
 \mathscr{H}(r)= -\sqrt{2}\pi r(r^2+1) e^{r^2}\xi(r)^2 - 4\sqrt{2}r ~+
2\pi^{3/2}  e^{3r^2/2} \xi(r)^3 - 2\sqrt{\pi} (3r^2+2) e^{r^2/2} \xi(r).
\edm
Since $\mathscr{H}(r) \rightarrow 0$ only as $r\rightarrow -\infty$, the slope exists for all
real $r$.
The minimum of the slope, $\min_r \! \left(d \sigma_1/d\mu \right) \approx -0.3247082711469855$,  
occurs at $r \approx  0.5987678543728484$, while the supremum is 0 and occurs at $r\rightarrow \pm \infty$.
See Figure~\ref{fig:ds1_du_plot} in Supplementary Material (Section~\ref{subsec:plot_ds1_du}). 

By contrast, although Form II of the variance (Eq. \ref{var_M_r}) depends explicitly on $M$, $r$ and $a$,
the $M$ and $a$ parameters are taken as given--i.e., constant--in these applications. Thus,
\bdm
\frac{d\sigma_2}{d\mu} =-\frac{\frac{d \mathrm{Var}(x;r)}{d\mu}}
{\frac{d \mathrm{Var}(x;r)}{d\sigma}} =
-\frac{\frac{\partial \mathrm{Var}(x;r)}{\partial r} \frac{\partial r}{\partial \mu} }
{\frac{\partial \mathrm{Var}(x;r)}{\partial r} \frac{\partial r}{\partial \sigma} }
= -\frac{1/\sigma}{-r/\sigma} = \frac{1}{r} = \frac{\sigma}{\mu-a},
\edm
provided $\mu \ne a$,
as $\frac{\partial \mathrm{Var}(x;r)}{\partial r}$ cancels.
Hence, the slopes of $\sigma(\mu)$ under constant variance that arise from the
two forms of the variance differ in general.
In addition, for Form II, since $d\sigma/\sigma = d\mu/(\mu-a)$, integrating both sides gives 
$\sigma(\mu) = e^{c_0}(\mu-a)$, where $c_0$ is a constant of integration.
(Since Form II depends only on $M$, $a$ and $r$, if the
first two variables are constant and the variance is fixed, $r$ must be constant.)
Thus, for a UTGD, $\sigma(\mu)$ derived from Form II under a constant variance, mean and cutoff value
is always exactly linear.
\subsection{Convergence of inner and outer truncation for UTSCDs}
\label{inner_outer_converge}
 As mentioned in Section~\ref{subsec:outer_truncation}, it is expected that the variances
for inner- and outer-truncated scaled chi distributions converge in the limiting cases where the
maximal extent of the distributions is included--namely, $|r| \rightarrow 0$ for inner truncation
and $|r| \rightarrow \infty$ for outer truncation.
This does turn out to be the case for positive dimensionalities, as shown in the prior results.

For negative dimensionalities, outer-truncated scaled chi distributions
\bdm
g(R;\sigma,n)_{|R \le a} dR =
\frac{2 e^{-R^2/2\sigma^2} R^{n-1}}{(\sqrt{2} \sigma)^n {\lgamma}\left(\frac{n}{2},\frac{r^2}{2}\right)}  dR
\edm
can take on negative values, because unlike the upper incomplete gamma function, $\Gamma(s,x)$, the lower incomplete gamma function, $\lgamma(s,x)$, has real zeroes for $x>0$,
specifically over the intervals $s \in (-2,-1)\cup (-4,-3) \cup (-6,-5)...$  Over the interval $s \in (-1,0)$,
for example, $\lgamma(s,x) < 0$, for $x>0$, and hence $g(R;\sigma,n)_{|R \le a} dR  < 0$ for $n \in (-2,0) \subset \mathbb{R}$.
In addition, for negative dimensionalities,
the expressions for the raw moments of the outer-truncated scaled chi distributions (Eq. \ref{kth_raw_moment_outer} in Appendix)
are not always well-defined,
because the integrals only converge if $n > -k$, where $k$ is the order of the moment.

\begin{singlespace}
\begin{table}[H]
\setstretch{1}
\renewcommand{\arraystretch}{0.4}
\tabcolsep=0pt
\begin{tabular*}{1.0\columnwidth}{@{\extracolsep{\fill}}ccc@{}}
\hline
\hline
\noalign{\vskip 4pt}
\multicolumn{1}{c}{$\lim_{|r| \to }$} & \multicolumn{1}{c}{outer truncation} & \multicolumn{1}{c}{inner truncation}\\
\hline
\noalign{\vskip 4pt}
\hline
\noalign{\vskip 4pt}
\multicolumn{3}{c}{$\mathrm{\mathbf{Var}}\boldsymbol{(R;M,r,n)}$ } \\
\noalign{\vskip 4pt}
\hline
\\
0  &   $\frac{M^2}{n(n+2)}$ , $n \notin \{0,-2,-4...\}$\textsuperscript{**}  &
$M^2\left(\frac{n~\Gamma\!\left(\frac{n}{2}\right)^2}
{2 \Gamma\!\left(\frac{n+1}{2}\right)^2} - 1\right),$ $n>0$ \\
\\
& & $\infty$, $n \in [-2,0]$\\
\\
& & $\frac{M^2}{n(n+2)}$, $n<-2$\\
\\
\hdashline
\\
$\infty$ &   $M^2\left(\frac{n~\Gamma\!\left(\frac{n}{2}\right)^2}
{2 \Gamma\!\left(\frac{n+1}{2}\right)^2} - 1\right), n \notin \{0,-2,-4...\}$\textsuperscript{\S} &
0 \\
\\
\hline
\noalign{\vskip 4pt}
\multicolumn{3}{c}{$\boldsymbol{\sigma(M,r,n)}$}\\
\noalign{\vskip 4pt}
\hline
\hline
\\
0   &$\infty,~n<-1$ or $n>0$\textsuperscript{\P} & $\frac{M \Gamma(\frac{n}{2})}{\sqrt{2} \Gamma(\frac{n+1}{2})}$, $n>0$     \\
\\
& & $\infty$, $n \le 0$\\
\\
\hdashline
\\
$\infty$  & $\frac{M \Gamma(\frac{n}{2})}{\sqrt{2} \Gamma(\frac{n+1}{2})}$, $n \notin \{0,-2,-4...\}$  &  0   \\
\\
\hline
\hline
\noalign{\vskip 4pt}
\multicolumn{3}{c}{$\boldsymbol{a}$ {\bf{(cutoff)}}}  \\
\noalign{\vskip 4pt}
\hline
\\
0  &  $\frac{M(n+1)}{n},~ n<-1$ or $n>0^\ddagger$  &  0, $n \ge -1$ \\
\\
&  &  $\frac{M(n+1)}{n}$, $n < -1$  \\
\\
\hdashline
\\
$\infty$  & $\infty,~n \notin \{0,-1,-2,-3,...\} $   &   $M$  \\
\noalign{\vskip 4pt}
\hline
\noalign{\vskip 4pt}
\end{tabular*}
\caption{Comparison between the behavior of outer- ($R \le a$) and
inner- ($R \ge a$) truncated \Tstrut scaled
chi-distributions with $n$ distributional degrees of freedom at limiting values of the parameter
$|r|$, for $n \in \mathbb{R}$.
This version of the table includes results for outer truncation with $n<0$, which are obtained by extending
the domains of the moments to these values of $n$.  The table headings are the same as for
Table \ref{tab:in_out_comp_trans}. $\boldsymbol{**}$ For $n \in (-2,0)$, variances are negative. 
$\boldsymbol{\S}$ For $n<0$, variances are negative. $\boldsymbol{\P}$ For $n \in (-1,0)$, $\sigma$ would be negative.
$\boldsymbol{\ddagger}$ For $n \in (-1,0)$, cutoff would be negative. }
\label{tab:in_out_comp_trans_full}
\end{table}
\end{singlespace}

However, if we are prepared to accept negative probabilities(\cite{Wigner1932,blass2018negative,Sorin-Nicolae2022}), and if we extend the form
of the moments for outer truncation to negative $n$, we can examine these variances as well.  Notably,
the maximal variances for UTSCDs having negative dimensionalities show complex behavior, converging for inner and outer truncation only
over particular intervals of $n$, as shown in Figure~\ref{fig:non-converging_var_comb}, Panel (b).  Also, as shown in Figure 
\ref{fig:rvmx}, unlike the case for positive $n$, the cutoff values resulting in the maximal variance are not always
zero; they are non-zero when the maximal variance differs for inner and outer truncation.  

In addition, we can show both numerically and analytically that
the variances for inner and outer truncation unexpectedly diverge in the limiting cases where the distributions are included
to their maximal extent (when $|r| \rightarrow 0$ for inner truncation and $|r| \rightarrow \infty$ for outer truncation).
This is illustrated in Panel (a) of Figure~\ref{fig:non-converging_var_comb}. It occurs because the 
forms of the variances differ only by the substitution of $\lgamma(,)$ for $\Gamma(,)$, and
because the relation 
$\lim_{x \to \infty} \lgamma(s,x) = \lim_{x \to 0^+} \Gamma(s,x) = \Gamma(s)$ holds only for $s>0$.
For $s<0 \subset \mathbb{R}$, 
\\
$\lim_{x \to \infty} \lgamma(s,x) = \Gamma(s)$, provided $s$ is not 0 or a negative
integer, but $\lim_{x \to 0^+} \Gamma(s,x) = +\infty$, so the limits do not correspond.

Moreover, as shown in Table \ref{tab:in_out_comp_trans_full} and as can also be demonstrated numerically and analytically 
(see Appendix, Section~\ref{subsec:lim_var_chi_dist}), under this extended-domain definition of the moments, 
we find, somewhat counter-intuitively, that for $n<-2, n \notin \{-4,-6...\}$, the results for the variance 
under outer truncation as $|r| \rightarrow 0$ (i.e., for collapsing distributions) correspond to those under 
inner truncation as $|r| \rightarrow 0$
(i.e., for maximally expanding distributions). The results for the variance and maximal variance under 
inner truncation, as $|r| \rightarrow 0$, are illustrated by the
red curve in Figure~\ref{fig:non-converging_var_comb} Panel (a) and the red portions of the curve in Panel (b).

\begin{figure}[H]
\vspace{-0.5cm}
\hspace{-1.0cm}
\includegraphics[width=7.25in]{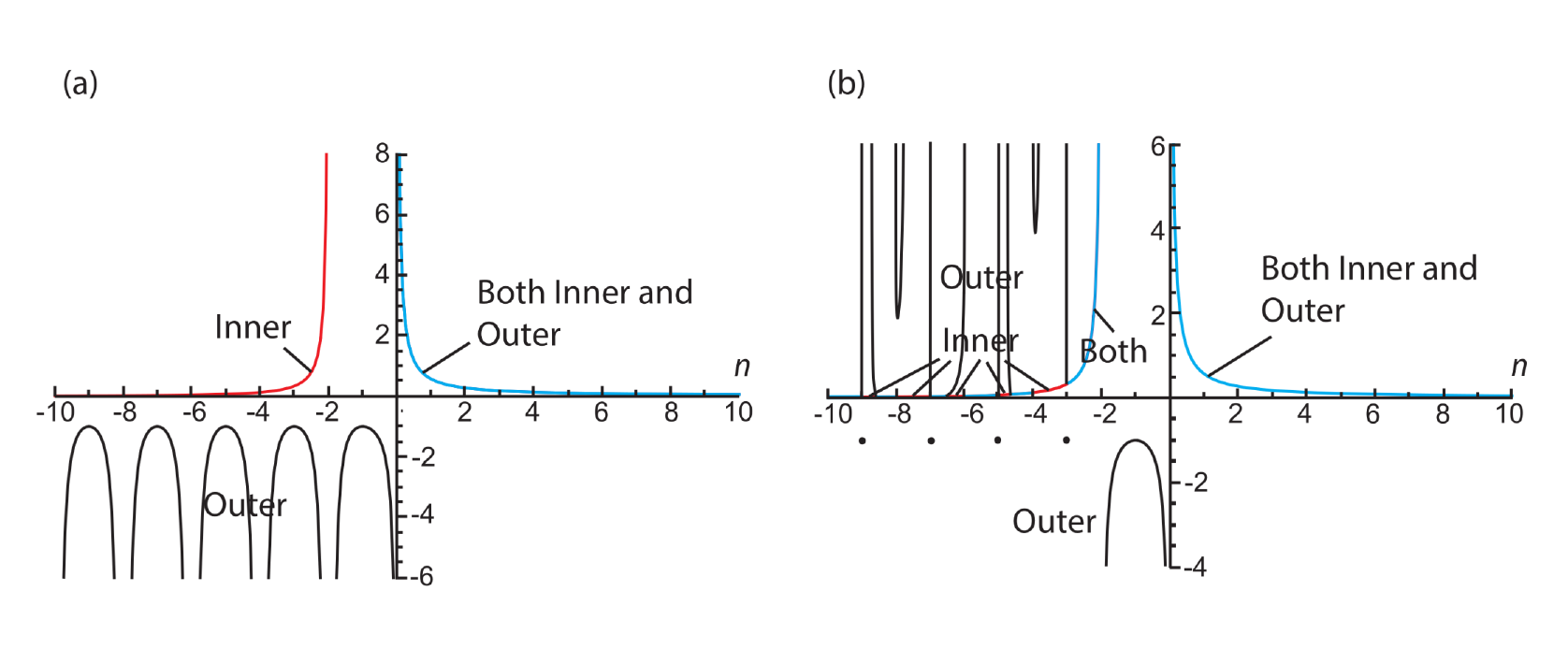}
\vspace{-1.0cm}
\caption{\small
{\bf{Panel (a)}} Partial convergence of the limiting variances for scaled chi distributions under inner and outer truncation,
as a function of the number of distributional degrees of freedom (or dimensionality), $n$.
The curves indicate the  variance, given a fixed mean of 1
and fixed dimensionality, $n \in \mathbb{R}$, for different truncation conditions.
The red curve (labeled ``Inner'') gives the limit of the variance obtained with
inner truncation (Eq. \ref{multidim_var_inv}) as $|r| \rightarrow 0$.
The black curves (labeled ``Outer'') give that obtained with outer truncation
as $|r| \rightarrow \infty$, and they indicate that the results for $n<0$ diverge.
For $n>0$, the results converge, as indicated by the blue curve (labeled ``Both Inner and Outer'').
{\bf{Panel (b)}}
Comparison of the maximal variance $\mathrm{V_{max}}_{,r}(n)$ in UTSCDs, with a mean of 1,
using either inner or outer truncation, as a function of dimensionality ($n$).
The labeling and color schemes are as in Panel (a).  For $n>0$, the results for both truncation
schemes coincide and correspond to Eq.~(\ref{var_ndim_r=0}).  However, for $n<0$, they coincide only in particular
intervals:  \( n \in (-3,-2] \cup (\approx -4.65148785,-4) \cup (-6,-5) \cup (-7,\approx -6.61423671) \cup 
(\approx -8.57469649,-8) \cup (-10,-9) \), etc.
In the intervening regions, the maximal variances obtained under the two truncation
schemes diverge. For outer truncation, there are singular points in the variance
at $(-2k-1,-1)$, for $k \in \mathbb{Z+}$, four of which are shown in the plot.
}
\label{fig:non-converging_var_comb}
\end{figure}
Further, as indicated in Table \ref{tab:in_out_comp_trans_full},
for $n<-2$, not only the variances, but also the $\sigma$ parameters and
cutoff values under inner and outer truncation correspond in the limit of $|r| \rightarrow 0$.
 In addition, it can be shown by first-order series expansion of $\lgamma(n/2,r^2/2)$
and $\Gamma(n/2,r^2/2)$ that for small $|r|$ and $n<0$, the inner and outer truncated scaled chi distributions are
additive inverses of each other: $g(R;\sigma,r,n)_{|R\ge a} dR =
-g(R;\sigma,r,n)_{|R\le a} dR = - n e^{-R^2/2\sigma^2} R^{n-1}/(|r|\sigma)^n dR$
(from Eqs. \ref{gaussian_pdf_of_R} and \ref{gaussian_pdf_of_R_inv}).
Since the moments of these limiting distributions are equal,
it would appear that for negative dimensionalities
($n < -2$) and in the limit of small $|r|$, collapsing UTSCDs under outer truncation
and maximally expanding UTSCDs under inner truncation are, in this sense, mathematically equivalent.
\begin{figure}[H]
\centering
\includegraphics[width=4.5in]{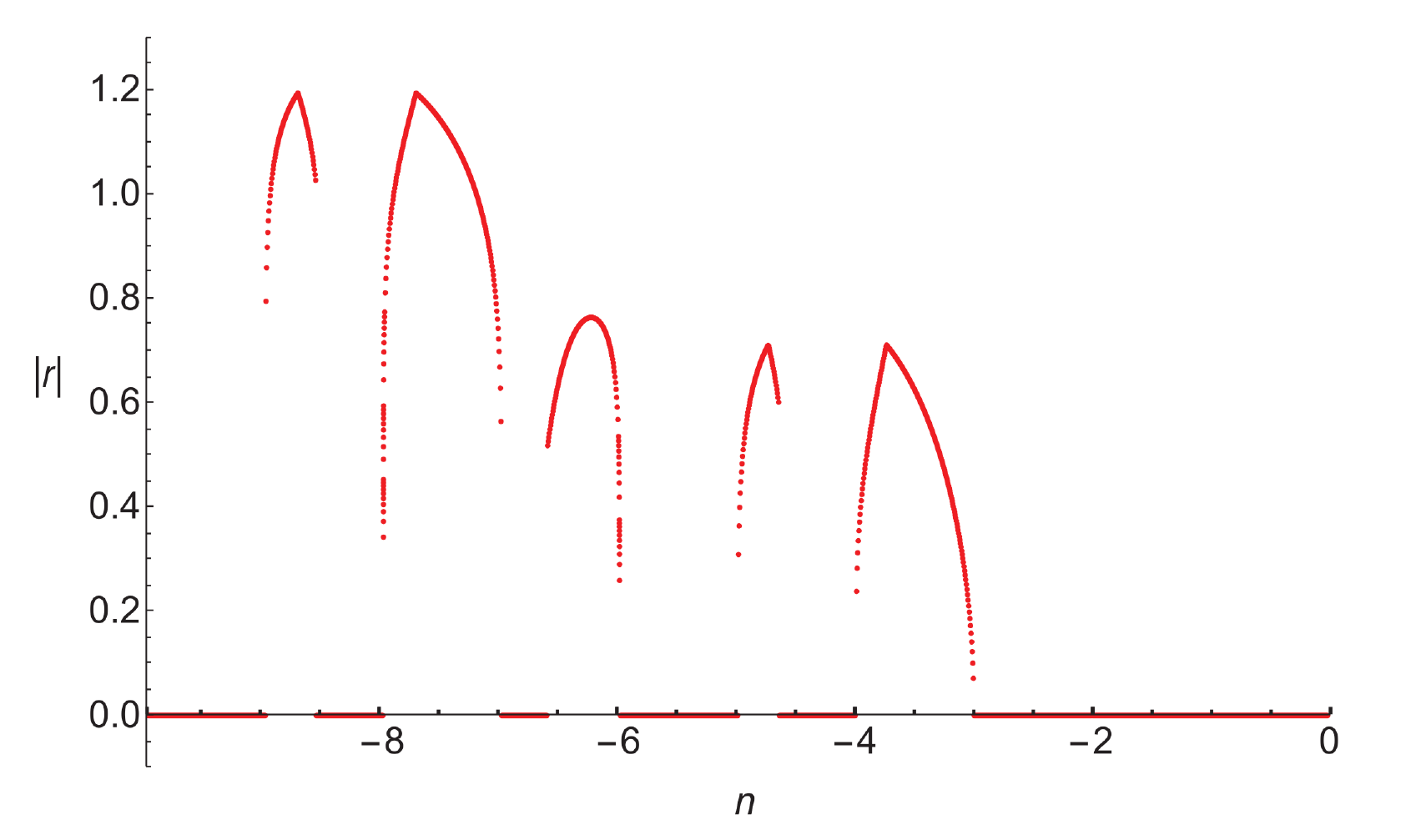}
\vspace{-0.5cm}
\caption{\small
Plot of the $|r|$ value that results in $\mathrm{V_{max,r}}(n)$ for $n < 0 \subset \mathbb{R}$ in scaled chi distributions,
using outer truncation with a fixed mean of 1.}
\label{fig:rvmx}
\end{figure}

\subsection{Functional approximations of the maximal variance of UTSCDs}
\label{subsec:approx_max_var_inner}
\begin{figure}[H]
\vspace{-0.5cm}
\hspace{-1.5cm}
\includegraphics[width=7.5in]{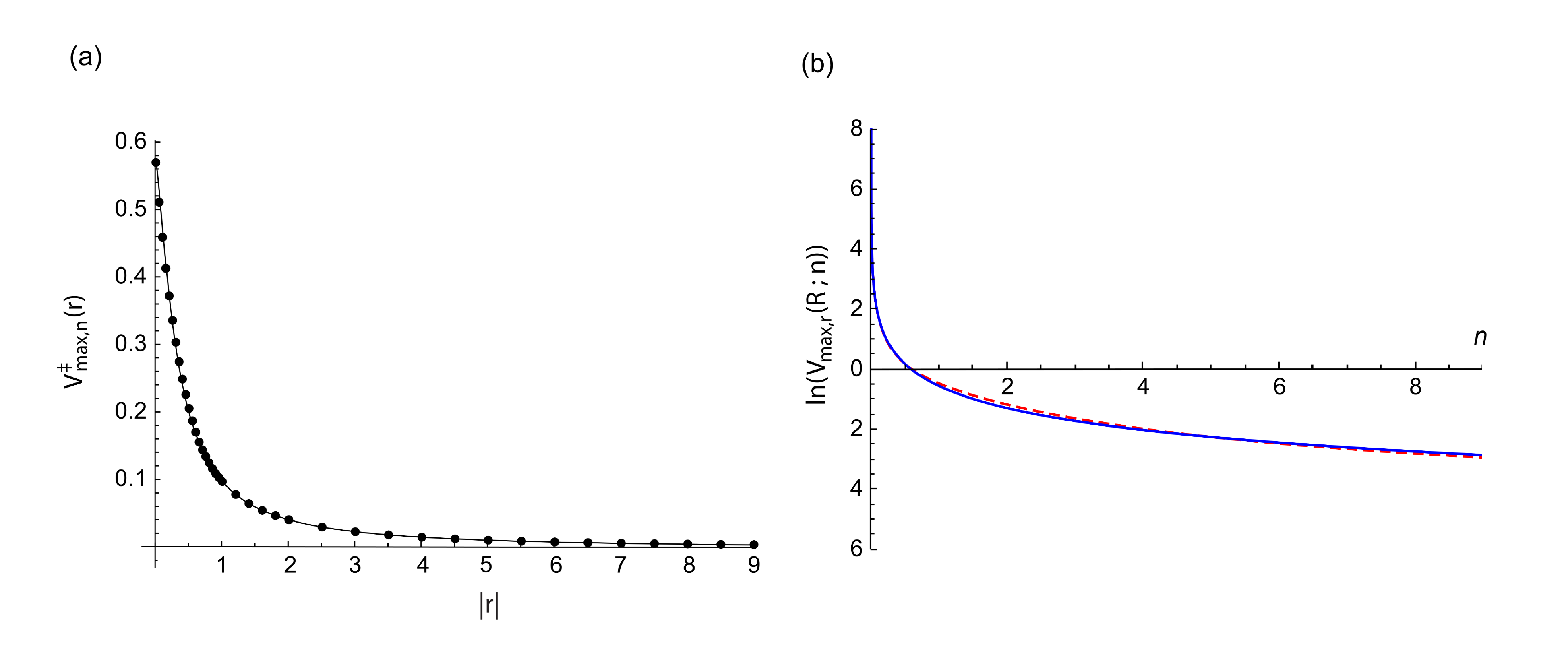}
\vspace{-1.5cm}
\caption{\small Plots of the maximal variance of UTSCDs under inner truncation ($g(R;\sigma,r,n)_{|R\ge a} dR$.
{\bf{Panel (a)}} Comparison of the maximal variance, $\mathrm{V^\ddagger_{max}}_{,n}(r)$ (filled circles) for
UTSCDs over the interval $R \in [a,\infty)$,
and the approximation from Eq. \ref{Vmax_approx} (curve), as a function of the $r$ parameter.
{\bf{Panel(b)}} Plot of the log of the maximal variance $\mathrm{V_{max}}_{,r}(n)$ of UTSCDs over the
interval $R \in [a,\infty)$ as a function of the number of distributional
degrees of freedom, $n>0 \subset \mathbb{R}$, for a fixed mean of 1.
The solid blue curve was generated from the log of the variance as given in Eq.~(\ref{var_ndim_r=0})
and the red, dashed curve from the log of the approximation $\frac{2 e^{-a n}}{n \pi} + \frac{e^{-b/n}}{2n}$,
where $a \approx 0.33118$ and $b \approx 1.10590$.
}
\label{fig:Vmax_approx_comb}
\end{figure}
\subsection{Limits of the variance of truncated, scaled chi distributions}
\label{subsec:lim_var_chi_dist}
\hfill\\
{\bf{A) As $\bm{|r| \rightarrow 0}$}}

First order series expansions as $|r| \rightarrow 0$ give
\bdm
\lgamma\!\left(\frac{n}{2},\frac{r^2}{2}\right) \approx \frac{2|r|^n}{(\sqrt{2})^{\rbox{2pt}{n}} n},
\edm

\bdm
\lgamma\!\left(\frac{n+2}{2},\frac{r^2}{2}\right) \approx \frac{|r|^{n+2}}{(\sqrt{2})^{\rbox{2pt}{n}} (n+2)},
\edm
and
\bdm
\lgamma\!\left(\frac{n+1}{2},\frac{r^2}{2}\right) \approx \frac{|r|^{n+1}}{(\sqrt{2})^{\rbox{2pt}{n-1}} (n+1)}.
\edm

Hence, from Eq.~(\ref{multidim_var}), the variance of an outer-truncated, scaled chi distribution in the limit of
$|r|\rightarrow 0$ is
\beq 
\lim_{|r| \to 0} \mathrm{Var}(R;M,r,n)_{|R\le a}  =
M^2 \left(
\frac{
\left( \frac{2|r|^n}{(\sqrt{2})^{\rboxx{0.8pt}{n}}\, n} \right)
\left( \frac{|r|^{n+2}}{(\sqrt{2})^{\rboxx{0.8pt}{n}}\, (n+2)} \right)
}
{
\left( \frac{|r|^{n+1}}{(\sqrt{2})^{\rboxx{0.8pt}{n-1}}\, (n+1)} \right)^2
} - 1 
\right)
= \frac{M^2}{n(n+2)},
\label{series_LIG_r=0} 
\eeq

%
where $n \ne \{0,-2\}$.
Since, $\Gamma(s,x) = \Gamma(s) - \lgamma(s,z)$, the corresponding limit of the variance for inner truncation is
\bdm
\lim_{|r| \to 0} \mathrm{Var}(R;M,r,n)_{|R \ge a} =
M^2 \left(
\frac{
\left( \Gamma\!\left(\frac{n}{2}\right) -
\frac{2|r|^n}{(\sqrt{2})^{\rboxx{0.8pt}{n}}\, n} \right)
\left( \Gamma\!\left(\frac{n+2}{2}\right) -
\frac{|r|^{n+2}}{(\sqrt{2})^{\rboxx{0.8pt}{n}}\, (n+2)} \right)
}{
\left( \Gamma\!\left(\frac{n+1}{2}\right) -
\frac{|r|^{n+1}}{(\sqrt{2})^{\rboxx{0.8pt}{n-1}}\, (n+1)} \right)^2
} - 1 \right).
\edm


For $n<-2$, the $\Gamma()$ terms are negligible, so the expression reduces to Eq. \ref{series_LIG_r=0} above.
For $n>0$, the $\Gamma()$ terms dominate, and the expression reduces to Eq. \ref{var_ndim_r=0}.
For $-2<n<-1$, the limiting variance is
\bdm
\lim_{|r| \to 0} \mathrm{Var}(R;M,r,n)_{|R\ge a} = \frac{-(\sqrt{2})^n (n+1)^2 \,\Gamma\!\left(\frac{n+2}{2}\right)}
{|r|^{n+2} n} -1  = +\infty.
\edm
For $-1<n<0$, the expression becomes
\bdm
\lim_{|r| \to 0} \mathrm{Var}(R;M,r,n)_{|R\ge a} = \frac{-2|r|^n\, \Gamma\!\left(\frac{n+2}{2}\right)}
{(\sqrt{2})^{\rbox{2pt}{n}} n \Gamma\!\left(\frac{n+1}{2}\right)  } -1   = +\infty.
\edm
\noindent\rule{5cm}{0.4pt} 

\noindent{\bf{B) As $\bm{|r| \rightarrow \infty}$}}\Tstrut  

First-order (Laurent) series expansions at $|r| \rightarrow \infty$ give
\bdm
\lgamma\!\left(\frac{n}{2},\frac{r^2}{2}\right) \approx -\frac{2 e^{-r^2/2}\, |r|^{n-2}}{(\sqrt{2})^{\rbox{2pt}{n}}}
+ \Gamma\!\left(\frac{n}{2}\right),
\edm
\bdm
\lgamma\!\left(\frac{n+2}{2},\frac{r^2}{2}\right) \approx -\frac{ e^{-r^2/2}\, |r|^{n}}{(\sqrt{2})^{\rbox{2pt}{n}}}
+ \Gamma\!\left(\frac{n+2}{2}\right),
\edm
and
\bdm
\lgamma\!\left(\frac{n+1}{2},\frac{r^2}{2}\right) \approx -\frac{ e^{-r^2/2}\, |r|^{n-1}}{(\sqrt{2})^{\rbox{2pt}{n-1}}}
+ \Gamma\!\left(\frac{n+1}{2}\right).
\edm

Hence, from Eq.~(\ref{multidim_var}), in the limit as $|r| \rightarrow \infty$, the variance for outer-truncated 
chi distributions is
\bdm
\lim_{|r| \to \infty} \mathrm{Var}(R;M,r,n)_{|R \le a} =
M^2 \left(
\frac{
\left(
\frac{2 e^{-r^2/2}\, |r|^{n-2}}{(\sqrt{2})^{\rboxx{0.8pt}{n}}} + \Gamma\!\left(\frac{n}{2}\right)
\right)
\left(
\frac{e^{-r^2/2}\, |r|^{n}}{(\sqrt{2})^{\rboxx{0.8pt}{n}}} + \Gamma\!\left(\frac{n+2}{2}\right)
\right)
}{
\left(
\frac{e^{-r^2/2}\, |r|^{n-1}}{(\sqrt{2})^{\rboxx{0.8pt}{n-1}}} + \Gamma\!\left(\frac{n+1}{2}\right)
\right)^2
} - 1 \right),
\label{series_LIG_r=inf}
\edm

For large $r$, the non-$\Gamma()$ terms vanish, and the expression again reduces to that in Eq.~(\ref{var_ndim_r=0}).

For inner truncation, again using the relationship between $\Gamma(,)$ and $\lgamma(,)$, along with the above, we have
\bdm
\lim_{|r| \to \infty} \mathrm{Var}(R;M,r,n)_{|R \ge a} =
M^2 \left(
\frac{
\left(
\frac{2 e^{-r^2/2}\, |r|^{n-2}}{(\sqrt{2})^{\rboxx{0.8pt}{n}}}
\right)
\left(
\frac{e^{-r^2/2}\, |r|^{n}}{(\sqrt{2})^{\rboxx{0.8pt}{n}}}
\right)
}{
\left(
\frac{e^{-r^2/2}\, |r|^{n-1}}{(\sqrt{2})^{\rboxx{0.8pt}{n-1}}}
\right)^2
} - 1 \right),
\label{series_UIG_r=inf}
\edm

The large fractional term  is 1, so the limit is 0.
\subsection{Expressions related to outer-truncated UTSCDs}
\label{subsec:express_outer_trun}
The pdf for a scaled chi distribution, $g(R;\sigma, n) dR$, with support over the interval $R \in [0,a]$ 
is similar to the expression in Eq.~(\ref{gaussian_pdf_of_R}):
\beq
g(R;\sigma,n)_{|R \le a} dR =
\frac{2 e^{-R^2/2\sigma^2} R^{n-1}}{(\sqrt{2} \sigma)^{\rbox{2pt}{n}}\, {\lgamma}\!\left(\frac{n}{2},\frac{r^2}{2}\right)}  dR.
\label{gaussian_pdf_of_R_inv}
\eeq

The $k$th raw moment, similar to Eq.~(\ref{kth_raw_moment}), is:
\beq
M_k(\sigma,r,n)_{|R\le a} = \frac{\left(\sqrt{2} \sigma\right)^k \lgamma\!\left(\frac{n+k}{2},\frac{r^2}{2}\right)}
{\lgamma\!\left(\frac{n}{2},\frac{r^2}{2}\right)},
\label{kth_raw_moment_outer}
\eeq
provided $n>-k$.

Form I of the variance, similar to Eq.  (\ref{inner_unsub_var}),  is
\beq
\mathrm{Var}(R;\sigma,r,n)_{|R \le a}=\frac{2 \sigma^2\left(\lgamma\!\left(\frac{n}{2},\frac{r^2}{2}\right) 
\lgamma\!\left(\frac{n+2}{2},\frac{r^2}{2}\right)
-\lgamma\!\left(\frac{n+1}{2},\frac{r^2}{2}\right)^2\right)}
{\lgamma\!\left(\frac{n}{2},\frac{r^2}{2}\right)^2}
\label{outer_unsub_var}
\eeq
Letting $k = 1$ in Eq.~(\ref{kth_raw_moment_outer}),
\beq
\sigma(M,r,n)_{|R\le a} = \frac{M}{\sqrt{2}} \cdot
\frac{\lgamma\!\left(\frac{n}{2},\frac{r^2}{2}\right)}
{\lgamma\!\left(\frac{n+1}{2},\frac{r^2}{2}\right)}
\label{sigma_multi_inv}
\eeq
Form II of the variance, then, is similar to the expression in Eq.~(\ref{multidim_var}):
\beq
\mathrm{Var}(R;M,r,n)_{|R\le a} =
M^2 \left(
\frac{\lgamma\!\left(\frac{n}{2},\frac{r^2}{2}\right)
\lgamma\!\left(\frac{n+2}{2},\frac{r^2}{2}\right)}
{\lgamma\!\left(\frac{n+1}{2},\frac{r^2}{2}\right)^2} -1
\right)
\label{multidim_var_inv}
\eeq
where, again, the large, fractional term is the second raw moment divided by $M^2$.
\subsection{Doubly truncated scaled chi distributions}
\label{doubly_truncated_chi}
A scaled chi distribution with $n$ distributional degrees of freedom over the interval
$R \in [a,b], b > a, \{b,a\} \in \mathbb{R+}$ has the form
\beq
g(R;\sigma,r_{_1},r_{_2},n)_{|a\le R\le b} dR =
\frac{2 e^{-R^2/2\sigma^2} R^{n-1}}
{(\sqrt{2}\sigma)^{\rbox{2pt}{n}}~ \left(\Gamma\!\left(\frac{n}{2},\frac{{r_1}^2}{2}\right)-
\Gamma\!\left(\frac{n}{2},\frac{{r_2}^2}{2}\right)\right)} dR ,
\label{multidim_gaussian_pdf_gen}
\eeq
where $r_1 = -a/\sigma$ and $r_2 = -b/\sigma$.

Define the function $\Gamma\!(x,y_{_1},y_{_2}) = \Gamma(x,y_{_1}) - \Gamma(x,y_{_2})$.
This is sometimes referred to as the generalized Gamma function.
Then the $k$th raw moment over $R \in [a,b]$ is
\beq
M_k(\sigma,k,a,b,n)_{|a\le R\le b} = \frac{\left(\sqrt{2} \sigma\right)^k
\Gamma\!\left(\frac{n+k}{2},\frac{a^2}{2 \sigma^2}, \frac{b^2}{2 \sigma^2}\right)}
{\Gamma\!\left(\frac{n}{2},\frac{a^2}{2 \sigma^2},\frac{b^2}{2 \sigma^2}\right)}
\label{kth_raw_moment_gen}
\eeq
and the unsubstituted form or Form I of the variance is
\beq
\mathrm{Var}(R;\sigma,a,b,n)_{|a\le R\le b} =
\frac{2 \sigma^2
\Gamma\!\left(\frac{n+2}{2},\frac{a^2}{2 \sigma^2},\frac{b^2}{2 \sigma^2}\right)
\left(\Gamma\!\left(\frac{n}{2},\frac{a^2}{2 \sigma^2},\frac{b^2}{2 \sigma^2}\right) -
\Gamma\!\left(\frac{n+1}{2},\frac{a^2}{2 \sigma^2},\frac{b^2}{2 \sigma^2}\right)\right)}
{\Gamma\!\left(\frac{n}{2},\frac{a^2}{2 \sigma^2},\frac{b^2}{2 \sigma^2}\right)^2}.
\label{gen_unsub_var}
\eeq
Form II of the variance is
\beq
\mathrm{Var}(R;M,a,b,n)_{|a \le R \le b} =
M^2 \left(\frac{\Gamma\!\left(\frac{n}{2},\frac{a^2}{2\sigma^2},\frac{b^2}{2\sigma^2}\right)
\Gamma\!\left(\frac{n+2}{2},\frac{a^2}{2\sigma^2},\frac{b^2}{2\sigma^2}\right)}
{\Gamma\!\left(\frac{n+1}{2},\frac{a^2}{2\sigma^2},\frac{b^2}{2\sigma^2}\right)^2}-1\right)
\label{multidim_var_gen}
\eeq
where the fractional term is the second raw moment divided by $M^2$.

\newpage
\begin{center}
{\large\bf {Acknowledgments}}
\end{center}
The author thanks Victor Ovchinnikov for comments on the manuscript.
He thanks the late Prof. Martin Karplus, 
Prof. Dan Kahne, the Department of Chemistry and Chemical Biology at
Harvard University, Harvard Medical School, 
and Paul Conlin and the Boston VA Medical Center for support.

\bibliographystyle{jasa3}
\bibliography{rjp.gauss.var.11.14.25}

\begin{thebibliography}{21}
\newcommand{\enquote}[1]{``#1''}
\expandafter\ifx\csname natexlab\endcsname\relax\def\natexlab#1{#1}\fi
\expandafter\ifx\csname url\endcsname\relax
  \def\url#1{{\tt #1}}\fi
\expandafter\ifx\csname urlprefix\endcsname\relax\def\urlprefix{URL }\fi

\bibitem[\protect\citeauthoryear{Alexander S.~Mikhailov}{Alexander
  S.~Mikhailov}{1996}]{Mikhailov1996}
Alexander S.~Mikhailov, A. Y.~L. (1996), {\em Fractals\/}, Springer Berlin
  Heidelberg, 55--68.

\bibitem[\protect\citeauthoryear{Bera and Sharma}{Bera and
  Sharma}{1999}]{Bera1999}
Bera, A. and Sharma, S. (1999), \enquote{Estimating {P}roduction {U}ncertainty
  in {S}tochastic {F}rontier {P}roduction {F}unction {M}odels,} {\em Journal of
  Productivity Analysis\/}, 12,
  \urlprefix\url{https://doi.org/10.1023/A:1007828521773}.

\bibitem[\protect\citeauthoryear{Bidaoui, El~Abbassi, El~Bouardi, and
  Darcherif}{Bidaoui et~al.}{2019}]{bidaoui2019wind}
Bidaoui, H., El~Abbassi, I., El~Bouardi, A., and Darcherif, A. (2019),
  \enquote{Wind speed data analysis using {W}eibull and {R}ayleigh distribution
  functions, case study: five cities northern {M}orocco,} {\em Procedia
  Manufacturing\/}, 32, 786--793,
  \urlprefix\url{https://doi.org/10.1016/j.promfg.2019.02.286}.

\bibitem[\protect\citeauthoryear{Blass and Gurevich}{Blass and
  Gurevich}{2018}]{blass2018negative}
Blass, A. and Gurevich, Y. (2018), \enquote{Negative probabilities, {II}:
  {W}hat they are and what they are for,}
  \urlprefix\url{https://doi.org/10.48550/arXiv.1807.10382}.

\bibitem[\protect\citeauthoryear{{C}ensus {B}ureau}{{C}ensus
  {B}ureau}{2023}]{USCensusIncome2022}
{C}ensus {B}ureau, U. (2023), \enquote{American community survey: 2022 1-year
  {PUMS},}
  \urlprefix\url{https://www2.census.gov/programs-surveys/acs/data/pums/2022/1-Year/csv_pus.zip}.

\bibitem[\protect\citeauthoryear{Chhabra and Sreenivasan}{Chhabra and
  Sreenivasan}{1991}]{Chhabra1991}
Chhabra, A.~B. and Sreenivasan, K.~R. (1991), \enquote{Negative dimensions:
  {T}heory, computation, and experiment,} {\em Phys. Rev. A\/}, 43, 1114--1117,
  \urlprefix\url{https://link.aps.org/doi/10.1103/PhysRevA.43.1114}.

\bibitem[\protect\citeauthoryear{Dyer}{Dyer}{1973}]{Dyer1973chidist}
Dyer, D.~D. (1973), \enquote{Estimation of the {S}cale {P}arameter of the {C}hi
  {D}istribution {B}ased on {S}ample {Q}uantiles,} {\em Technometrics\/}, 15,
  489--496, \urlprefix\url{http://www.jstor.org/stable/1266854}.

\bibitem[\protect\citeauthoryear{Hazelton}{Hazelton}{2011}]{Hazelton2011}
Hazelton, M.~L. (2011), {\em Methods of Moments Estimation\/}, Berlin,
  Heidelberg: Springer Berlin Heidelberg, 816--817,
  \urlprefix\url{https://doi.org/10.1007/978-3-642-04898-2_364}.

\bibitem[\protect\citeauthoryear{Horrace}{Horrace}{2015}]{Horrace2015}
Horrace, W. (2015), \enquote{Moments of the truncated normal distribution,}
  {\em J Prod Anal\/}, 43, 133--138,
  \urlprefix\url{https://www.jstor.org/stable/43550058}. [Online; accessed
  1-December-2023].

\bibitem[\protect\citeauthoryear{Mandelbrot}{Mandelbrot}{1990}]{MANDELBROT1990}
Mandelbrot, B.~B. (1990), \enquote{Negative fractal dimensions and
  multifractals,} {\em Physica A: Statistical Mechanics and its
  Applications\/}, 163, 306--315,
  \urlprefix\url{https://www.sciencedirect.com/science/article/pii/037843719090339T}.

\bibitem[\protect\citeauthoryear{Peckham and McNaught}{Peckham and
  McNaught}{1992}]{Peckham1992}
Peckham, G.~D. and McNaught, I.~J. (1992), \enquote{Applications of
  {M}axwell-{B}oltzmann distribution diagrams,} {\em Journal of Chemical
  Education\/}, 69, 554, \urlprefix\url{https://doi.org/10.1021/ed069p554}.

\bibitem[\protect\citeauthoryear{Petrella}{Petrella}{2025}]{Petrella2025}
Petrella, R. (2025), \enquote{Antibodies and crytographic hash functions:
  quantifying the specificity paradox,} {\em Frontiers in Immunology\/}, 16,
  1585421.

\bibitem[\protect\citeauthoryear{Sorin-Nicolae}{Sorin-Nicolae}{2022}]{Sorin-Nicolae2022}
Sorin-Nicolae, C. (2022), \enquote{The {U}se {O}f {N}egative {P}robabilities
  {I}n {E}conomics,} {\em Revista Economica\/}, 74, 65--74,
  \urlprefix\url{https://ideas.repec.org/a/blg/reveco/v74y2022i1p65-74.html}.

\bibitem[\protect\citeauthoryear{Suman}{Suman}{2020}]{suman2020}
Suman, S. (2020), \enquote{Maxwellian {D}istribution in {F}our {D}imensional
  {C}lassical {LJ} 6-12 {G}as: {A} {S}tudy using {M}olecular {D}ynamics
  {S}imulation,} {\em Int. Res. J. Eng. Tech\/}, 7, 3097--3099,
  \urlprefix\url{https://www.irjet.net/archives/V7/i8/IRJET-V7I8524.pdf}.

\bibitem[\protect\citeauthoryear{Wigner}{Wigner}{1932}]{Wigner1932}
Wigner, E. (1932), \enquote{On the {Q}uantum {C}orrection {F}or {T}hermodynamic
  {E}quilibrium,} {\em Phys. Rev.\/}, 40, 749--759,
  \urlprefix\url{https://link.aps.org/doi/10.1103/PhysRev.40.749}.

\bibitem[\protect\citeauthoryear{Wikipedia}{Wikipedia}{2023}]{FermatTheoremWiki}
Wikipedia (2023), \enquote{Fermat's theorem (stationary points)
  ---{W}ikipedia{,} {T}he {F}ree {E}ncyclopedia,}
  \urlprefix\url{https://en.wikipedia.org/wiki/Fermat%27s_theorem_(stationary_points)}.
  [Online; accessed 1-December-2023].

\bibitem[\protect\citeauthoryear{Wikipedia}{Wikipedia}{2024{\natexlab{a}}}]{chi-distributionWiki}
--- (2024{\natexlab{a}}), \enquote{Chi distribution--{W}ikipedia{,} {T}he
  {F}ree {E}ncyclopedia,}
  \urlprefix\url{https:/https://en.wikipedia.org/wiki/Chi_distribution}. Last
  updated May 3, 2024. Online; accessed May, 2024.

\bibitem[\protect\citeauthoryear{Wikipedia}{Wikipedia}{2024{\natexlab{b}}}]{HalfNormalWiki}
--- (2024{\natexlab{b}}), \enquote{Half-normal distribution--{W}ikipedia{,}
  {T}he {F}ree {E}ncyclopedia,}
  \urlprefix\url{https://en.wikipedia.org/wiki/Half-normal_distribution}. Last
  updated April 4, 2024. Online; accessed June, 2024.

\bibitem[\protect\citeauthoryear{Wikipedia}{Wikipedia}{2024{\natexlab{c}}}]{MethodofMomentsWiki}
--- (2024{\natexlab{c}}), \enquote{Method of moments
  (statistics)--{W}ikipedia{,} {T}he {F}ree {E}ncyclopedia,}
  \urlprefix\url{https://en.wikipedia.org/wiki/Method_of_moments_(statistics)}.
  Last updated March 19, 2024. Online; accessed May, 2024.

\bibitem[\protect\citeauthoryear{Wikipedia}{Wikipedia}{2024{\natexlab{d}}}]{multivariateWiki}
--- (2024{\natexlab{d}}), \enquote{Multivariate normal
  distribution--{W}ikipedia{,} {T}he {F}ree {E}ncyclopedia,}
  \urlprefix\url{https://en.wikipedia.org/wiki/Multivariate_normal_distribution}.
  Last updated May 11, 2024. Online; accessed May, 2024.

\bibitem[\protect\citeauthoryear{Łukaszyk}{Łukaszyk}{2022}]{Lukaszyk2022}
Łukaszyk, S. (2022), \enquote{Novel {R}ecurrence {R}elations for {V}olumes and
  {S}urfaces of n-{B}alls, {R}egular n-{S}implices, and n-{O}rthoplices in
  {R}eal {D}imensions,} {\em Mathematics\/}, 10,
  \urlprefix\url{https://www.mdpi.com/2227-7390/10/13/2212}.

\end{thebibliography}

\setcounter{figure}{0}
\setcounter{table}{0}
\renewcommand{\thefigure}{S\arabic{figure}}
\renewcommand{\thetable}{S\arabic{table}}
\clearpage

\section{Supplementary Material}
\subsection{Numerical values for the variance of a truncated Gaussian distribution and its derivative}
\label{subsec:val_partial_der_var}
\hfill\\
Table S1 lists $\partial (\mathrm{Var}(r))/\partial r$ and $\mathrm{Var}(r)$ values 
for a Gaussian distribution
over the interval $x \in [0,\infty)$ and at specific values of $r = \mu/\sigma$, assuming a mean of 1.

\begin{table}[h]
\tabcolsep=0pt
\begin{tabular*}{1.0\columnwidth}{@{\extracolsep{\fill}}lrrr@{}}
\hline
\hline
\multicolumn{2}{c}{$r$} & \multicolumn{1}{c}{$\partial (\mathrm{Var}(r))/\partial r$} &
\multicolumn{1}{c}{$\mathrm{Var}(r)$}\\
\hline
 -262144 & $-2^{18}$ & $-2.22044605 \times 10^{-16}$ & $1- 2.91038304 \times 10^{-11}$\\
 -131072 & $-2^{17}$ & $-1.77635684  \times 10^{-15}$ & $1- 1.16415322 \times 10^{-10}$\\
 -65536 & $-2^{16}$ & $-1.42108547 \times 10^{-14}$ & $1- 4.65661286 \times 10^{-10}$\\
 -32768 & $-2^{15}$ & $-1.13686836 \times 10^{-13}$ & $1 - 1.86264513 \times 10^{-9}$\\
 -16384 & $-2^{14}$ & $-9.09494641 \times 10^{-13}$ & $1 - 7.45058035  \times 10^{-9}$\\
 -8192 & $-2^{13}$ & $-7.27595566 \times 10^{-12}$ & $1 - 2.98023184 \times 10^{-8}$\\
 -4096 & $-2^{12}$ & $-5.82075985 \times 10^{-11}$ & $1 - 1.19209226 \times 10^{-7}$ \\
 -2048 & $-2^{11}$ & $-4.65659289 \times 10^{-10}$ & $1 - 4.76836135 \times 10^{-7}$ \\
 -1024 & $-2^{10}$ & $-3.72522635 \times 10^{-9}$ &  $1 - 1.90733226 \times 10^{-6}$ \\
 -512 & $-2^9$ & $-2.98002762 \times 10^{-8}$ &        $1 - 7.62913261 \times 10^{-6}$ \\
 -256 & $-2^8$ & $-2.38353113 \times 10^{-7}$ &        $1 - 0.000030513388$ \\
 -128 & $-2^7$ & $-1.90525539 \times 10^{-6}$ &        $1 - 0.00012200330$ \\
 -64 & $-2^6$ & -0.000015192019 & 0.99951279 \\
 -32 & $-2^5$ & -0.00011996057 &0.99806385 \\
 -16 & $-2^4$ & -0.00091227837 &0.99245028 \\
 -8 & $-2^3$ & -0.0060833737 & 0.97248274 \\
 -4 & $-2^2$ & -0.028580298 & 0.91697715 \\
 -2 & $-2^1$ & -0.076308287 & 0.82044107 \\
 -1 & $-2^0$ & -0.12343379 & 0.72197770  \\
 -0.5 & $-2^{-1}$ & -0.15152240 & 0.65326766 \\
 -0.25 & $-2^{-2}$ & -0.16516707 & 0.61366302 \\
 -0.125 & $-2^{-3}$ & -0.17154333 & 0.59261450 \\
 -0.0625 & $-2^{-4}$ & -0.17456721 & 0.58179790 \\
 -0.03125 & $-2^{-5}$ & -0.17603105 & 0.57631972 \\
 -0.015625 & $-2^{-6}$ & -0.17675005 & 0.57356361 \\
 -0.0078125 & $-2^{-7}$ & -0.17710620 & 0.57218135 \\
 -0.00390625 & $-2^{-8}$ & -0.17728342 & 0.57148919 \\
 -0.001953125 & $-2^{-9}$ & -0.17737182 & 0.57114284\\
 -0.0009765625 & $-2^{-10}$ & -0.17741597 & 0.57096961 \\
 -0.00048828125 & $-2^{-11}$ & -0.17743802 & 0.57088297 \\
 -0.00024414062 & $-2^{-12}$ & -0.17744905 & 0.57083965 \\
 0.00024414062 & $2^{-12}$ & -0.17747110 & 0.57075300 \\
 0.00048828125 & $2^{-11}$ & -0.17748212 & 0.57070967 \\
 0.0009765625 & $2^{-10}$ & -0.17750415 & 0.57062300 \\
\hline
\end{tabular*}
\label{tab:values_partial_Var}
\end{table}

\clearpage
\newpage
\begin{table}[h]
\tabcolsep=0pt
\begin{tabular*}{1.0\columnwidth}{@{\extracolsep{\fill}}lrrr@{}}
\hline
\multicolumn{2}{c}{$r$} & \multicolumn{1}{c}{$\partial (\mathrm{Var}(r))/\partial r$} &
\multicolumn{1}{c}{$\mathrm{Var}(r)$}\\
\hline
 0.001953125 & $2^{-9}$ & -0.17754818 & 0.57044964 \\
 0.00390625 & $2^{-8}$ & -0.17763615 & 0.57010278 \\
 0.0078125 & $2^{-7}$ & -0.17781164 & 0.56940855 \\
 0.015625 & $2^{-6}$ & -0.17816088 & 0.56801803 \\
 0.03125 & $2^{-5}$ & -0.17885219 & 0.56522885 \\
 0.0625 & $2^{-4}$  & -0.18020529 & 0.55961847 \\
 0.125 & $2^{-3}$  & -0.18278603 & 0.54827409 \\
 0.25 & $2^{-2}$  & -0.18738644 & 0.52513005 \\
 0.5 & $2^{-1}$  & -0.19388347 & 0.47738919 \\
 1 & $2^0$ & -0.19311248 & 0.37980643 \\
 2 & $2^1$ & -0.13737522 & 0.20985861 \\
 4 & $2^2$ & -0.031087941 & 0.062462360 \\
 8 & $2^3$ & -0.0039062500 & 0.015625000 \\
 16 & $2^4$ & -0.00048828125 & 0.0039062500 \\
 32 & $2^5$ & -0.000061035156 & 0.00097656250 \\
 64 & $2^6$ & $-7.62939453 \times 10^{-6}$ & 0.00024414062 \\
 128 & $2^7$ & $-9.53674316 \times 10^{-7}$ & 0.000061035156 \\
 256 & $2^8$ & $-1.19209290  \times 10^{-7}$ & 0.000015258789 \\
 512 & $2^9$ & $-1.49011612  \times 10^{-8}$ & $3.81469727 \times 10^{-6}$ \\
 1024 & $2^{10}$ & $-1.86264515  \times 10^{-9}$ & $9.53674316 \times 10^{-7}$ \\
 2048 & $2^{11}$ & $-2.32830644  \times 10^{-10}$ & $2.38418579 \times 10^{-7}$ \\
4096 & $2^{12}$ & $-2.91038305 \times 10^{-11}$ & $5.96046448 \times 10^{-8}$ \\
8192 & $2^{13}$ & $-3.63797881 \times 10^{-12}$ & $1.49011612 \times 10^{-8}$ \\
16384 & $2^{14}$ & $-4.54747351 \times 10^{-13}$ & $3.72529030 \times 10^{-9}$ \\
32768 & $2^{15}$ & $-5.68434189 \times 10^{-14}$ & $9.31322575 \times 10^{-10}$ \\
65536 & $2^{16}$ & $-7.10542736 \times 10^{-15}$ & $2.32830644  \times 10^{-10}$ \\
131072 & $2^{17}$ & $-8.88178420 \times 10^{-16}$ & $5.82076609 \times 10^{-11}$ \\
\hline
\end{tabular*}
\caption{Values of $\partial (\mathrm{Var}(x;M,r)_{|x \ge 0}/\partial r$ and
$\mathrm{Var}(x;M,r)_{|x \ge 0}$ for a Gaussian \Tstrut  distribution 
 $f(x)_{|x > 0}$ at different values of $r$, evaluated at $M=1$.}
\label{tab:end_values_part_var}
\end{table}

\subsection{Derivative of the variance of a UTGD at large $r$}
\label{subsec:derive_var_large_r}
The form of the of derivative of $ \mathrm{Var}(x;M,r)_{|x \ge a}$
is given by Eq. (\ref{eq_dvar}) in Lemma~\ref{lemma:var_different}.

A first-order power series expansion around $r \rightarrow +\infty$ gives
\bdm
\frac{1}{(M-a)^2}\frac{\partial (\mathrm{Var}(r))}{\partial r} \approx 
\frac{e^{-r^2/2}(3/r^2 + 1)}{\sqrt{2\pi}} -\frac{2}{r^3},
\edm
where the
shorthand $\mathrm{Var}(r) = \mathrm{Var}(x;M,r)_{|x \ge a}$ is being used.

For large, positive $r$, the first term vanishes and thus
\bdm
\frac{1}{(M-a)^2}\frac{\partial (\mathrm{Var}(r))}{\partial r} \sim -\frac{2}{r^3},
\edm

For large, negative $r$, the 1st-order and all even-order terms are zero, so
a higher-order expansion must be done. A ninth-order expansion gives:
\bdm
\frac{1}{(M-a)^2}\frac{\partial (\mathrm{Var}(r))}{\partial r} \approx \frac{4}{r^3} -\frac{72}{r^5}
 + \frac{1260}{r^7} -\frac{23184}{r^9}.
\edm
As $r \rightarrow -\infty$, the first term dominates, so that
\bdm
\frac{1}{(M-a)^2}\frac{\partial (\mathrm{Var}(r))}{\partial r} \approx  \frac{4}{r^3}.
\edm
\noindent\rule{5cm}{0.4pt} 
\subsection{Fifth and sixth central moments of a truncated Gaussian}
\label{subsec:5th_6th_CM}
\hfill\\
\indent For a Gaussian distribution $f(x)$, the (unnormalized) 5th central moment over the interval 
$x \in [a,\infty)$ can be written as
\bdm
\overline{M}{_5} = \left(\frac{M-a}{rB+2}\right)^5 
2\left( (r^4 + 4r^2 -7)B^4 + 10 r (r^2+3) B^3 + 40(r^2+1) B^2 + 80 r B +64\right),
\edm
and the 6th is
\bdm
\overline{M}{_6} = \left(\frac{M-a}{rB+2}\right)^6 
\left(15 B^6 -2r (r^4+5r^5 + 15) B^5 - 12 (2r^4 + 8 r^2 + 1) B^4 \right. 
\edm
\bdm
\left.-120 r(r^2+3) B^3 -80 (4r^2 + 5) B^2 -480 r B -320 \right) ,
\edm
where $B = \sqrt{2\pi} e^{r^2/2} \xi(r) = \sqrt{2\pi} e^{r^2/2} (\erf\!\left(r/\sqrt{2}\right)+ 1). $
\subsection{Expressions for calibration of truncated log-normal distributions}
\label{subsec:calib_log-normal}
If a truncated Gaussian with parameters $\mu$ and $\sigma$ describes a natural-log-transformed dataset $X = \ln(Y)$,
over support $x > a$, 
then the natural log of the Form I variance of the original dataset, $Y$, over support $y > e^a$, 
expressed in terms of those same $\mu$ and $\sigma$ parameters is
\bdm
\ln\left(\mathrm{Var}(y;r,\sigma,\mu)_{|y>e^a}\right) = 
2\sigma^2 + 2\mu + \ln(\xi(r+2\sigma)) + 
\ln\left(1 - \frac{e^{-\sigma^2} \xi(r+\sigma)^2}{\xi(r+2\sigma)\cdot \xi(r) }\right) - \ln(\xi(r)),
\edm
and the log of the Form II variance is
\bdm
\ln\left(\mathrm{Var}(y;r,M,\sigma,\mu)_{|y>e^a}\right) =
  2 \ln(M) + \ln\left(\frac{e^{-2\mu}M^2 \xi(r+2\sigma)\cdot \xi(r)^3}
{\xi(r+\sigma)^4}-1\right), 
\edm
where $M$ is the mean, $\xi(r) = \erf(r/\sqrt{2}) + 1$, and $r = (\mu-a)/\sigma$.

In addition, the slope of $\sigma(\mu)$ under constant variance, for Form I of the variance of the
original truncated dataset, is:  
\bdm
\frac{d\sigma_1}{du} = -\frac{\frac{d\left(\ln(\mathrm{Var_{_I}})\right)}{du}}{\frac{d\left(\ln(\mathrm{Var_{_I}})\right)}{d\sigma}} =
-\frac{\frac{d\mathrm{Var_{_I}}/du}{\mathrm{Var_{_I}}}}  {\frac{d\mathrm{Var_{_I}}/d\sigma}{\mathrm{Var_{_I}}}} =
\frac{ \frac{d\mathrm{Var_{_I}}}{du}} {\frac{d\mathrm{Var_{_I}}}{d\sigma}} = 
\edm

\[
\frac{
  \begin{array}{c}
      e^{\sigma^2} 
      \left(
        2 e^{-\sigma(r+3\sigma/2)} \xi(r+\sigma) 
        - \xi(r+2\sigma) 
        - e^{-2 \sigma(r + \sigma)} \xi(r)
      \right)
      + \\
      + \left(
         \frac{\xi(r+\sigma)^2}{\xi(r)} 
        - e^{\sigma^2} \xi(r+2\sigma) 
      \right)
      \left(
        \sqrt{2\pi} \sigma e^{r^2/2} \xi(r)-2
      \right)
  \end{array}
}{
\begin{array}{c}
    2 (r - \sigma) e^{-\sigma(r+ \sigma/2)}\xi(r+\sigma) 
    + r \left(
        e^{\sigma^2}\xi(r+2\sigma) 
        - \frac{2 \xi(r+\sigma)^2}{\xi(r)}
    \right) + \\
    + (2\sigma - r)  e^{-\sigma(2r+\sigma)}  \xi(r)
    + \sqrt{2\pi} \sigma^2 e^{r^2/2}  
    \left(
        2 e^{\sigma^2} \xi(r+2\sigma) \xi(r)-\xi(r+\sigma)^2
    \right)
\end{array}
}~~,
\]

and, similarly, for Form II, the slope of $\sigma(\mu)$ is: 
\bdm
\frac{d\sigma_2}{du} = 
-\frac{ \frac{d\mathrm{Var_{_{II}}}}{du}} {\frac{d\mathrm{Var_{_{II}}}}{d\sigma}} = 
\edm

\bdm
\frac{
\begin{array}{c}
3J e^{2\sigma(r+\sigma)} + \xi(r+\sigma) - 4   e^{\sigma(r+3\sigma/2)} \xi(r+2\sigma)
-\sqrt{2\pi}\sigma  e^{(r+2\sigma)^2/2} \xi(r+\sigma)\cdot \xi(r+2\sigma)
\end{array}
}
{
\begin{array}{c}
3Jr e^{2\sigma(r+\sigma)} + (r-2\sigma)\xi(r+\sigma) - 4(r-\sigma) e^{\sigma(r+3\sigma/2)}\xi(r+2\sigma)
\end{array}
}, 
\edm
where $J = \xi(r + \sigma)\cdot \xi(r+2\sigma)/\xi(r)$.

\subsection{Laurent series expansions}
\label{subsec:Laurent_ser_expan}
\hfill\\
\noindent Let $\xi(r)= \erf\!\left(r/\sqrt{2}\right)+ 1$ and
\bdm
S_1 = \frac{1}{r^2} - \frac{2}{r^4} + \frac{7}{r^6}
- \frac{36}{r^8} + \frac{249}{r^{10}} - \frac{2190}{r^{12}} + \frac{23535}{r^{14}} - \frac{299880}{r^{16}} +\mathcal{O}\left((1/r)^{18}\right),
\edm
\bdm
S_2 = \frac{1}{r^2}
- \frac{3}{r^4} + \frac{15}{r^6} -\frac{105}{r^8} + \frac{945}{r^{10}} - \frac{10395}{r^{12}} + \frac{135135}{r^{14}}
- \frac{2027025}{r^{16}} + \mathcal{O}\left((1/r)^{18}\right),
\edm
\bdm
S_3 = \frac{1}{r^4}
- \frac{6}{r^6} + \frac{39}{r^8} -\frac{300}{r^{10}} + \frac{2745}{r^{12}} - \frac{29610}{r^{14}} + \frac{372015}{r^{16}}
+  \mathcal{O}\left((1/r)^{18}\right).
\edm

{\setlength{\abovedisplayskip}{4pt}
 \setlength{\belowdisplayskip}{4pt}
\noindent Then, as $r \rightarrow -\infty$,
\begin{flalign*}
\hspace{1cm} \pi e^{r^2}\xi(r)^2  =  2 S_1, &&
\end{flalign*}
\begin{flalign*}
\hspace{1cm} \sqrt{2\pi} r e^{r^2/2} \xi(r)  =  2(S_2-1), ~\mbox{and} &&
\end{flalign*}
\begin{flalign*}
\hspace{1cm} \pi\left(\sqrt{2/\pi} + r e^{r^2/2} \xi(r) \right)^2  =  2 S_3. &&
\end{flalign*}
\noindent And as $r \rightarrow +\infty,$
\begin{flalign*}
\hspace{1cm} \pi e^{r^2} \xi(r)^2  =  4\pi e^{r^2} - 2 S_1 + \frac{4\sqrt{2\pi} e^{r^2/2} (S_2-1)}{r},  &&
\end{flalign*}
\begin{flalign*}
\hspace{1cm} \sqrt{2\pi} r e^{r^2/2} \xi(r)  =  \sqrt{8\pi} r e^{r^2/2} + 2(S_2-1), &&
\end{flalign*}
\begin{flalign*}
\hspace{1cm} \pi\left(\sqrt{2/\pi} + r e^{r^2/2} \xi(r)
\right)^2  =  4\pi r^2 e^{r^2} + 4\sqrt{2\pi}r e^{r^2/2} S_2 + 2S_3, ~\mbox{and} &&
\end{flalign*}
\begin{flalign*}
\hspace{1cm} e^{r^2/2} \xi(r)  =   2 e^{r^2/2} +  \frac{\sqrt{2/\pi}(S_2-1)}{r}. &&
\end{flalign*}
}

\subsection{Plot of $d\sigma_1/d\mu$ for a UTGD}
\label{subsec:plot_ds1_du}
The plot of $d\sigma_1/d\mu$--i.e., the derivative of the function $\sigma(\mu)$ under constant
variance, for Form I of the variance of a UTGD, is shown in Figure \ref{fig:ds1_du_plot}.  
\begin{figure}[h]
\centering
\includegraphics[width=5in]{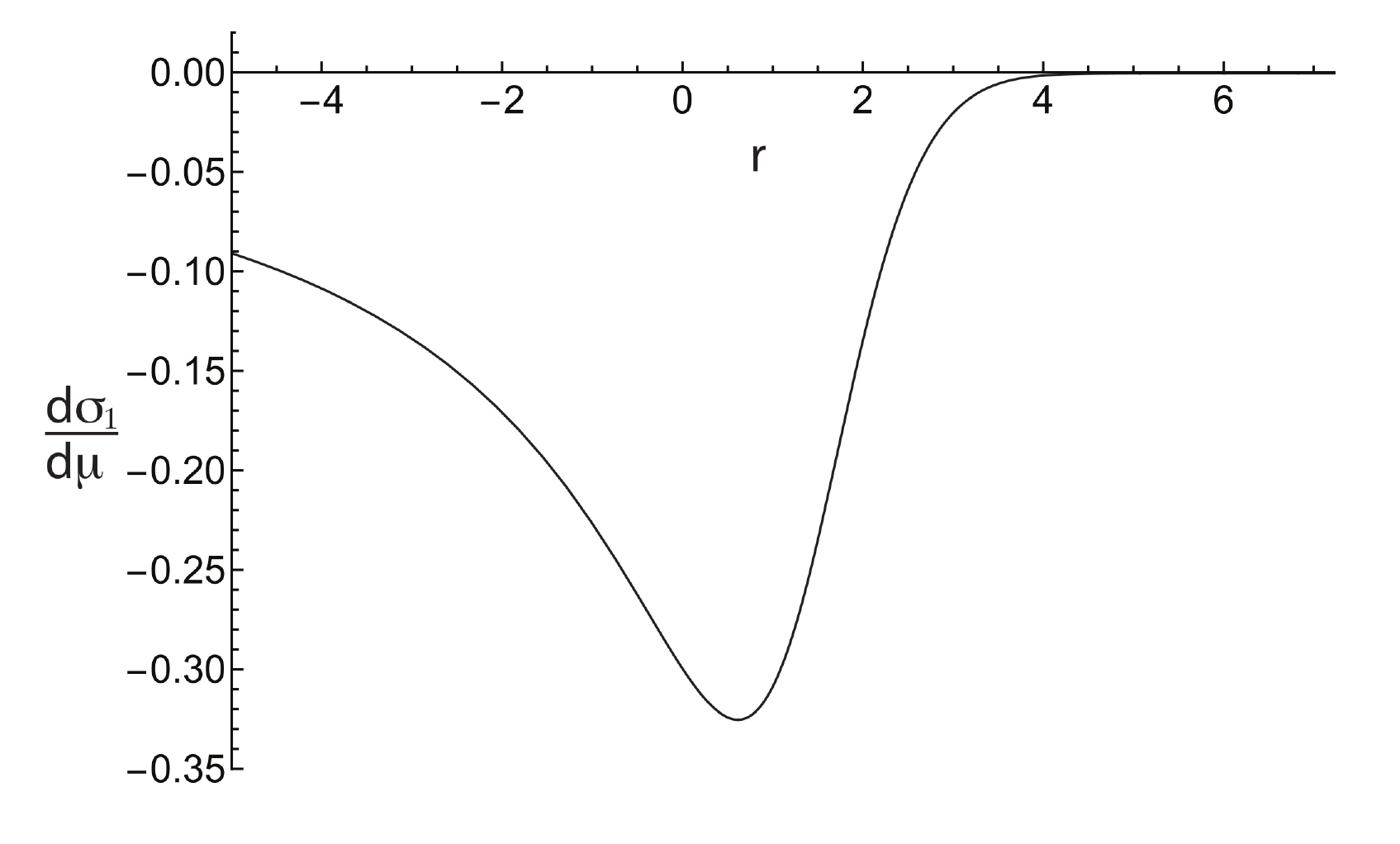}
\vspace{-0.25cm}
\caption{\small
Plot of $d\sigma_1/d\mu$, which is the derivative of the function $\sigma(\mu)$ 
under the condition of a constant variance
and using Form I (the unsubstituted form) of the variance (i.e., Eq. \ref{var_sig_r})
of a unilaterally truncated Gaussian distribution with support over $x \in [a,\infty)$. 
}
\label{fig:ds1_du_plot}
\end{figure}

\end{document}